\documentclass[11pt, a4paper]{amsart}

\usepackage[dvipsnames]{xcolor}
\usepackage[pagebackref,backref=true,colorlinks=true, linkcolor=Sepia, citecolor=Sepia]{hyperref}
\usepackage{amsmath,amsthm,amssymb,fancyhdr,graphicx,bbm,cancel,mathrsfs,todonotes,xypic,pinlabel}
\usepackage{tikz}
\usepackage{extarrows}
\usetikzlibrary{matrix}
\usepackage[all]{xy}
\usepackage{mathtools}
\usepackage{bm}
\usepackage{verbatim}
\usepackage{microtype}
\usepackage{subcaption}

\DeclareFontFamily{T1}{cbgreek}{}
\DeclareFontShape{T1}{cbgreek}{m}{n}{<-6>  grmn0500 <6-7> grmn0600 <7-8> grmn0700 <8-9> grmn0800 <9-10> grmn0900 <10-12> grmn1000 <12-17> grmn1200 <17-> grmn1728}{}
\DeclareSymbolFont{quadratics}{T1}{cbgreek}{m}{n}
\DeclareMathSymbol{\qoppa}{\mathord}{quadratics}{19}
\DeclareMathSymbol{\Qoppa}{\mathord}{quadratics}{21}

\theoremstyle{theorem}
\newtheorem{thm}{Theorem}[section]
\newtheorem{mainthm}{Theorem}

\newtheorem{lem}[thm]{Lemma}
\newtheorem{prop}[thm]{Proposition}
\newtheorem{cor}[thm]{Corollary}
\theoremstyle{definition} 
\newtheorem{deftn}[thm]{Definition}
 
\newtheorem{example}[thm]{Example}

\theoremstyle{remark} 
\newtheorem{rem}[thm]{Remark}

\usepackage{enumerate,pdfpages,stmaryrd} 
\usepackage[margin=1.2in, marginparwidth=1in]{geometry}

\usepackage{tikz}
\usepackage{dsfont}
\usetikzlibrary{matrix}

\renewcommand{\xto}{\xrightarrow}

\newcommand{\calS}{\mathcal{S}}

\newcommand{\calP}{\mathcal{P}}

\newcommand{\Z}{\mathbb{Z}}
\newcommand{\Q}{\mathbb{Q}}
\newcommand{\R}{\mathbb{R}}

\newcommand{\F}{\mathbb{F}}

\newcommand{\beq}{\begin{equation*}}
\newcommand{\eeq}{\end{equation*}}

\newcommand{\On}{\mathrm{O}}
\newcommand{\SOn}{\mathrm{SO}}

\DeclareMathOperator{\Imm}{Imm}
\DeclareMathOperator{\Emb}{Emb}

\DeclareMathOperator{\blEmb}{\widetilde{\Emb}}
\DeclareMathOperator{\Top}{Top}
\DeclareMathOperator{\blTop}{\widetilde{\Top}}
\DeclareMathOperator{\Diff}{Diff}
\DeclareMathOperator{\blDiff}{\widetilde{\Diff}}
\DeclareMathOperator{\Map}{Map}
\DeclareMathOperator{\Hom}{Hom}

\DeclareMathOperator{\hAut}{hAut}

\DeclareMathOperator*{\hocolim}{hocolim}
\DeclareMathOperator{\Fr}{Fr}

\DeclareMathOperator*{\colim}{colim}

\DeclareMathOperator{\Bun}{Bun}

\DeclareMathOperator{\fr}{fr}
\DeclareMathOperator{\sfr}{sfr}
\DeclareMathOperator{\id}{id}

\date{\today}
\begin{document}

\title{On automorphisms of high-dimensional solid tori}

\author{Mauricio Bustamante}
\email{mauricio.bustamante@mat.uc.cl}
\address{Departamento de Matem\'aticas, Universidad Cat\'olica de Chile}

\author{Oscar Randal-Williams}
\email{o.randal-williams@dpmms.cam.ac.uk}
\address{Centre for Mathematical Sciences, Wilberforce Road, Cambridge CB3 0WB, UK}

\begin{abstract}
We study the infinite generation in the homotopy groups of the group of diffeomorphisms of $S^1 \times D^{2n-1}$, for $2n \geq 6$, in a range of degrees up to $n-2$. Our analysis relies on understanding the homotopy fibre of a linearisation map from the plus-construction of the classifying space of a certain space of self-embeddings of stabilisations of this manifold to a form of Hermitian $K$-theory of the integral group ring of $\pi_1(S^1)$. We also show that these homotopy groups vanish rationally.
\end{abstract}

\maketitle
\setcounter{tocdepth}{1}
    \begingroup
    \hypersetup{linkcolor=black}
    \tableofcontents
    \endgroup
		
\section{Introduction and statement of results}
The mapping class group of a compact smooth manifold $M$ is the group of isotopy classes of diffeomorphisms $M\to M$ which fix a neighbourhood of the boundary pointwise. For simply connected closed manifolds of dimension $d>5$, Sullivan has shown that these groups are finitely generated \cite[Theorem 13.3]{sullivan}. In contrast, the mapping class group of a non-simply connected high-dimensional manifold can fail to be finitely generated, for example the torus $M=T^n$ for $n \geq 5$, cf.\ \cite[Theorem 4.1]{HatcherConc}. In fact, this phenomenon already arises for the solid torus \cite[Corollary 5.5]{HatcherWagoner}.

A $d$-dimensional solid torus is a manifold diffeomorphic to the product  $S^1\times D^{d-1}$ of a circle with a closed $(d-1)$-disc. The mapping class group $\pi_0(\Diff_{\partial}(S^1\times D^{d-1}))$ of this manifold is infinitely generated, at least when $d\geq 6$. This was known from early works of Hatcher--Wagoner \cite{HatcherWagoner}, and Hsiang--Sharpe \cite{HsiangSharpe}, who used surgery theory, a parametrised form of the $s$-cobordism theorem, and calculations of the algebraic $K$-groups of integral group rings, to prove that $\pi_0(\Diff_{\partial}(S^1\times D^{d-1}))$ contains a subgroup isomorphic to a countable infinite sum of $\Z/2$'s. In terms of Waldhausen's approach to pseudoisotopy theory \cite{waldhausen}, one can identify the source of this infinite generation as the homotopy fibre of the linearisation map 
\beq
\Omega^{\infty}\mathrm{A}(S^1)\to \Omega^\infty \mathrm{K}(\Z[t, t^{-1}])
\eeq
from Waldhausen's algebraic $K$-theory of spaces for $S^1$ to the algebraic $K$-theory of the integral group ring of $\pi_1(S^1) = \langle t \rangle$. Furthermore, this point of view leads to the discovery of infinitely generated torsion subgroups in the homotopy groups of $B\Diff_{\partial}(S^1\times D^{d-1})$, the classifying space of the topological group $\Diff_{\partial}(S^1\times D^{d-1})$ of self-diffeomorphisms of $S^1\times D^{d-1}$ which are the identity near the boundary, in degrees within the known pseudoisotopy stable range (currently known to be $* \lesssim \tfrac{2n}{3}$).

In this paper we propose a different method for computing the homotopy groups of $B\Diff_{\partial}(S^1\times D^{2n-1})$, which does not use the algebraic $K$-theory of spaces or the stable parametrised $h$-cobordism theorem and so is not subject to the dimension constraints imposed by pseudoisotopy theory. Our approach, which in addition gives a more geometric explanation of this infiniteness, focuses on manifolds of dimension $2n \geq 6$, and is based on certain ``Weiss fibre sequence''
\beq
B\Diff_{\partial}(S^1\times D^{2n-1})\to B\Diff_{\partial}(X_g)\to B\Emb^{\cong}_{\partial/2}(X_g)
\eeq
which lets us interpret the diffeomorphisms of a solid torus as the difference between diffeomorphisms and self-embeddings (fixing only a portion of the boundary) of the manifold 
$$X_g := S^1\times D^{2n-1}\# (S^n\times S^n)^{\# g}.$$
When $g\to\infty$ the homotopy type of the plus-construction of $B\Diff_{\partial}(X_g)$ is understood by work of Galatius--Randal-Williams \cite{GRWActa}, and it has finitely generated homotopy groups. Thus any infinite generation in the homotopy groups of $B\Diff_{\partial}(S^1\times D^{2n-1})$ is due to the space of self-embeddings of $X_g$, and even to the homotopy fibre of a linearisation map 
\beq
\hocolim_{g \to \infty} B\Emb_{\partial/2}^{\cong}(X_g)^+ \to \Omega^\infty_0 \mathrm{GW}(\Z[t, t^{-1}])
\eeq
to a form of Grothendieck--Witt theory (alias Hermitian $K$-theory) of the integral group ring of $\pi_1(S^1)$. By analogy with the above one might consider the domain of this map as a form of ``Hermitian $K$-theory of spaces for $S^1$", though we do not try to pursue this analogy.

This paper begins the systematic analysis of $B\Diff_{\partial}(S^1\times D^{2n-1})$ from this point of view. In this first instance we focus on the infinite-generation phenomenon and restrict to the range of degrees $* < n-2$, in which the embedding spaces $\Emb_{\partial/2}(X_g)$ may be replaced by their block analogues $\blEmb_{\partial/2}(X_g)$, by use of Morlet's lemma of disjunction, and therefore be analysed by surgery theory. 

\begin{mainthm}\label{thm:B}
For $n \geq 3$, all primes $p$, and $0 < k<\min(2p-3, n-2)$, the groups $\pi_k(B\Diff_{\partial}(S^1\times D^{2n-1}))$ are finitely-generated when localised at $p$. 

Furthermore if $2p-3 < n-2$ then there is a map
$$\bigoplus_{a> 0} \Z/p\{t^a - t^{-a}\} \to \pi_{2p-3}(B\Diff_{\partial}(S^1\times D^{2n-1}))$$
which is injective and whose cokernel is finitely-generated after localisation at $p$. When $p=2$ the cokernel is finitely-generated even before localisation.
\end{mainthm}

\begin{rem}
The elements in $\pi_{2p-3}(B\Diff_{\partial}(S^1\times D^{2n-1}))$ produced by our theorem have the following geometric interpretation. For $X_1 \simeq S^1 \vee S^n \vee S^n$ let us write $\pi_1(X_1) = \langle t \rangle$ and $x_1, x_2 \in \pi_n(X_1)$ for the inclusions of the two $n$-spheres. In the course of our proof, under the stated conditions we will show that for each $a>0$ there is a family
$$f_a : S^{2p-3} \to \Emb_{\partial/2}(X_1)$$
of self-embeddings of $X_1$ relative to $\tfrac{1}{2}\partial X_1 = S^1 \times D^{2n-2}_-\subset\partial X_g$, whose adjoint is in the homotopy class of
$$S^{2p-3} \times (S^1 \vee S^n \vee S^n) \to S^1 \vee S^{n+2p-3} \vee S^{n+2p-3} \xto{t \vee (- t^a x_1 \circ \alpha) \vee (t^{-a} x_2 \circ \alpha)} S^1 \vee S^n \vee S^n,$$ 
for $\alpha \in \pi_{n+2p-3}(S^n) \cong \pi_{2p-3}^s$ the standard generator of the $p$-torsion subgroup. The element $(-1)^n(t^a - t^{-a}) \in \pi_{2p-3}(B\Diff_{\partial}(S^1\times D^{2n-1}))$ in Theorem \ref{thm:B} then classifies the $S^1 \times D^{2n-1}$-bundle over $S^{2p-3}$ given by the family of complements of the embeddings $f_a$.
\end{rem}

\begin{rem}
In the pseudoisotopy stable range (currently $* \lesssim \tfrac{2n}{3}$) the analogue of Theorem \ref{thm:B} follows from the calculation of the $p$-local homotopy groups of the Whitehead spectrum of the circle, see \cite{GKM-Nil, hesselholt-WhS1}. In Appendix \ref{sec:classical} we explain how our calculation relates to this one.
\end{rem}

\begin{rem}
 In the case $p=2$, Farrell has given a construction of diffeomorphisms of $S^1 \times D^{2n-1}$ which are pseudoisotopic to the identity but not obviously isotopic to the identity. This construction is described in Hatcher's survey \cite[p.\ 9-10]{HatcherConc}, where it is claimed that the methods of \cite{HatcherWagoner} can be used to show that they are indeed nontrivial, but that ``a simpler proof of this would be quite welcome". It would be interesting to compare Farrell's geometric construction with ours.
\end{rem}

\begin{rem}
Our tool for showing that the map in Theorem \ref{thm:B} is injective, rather than merely having finitely-generated kernel, is the cyclotomic structure on $B\Diff_{\partial}(S^1\times D^{2n-1})$ corresponding to self-covering maps of $S^1\times D^{2n-1}$. We will show that the corresponding Frobenius operators $\{F_d\}_{d \in \mathbb{N}^\times}$ are given by
$$F_d(t^a - t^{-a}) = \begin{cases}
 t^{a/d} - t^{-a/d} & \text{if $d$ divides $a$},\\
0 & \text{if $d$ does not divide $a$}
\end{cases}$$
and so are (locally) nilpotent. On the other hand the various finitely-generated groups that arise in our analysis have the property that the operators $F_d \otimes \mathbb{Z}[d^{-1}]$ act invertibly.
\end{rem}

Although we have chosen to focus on the infinite-generation phenomenon, many of the intermediate results we obtain in this paper are of a more general nature and also serve as preparation for further study of $B\Diff_{\partial}(S^1\times D^{2n-1})$. As a first example, we have the following rational calculation in the same range of degrees. 

\begin{mainthm}\label{thm:A}
For all $n\geq 3$ and $0 < k<n-2$, we have
\beq
\pi_k(B\Diff_{\partial}(S^1\times D^{2n-1}))\otimes\Q = 0.
\eeq
\end{mainthm}

\begin{rem}
Budney and Gabai \cite{BudneyGabai}, and Watanabe \cite{watanabe2020thetagraph}, have recently shown that $\pi_{2n-3}(B\Diff_{\partial}(S^1\times D^{2n-1}))$ has infinite rank. In view of Theorem \ref{thm:A}, it would be interesting to know the rational homotopy groups of $B\Diff_{\partial}(S^1\times D^{2n-1})$ in the range of degrees $n-2\leq k\leq 2n-4$.
\end{rem}
In Section \ref{sec:Homeo} we will show that the statements of Theorems \ref{thm:B} and \ref{thm:A} continue to hold with diffeomorphisms replaced by homeomorphisms.
\subsection*{Structure of the paper}
In Section \ref{sect:mflds-quad} we introduce certain quadratic modules associated naturally to the manifolds $X_g$ and $W_{g,1}:=D^{2n}\#(S^n\times S^n)^{\# g}$. These quadratic modules and their groups of automorphisms are used in Section \ref{sec:hautXg} to describe the homotopy mapping class group of $X_g$, and also to compute some higher homotopy groups of the topological monoid of homotopy automorphisms of $X_g$. To carry out these calculations, we use several results in metastable homotopy theory which are collected in Appendix \ref{sec:metastable}. In Section \ref{sec:weiss+disjunction} we first describe the ``Weiss fibre sequence'' mentioned above, and then use surgery theory and disjunction techniques to replace the space of self-embeddings of $X_g$ (relative to half the boundary) by the analogous space of ``block'' self-embeddings. In Section \ref{sect:Emb mcg} we determine (up to extension) the set of path-components of these spaces of self-embeddings, using surgery theory and the results of Section \ref{sec:weiss+disjunction}. Section \ref{sec:sfr} concerns the higher homotopy groups of these spaces of self-embeddings, though we simplify the calculations by instead considering \textit{stably framed} self-embeddings. We must then revisit the results of Section \ref{sect:Emb mcg} and upgrade them to the stably framed setting. Section \ref{sec:proofs} then contains the proofs of the main theorems, using the calculations of the previous sections. A further ingredient is the calculation of certain modules of coinvariants of actions of unitary groups over $\Z[t,t^{-1}]$, which is done in Appendix \ref{sec:Coinv}. Finally, in Appendix \ref{sec:classical} we discuss the relation of this paper with the more classical methods of calculation via algebraic $K$-theory.
\subsection*{Acknowledgements} 
The authors would like to thank F.\ Hebestreit, M.\ Krannich, and M.\ Land for useful discussions. The authors were supported by the ERC under the European Union's Horizon 2020 research and innovation programme (grant agreement No.\ 756444), and ORW was supported by a Philip Leverhulme Prize from the Leverhulme Trust.


\section{Manifolds and quadratic modules}\label{sect:mflds-quad}

\subsection{Manifolds}

There are two families of manifolds we shall consider. The first is given by the manifolds
$$W_{g,1} := D^{2n} \# (S^n \times S^n)^{\# g}$$
obtained by connect sum of the $2n$-disc with $g$ copies of $S^n \times S^n$. These have boundary $\partial W_{g,1} = S^{2n-1}$, and we write $\tfrac{1}{2} \partial W_{g,1} = D^{2n-1}_-$ for the lower hemisphere. There is a homotopy equivalence $W_{g,1} \simeq \bigvee^{2g} S^n$.

The second family is given by the manifolds
$$X_g := (S^1 \times D^{2n-1}) \# (S^n \times S^n)^{\# g}$$
obtained by connect sum of the solid torus $S^1 \times D^{2n-1}$ with $g$ copies of $S^n \times S^n$. These have boundary $\partial X_{g} = S^1 \times S^{2n-2}$, and we write $\tfrac{1}{2} \partial X_{g} = S^1 \times D^{2n-2}_-$ for the circle times the lower hemisphere. There is a homotopy equivalence $X_g \simeq S^1 \vee \bigvee^{2g} S^{2n}$. We often think of $X_g$ as being obtained from $W_{g,1}$ by trivially attaching a 1-handle along an embedding $e : S^0 \times D^{2n-1} \hookrightarrow D^{2n-2}_+ \subset \partial W_{g,1}$, so that $\tfrac{1}{2} \partial W_{g,1} \subset \tfrac{1}{2} \partial X_g$. If we then attach a cancelling 2-handle, we obtain an embedding $X_g \hookrightarrow W_{g,1}$, such that $W_{g,1} \hookrightarrow X_g \hookrightarrow W_{g,1}$ is isotopic to the identity. These manifolds are depicted in Figure \ref{fig:Xg-Wg}.
\begin{figure}[h]
\centering
\begin{subfigure}{.5\textwidth}
  \centering
  \includegraphics[scale=0.3]{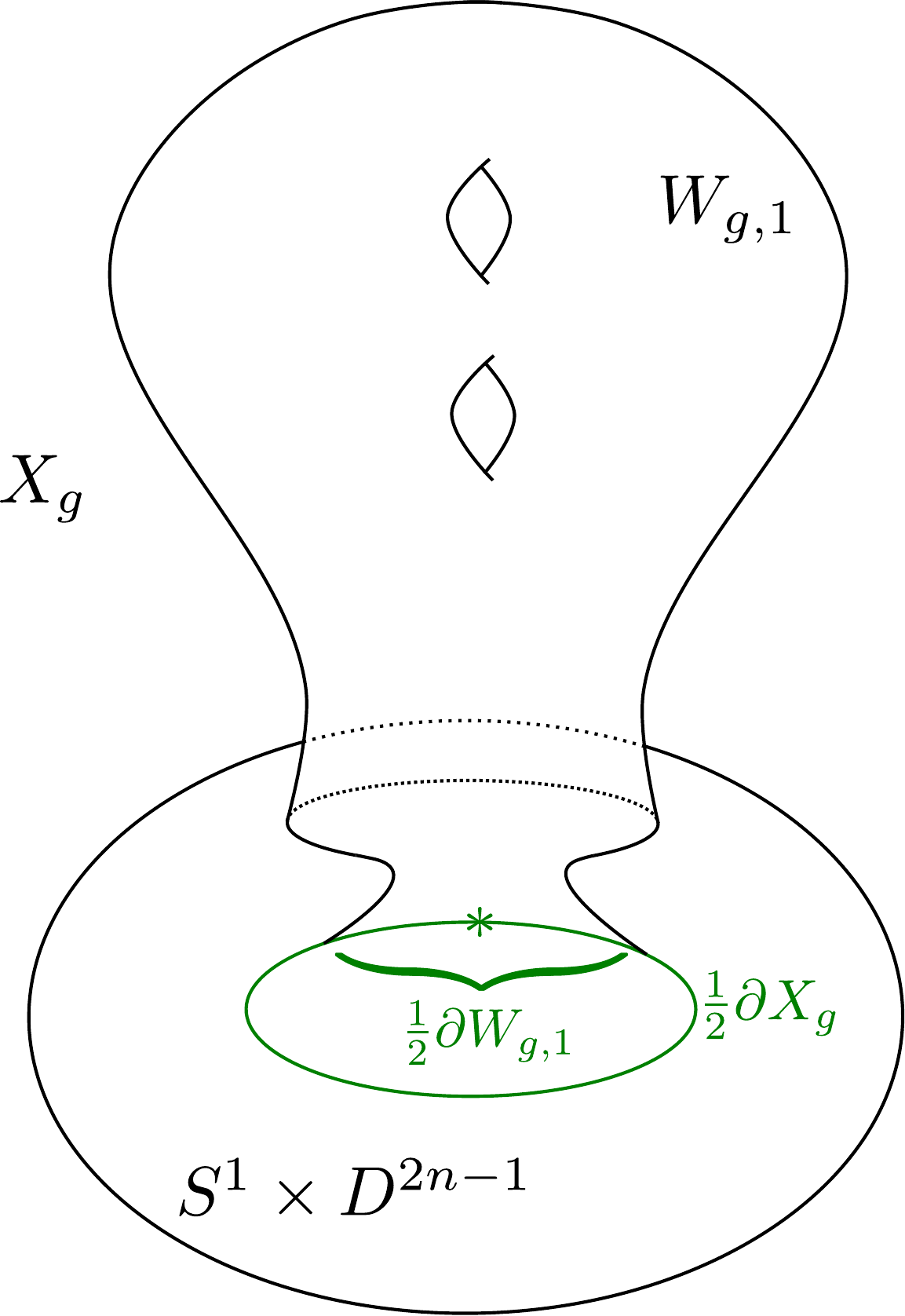}
  \label{fig:sub1}
\end{subfigure}%
\begin{subfigure}{.5\textwidth}
  \centering
  \includegraphics[scale=0.3]{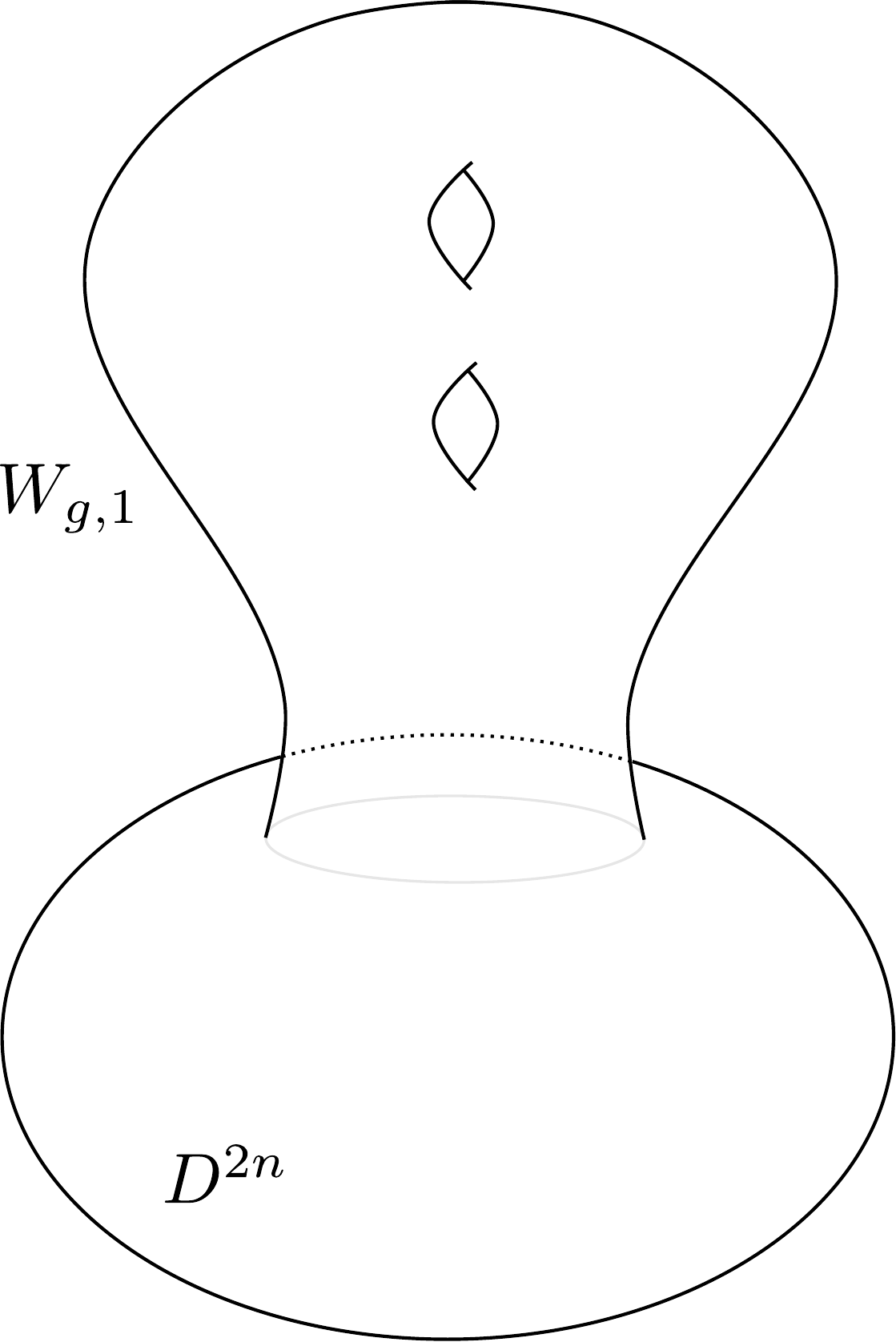}
  \label{fig:sub2}
\end{subfigure}
\caption{The manifolds $X_g$ (left) and $W_{g,1}$ (right) for $g=2$. Note that $W_{g,1}\hookrightarrow X_g\hookrightarrow W_{g,1}$.}
\label{fig:Xg-Wg}
\end{figure}

We choose a basepoint $* \in \tfrac{1}{2} \partial W_{g,1} \subset \tfrac{1}{2} \partial X_g$, and write $\pi = \pi_1(X_g, *) \cong \langle t \rangle$ for the infinite cyclic fundamental group. We take homotopy groups of $W_{g,1}$ and $X_g$ to be based at $*$, and consider those of $X_g$ as \emph{left} $\Z[\pi]$-modules. As usual, we use the antiinvolution $t \mapsto \bar{t} = t^{-1}$ of the ring $\Z[\pi]$ to convert between left and right $\Z[\pi]$-modules when necessary. As a matter of notation, we write
$$H := H_n(W_{g,1};\Z),$$
which is also isomorphic to $H_n(X_{g};\Z)$ via the natural map.

We write $a_i, b_i \in \pi_n(W_{g,1}, *) \subset \pi_n(X_g, *)$ for the homotopy classes of $S^n \times \{*\}$ and $\{*\} \times S^n$ inside the $i$th copy of $S^n \times S^n$ which, by the Hurewicz theorem, give $\Z$- or $\Z[\pi]$-module bases of $\pi_n(W_{g,1},*)$ or $\pi_n(X_g, *)$ respectively. In particular, we have natural identifications
$$\pi_n(W_{g,1}, *) \cong H \quad \text{ and } \quad \pi_n(X_g, *) \cong \Z[\pi] \otimes H$$
of left $\Z$- and $\Z[\pi]$-modules respectively.  We will write $x_i : S^n \to W_{g,1} \subset X_{g}$ for the inclusion of the $i$th $n$-sphere, which we choose by $x_i = a_i$ and $x_{g+i} = b_i$.

\subsection{Quadratic modules}\label{subsec:quadMod}
Following Wall \cite[Theorem 5.2]{wall-book}, in \cite[Section 3]{GRWHomStabI} and more precisely in \cite[Section 4.1]{Friedrich} it has been explained how to naturally associate a quadratic module to an oriented manifold $V$ of dimension $2n \geq 6$. Let us explain a variant of this, endowing $\pi_n(V)$ with a quadratic module structure when the tangent bundle of $V$ is oriented and trivial over the $n$-skeleton. We write $\pi = \pi_1(V, *)$, and denote the identity element in $\pi$ by $e$.

Firstly, there is an equivariant intersection form
\beq
\lambda : \pi_n(V, *) \otimes \pi_n(V, *) \longrightarrow \Z[\pi]
\eeq
defined by
\beq
\lambda(x\otimes y)=\sum_{g \in \pi} \langle g \cdot x,  y \rangle \cdot g^{-1},
\eeq
where we identify $\pi_n(V, *)\cong\pi_n(\widetilde{V}, *)$ and $\langle -, - \rangle$ denotes intersection form of the universal cover $\widetilde{V}$. It satisfies $\overline{\lambda(x,y)} = (-1)^n \lambda(y,x)$ and $\lambda(a \cdot x, b \cdot y) = a \cdot \lambda(x,y) \cdot \bar{b}$ for $a, b \in \Z[\pi]$, i.e.\ is a $(-1)^n$-hermitian form. 

Secondly, to obtain a quadratic refinement, choose a basepoint $b_V \in \Fr^+(TV)$ in the oriented frame bundle lying over $* \in V$. By the assumption that the tangent bundle of $V$ is trivial over the $n$-skeleton, there are exact sequences
\begin{align*}
0 \to \pi_1(\SOn(2n)) \xto{i} &\,\pi_1(\Fr^+(TV), b_V) \to \pi_1(V, *) \to 0\\
\pi_n(\SOn(2n)) \xto{i} &\,\pi_n(\Fr^+(TV), b_V) \to \pi_n(V, *) \to 0.
\end{align*}
A choice of framing of $S^n \times D^n$ determines a map from the set $\Imm^{\fr}_n(V)$, of regular homotopy classes of immersions $i: S^n\times D^n \looparrowright V$ equipped with a path from $Di(b_{S^n \times D^n})$ to the base point $b_V$ in $\Fr^+(TV)$, to the abelian group $\pi_n(\Fr^+(TV), b_V)$, and this is a bijection by Hirsch--Smale theory. Using the map $\Imm^{\fr}_n(V) \cong \pi_n(\Fr^+(TV), b_V) \to \pi_n(V,*)$ we may consider $\lambda$ as a pairing 
$$\lambda^{\fr} :\Imm^{\fr}_n(V) \otimes \Imm^{\fr}_n(V) \to \Z[\pi],$$ 
which is sesquilinear over $\pi_1(\Fr^+(TV), b_V) \to \pi_1(V, *)=\pi$.

Consider the following subgroup of $\Z[\pi]$
$$\Lambda_n^{\min} := \{a- (-1)^n\bar{a}\ |\ a\in\Z[\pi]\}.$$
As in \cite[Theorem 5.2]{wall-book} there is a function
$$q^{\fr} : \Imm^{\fr}_n(V) \to \Z[\pi]/\Lambda_n^{\min}$$
given by counting signed self-intersections of the core of a framed immersion, and if $a \in \pi_1(\Fr(TV), b_V)$ maps to $g \in \pi$, then it satisfies
\begin{equation*}
q^{\fr}(a \cdot x) = g \cdot q^{\fr}(x) \cdot \bar{g},
\end{equation*}
which is well-defined since  $g\Lambda_{n}^{\min}\bar{g}\subset\Lambda_n^{\min}$. Furthermore 
\begin{equation}\label{eq:Quad2}
q^{\fr}(x+y)- q^{\fr}(x) - q^{\fr}(y) \equiv \lambda^{\fr}(x,y)\ \mathrm{mod}\, \Lambda_n^{\min}.
\end{equation}
We wish to descend $q^{\fr}$ to a quadratic form on $\pi_n(V, *)$. As the pairing $\lambda^{\fr}$ on $\Imm^{\fr}_n(V) \cong\pi_n(\Fr^+(TV), b_V)$ is defined to factor over $\pi_n(V,*)$, the subgroup $i(\pi_n(\SOn(2n)))$ lies in the radical of $\lambda^{\fr}$, so by \eqref{eq:Quad2} the function $q^{\fr} \circ i$ is $\Z$-linear.

\begin{lem}\label{lem:HopfDims}
The image of the homomorphism $q^{\fr} \circ i : \pi_n(\SOn(2n)) \to \Z[\pi]/\Lambda_n^{\min}$ is
\begin{enumerate}[(i)]
\item trivial if $n\neq 3,7$, 
\item generated by $e \mod \Lambda_n^{\min}$ if $n$ is 3 or 7.
\end{enumerate}
\end{lem}
\begin{proof}
Interpreting $\pi_n(\SOn(2n))$ as the framed immersions $S^n \times D^n \looparrowright D^{2n} \subset V$ into a disc near the basepoint, it follows that this homomorphism has image in the subgroup spanned by $e \in \pi$. We then refer to \cite[Lemma 2.7]{KR-WFram}: briefly, if $n=3,7$ then the Whitney ``figure eight'' immersion is sent to $e \mod \Lambda_n^{\min}$, and if $n \neq 3,7$ then if there were an immersion sent to $e$ then connect-summing its core with the ``figure eight'' and removing double points would give an embedding $S^n \hookrightarrow D^{2n}$ with nontrivial normal bundle, which is impossible.
\end{proof}

In view of this lemma, letting $\Lambda_n = \Lambda_n^{\min}$ for $n \neq 3,7$, and $\Lambda_n$ be the subgroup of $\Z[\pi]$ generated by $\Lambda_n^{\min}$ together with $e$ if $n = 3,7$, the function $q^{\fr}$ descends to a well-defined function
$$q : \pi_n(V) \to \Z[\pi]/\Lambda_n,$$
which is a quadratic refinement of $\lambda: \pi_n(V) \otimes \pi_n(V) \to \Z[\pi]$ in the sense that it satisfies 
$$q(g \cdot x) = g \cdot q(x) \cdot \bar{g} \quad \text{ and } \quad q(x+y) - q(x) - q(y) \equiv \lambda(x,y) \, \mathrm{mod} \, \Lambda_n.$$

\begin{lem}\label{lem:QuadDetectsEmb}
As long as $n \geq 3$, an element $x \in \pi_n(V)$ may be represented by an embedding $S^n \times D^n \hookrightarrow V$ if and only if $q(x)=0$.
\end{lem}
\begin{proof}
The necessity is clear. For sufficiency, choose a lift $x' \in \Imm^{\fr}_n(V) \cong \pi_n(\Fr^+(TV), b_V)$ of $x$. It follows from \cite[Theorem 5.2]{wall-book} that the immersion $x'$ is regularly homotopic to an embedding if and only if $q^\mathrm{fr}(x')=0 \in \Z[\pi]/\Lambda_n^\mathrm{min}$. If $n \neq 3,7$ then $q^\mathrm{fr}(x') = q(x) \in \Z[\pi]/\Lambda_n^\mathrm{min}$ and we are done. If $n = 3, 7$ then $q(x)=0$ means that $q^\mathrm{fr}(x') \in \Lambda_n/\Lambda_n^\mathrm{min}$, so is a multiple of $e$. But then by Lemma \ref{lem:HopfDims} (ii) we may re-choose the lift $x'$ of $x$ by adding on a suitable element of $i(\pi_n(\SOn(2n)))$ to obtain an $x''$ for which $q^\mathrm{fr}(x'')=0$.
\end{proof}

Let us now specialise to the manifolds $W_{g,1}$ and $X_g$, which both have trivial tangent bundles, giving a quadratic module $(\pi_n(W_{g,1}), \lambda_W, q_W)$ over $\Z$ and a quadratic module $(\pi_n(X_{g}), \lambda_X, q_X)$ over $\Z[\pi] \cong \Z[t^{\pm 1}]$, with form parameters $((-1)^n, \Lambda_n)$ (though what $\Lambda_n$ denotes in the two cases is different). In both cases, the hermitian forms $\lambda_W$ and $\lambda_X$ are non-degenerate, the pairs $(a_i, b_i)$ are orthogonal hyperbolic pairs with respect to $\lambda_W$ or $\lambda_X$, and the quadratic refinements $q_W$ and $q_X$ vanish on the $a_i$ and $b_i$'s. This identifies $(\pi_n(W_{g,1}), \lambda_W, q_W)$ and $(\pi_n(X_{g}), \lambda_X, q_X)$ as genus $g$ hyperbolic quadratic modules over the appropriate rings and with the appropriate form parameters. 

We write
$$U_g(\Z, \Lambda_n) \quad \text{ and } \quad U_g(\Z[\pi], \Lambda_n)$$
for the associated unitary groups. By construction, any diffeomorphism of $W_{g,1}$ or $X_g$ relative to their boundaries, or more generally any smooth self-embedding of these manifolds relative to ``half their boundaries", induces a morphism of the associated quadratic module. 

\begin{rem}\label{rem:QuadTrivNEven}
If $n$ is even then we have $4 q_X(x) = q_X(2x) = q_X(x+x) = q_X(x)+q_X(x)+\lambda_X(x,x)$ and so $2q_X(x) = \lambda_X(x,x)$. Furthermore $\Lambda_n = \langle t^a - t^{-a} \, | \, a =1,2,3,\ldots \rangle_\Z$ is a direct summand of $\Z[\pi]$, so $\Z[\pi]/\Lambda_n$ is torsion-free: therefore $\lambda_X$ determines $q_X$ in this case. If $n$ is odd then $q_X$ is not determined by $\lambda_X$.
\end{rem}

\section{Homotopy automorphisms of $X_g$}\label{sec:hautXg}

In this section we will analyse the homotopy groups of the topological monoid $\hAut_{\partial}(X_g)$ of self-homotopy equivalences of the manifold 
$X_g$ which fix the boundary  $\partial X_g =S^1\times S^{2n-2}$ pointwise. For definitions and basic properties of Whitehead products we follow \cite{WhiteheadBook}. We continue using our notation from the previous section. In particular, the $\Z$- or $\Z[\pi]$-module bases of $\pi_n(W_{g,1},*)$ or $\pi_n(X_g, *)$ respectively are given by the elements $a_i, b_i \in \pi_n(W_{g,1}, *) \subset \pi_n(X_g, *)$ representing the homotopy classes of $S^n \times \{*\}$ and $\{*\} \times S^n$ inside the $i$th copy of $S^n \times S^n$. 
\subsection{A fibration sequence}

Our analysis is based on the fibration sequence
\begin{equation}\label{eq:XFibn}
\Map_{\partial}(X_g, X_g)\to\Map_{S^1}(X_g, X_g)\xto{\rho}\Map_{S^1}(S^1 \times S^{2n-2},X_g)
\end{equation}
given by restricting maps $f : X_g \to X_g$ (which are the identity on $S^1 \subset \partial X_g = S^1 \times S^{2n-1}$) along the inclusion $\iota : \partial X_g \to X_g$. We will completely describe the restriction map $\rho$ at the level of homotopy groups. To express our answer, note that the homotopy class of the inclusion
$$S^{2n-1} = \partial W_{g,1} \subset W_{g,1} \subset X_g$$
is the sum of Whitehead brackets $\omega := \sum_{i=1}^g [a_i, b_i] \in \pi_{2n-1}(X_g)$. The coinvariants for the action of $\pi$ on $\pi_j(X_g)$ will be denoted by $[\pi_j(X_g)]_{\pi}$. 
\begin{thm}\label{thm:HtyAut}
There are isomorphisms
\begin{align*}
\pi_k(\Map_{S^1}(X_g, X_g), \id) &\cong \Hom_{\Z[\pi]}(\pi_n(X_g), \pi_{k+n}(X_g))\\
\pi_k(\Map_{S^1}(S^1 \times S^{2n-2}, X_g), \iota) &\cong [\pi_{2n-1+k}(X_g)]_{\pi},
\end{align*}
under which the map $\pi_0(\rho)$ is given by
$$\phi \mapsto \sum_{i=1}^g [\phi(a_i), \phi(b_i)] - [a_i, b_i]$$
and, for $k>0$, the map  $\pi_k(\rho)$ is given by
$$\phi \mapsto \sum_{i=1}^g [\phi(a_i), b_i] + (-1)^{n k}[a_i, \phi(b_i)].$$
Under the identification 
\begin{align*}
\pi_n(X_g) \otimes_{\Z[\pi]} \pi_{k+n}(X_g) &\xlongrightarrow{\sim} \Hom_{\Z[\pi]}(\pi_n(X_g), \pi_{k+n}(X_g))\\
x \otimes y &\longmapsto \lambda_X(-, x) \cdot y
\end{align*}
(where we have used the antiinvolution to consider $\pi_n(X_g)$ as a right $\Z[\pi]$-module) the latter map is given by
$$x \otimes y \mapsto [y,x].$$
\end{thm}

The following construction and lemma will be used in the proof of Theorem \ref{thm:HtyAut}. Let $X$ be a based space, $s : S^n \to X$ be a based map, and $m>0$. Then there is a map
$$[-,s] : \Omega^m_0 X \to \Omega_0^{m+n-1}X$$ 
given by sending $g\in\Omega^m_0 X$ to the element 
$[g,s]\in\Omega_0^{m+n-1} X$ given by
\beq
g\circ pr_1\cup s\circ pr_2: \partial(I^m \times I^n) = I^m\times\partial I^n\cup\partial I^m\times I^n\to X,
\eeq
and similarly a map $[s, -] : \Omega^m_0 X \to \Omega_0^{m+n-1}X$ given by sending $g\in\Omega^m_0 X$ to the map $s\circ pr_1\cup g\circ pr_2: \partial(I^n \times I^m) =  I^n\times\partial I^m\cup\partial I^n\times I^m\to X$.

\begin{lem}\label{lem:WhiteheadMap}
The map $[-,s] : \Omega^m_0 X \to \Omega_0^{m+n-1}X$ induces\footnote{This is not a based map, but the spaces are simple so it nonetheless induces a well-defined map on homotopy groups.} $f \mapsto [f,s] : \pi_{k+m}(X) = \pi_{k}(\Omega^m_0X) \to \pi_{k}(\Omega_0^{m+n-1}X) = \pi_{k+n+m-1}(X)$ on homotopy groups. 

Consequently, the map $[s, -] : \Omega^m_0 X \to \Omega_0^{n+m-1}X$ induces $f \mapsto (-1)^{nk}[s,f]$ on $\pi_k(-)$.
\end{lem}
\begin{proof}
Using $\partial(I^s\times I^t)=I^s\times\partial I^t\cup\partial I^s\times I^t$, it follows that on $\pi_k$ the map $[-,s]$ is given by sending $f \in \pi_k(\Omega^m_0 X) \cong \pi_{k+m}(X)$ to the map $\partial (I^k \times I^m \times I^n) \to X$ defined by
\beq
f\circ pr_1\cup s\circ pr_2:(I^k \times I^m) \times \partial I^n\cup\partial (I^k \times I^m) \times I^n\to X
\eeq
which is $[f,s]$ by definition.

We deduce the second part from the first, using that the map $[s,-]$ is homotopic to $(-1)^{nm}[-,s]$ which by the first part induces the map $f \mapsto (-1)^{nm}[f,s]$ on homotopy groups, giving the required
\beq
(-1)^{nm} (-1)^{n(m+k)}[s,f] =  (-1)^{nk}[s,f].\qedhere
\eeq
\end{proof}
\begin{proof}[Proof of Theorem \ref{thm:HtyAut}]
Firstly, there is an equivalence $\Map_{S^1}(X_g, X_g) \simeq \Map_{*}(\vee^{2g} S^n, X_g)$, and taking homotopy groups based at the constant map $ct$ therefore gives
$$\pi_k(\Map_{*}(\vee^{2g} S^n, X_g), ct) \cong \Hom_{\Z}(H, \pi_{n+k}(X_g)).$$
Using the identification $\pi_n(X_g) \cong \Z[\pi] \otimes H$ we can write this as $\Hom_{\Z[\pi]}(\pi_n(X_g), \pi_{n+k}(X_g))$. Now the co-$H$-space structure on $\vee^{2g} S^n$ turns $\Map_{*}(\vee^{2g} S^n, X_g)$ into a group-like $H$-space, and therefore all its path components are homotopy equivalent. Hence adding on the inclusion map $inc : \vee^{2g} S^n \to X_g$ gives the required isomorphism
$$\pi_k(\Map_{*}(\vee^{2g} S^n, X_g), ct) \xto{\sim} \pi_k(\Map_{*}(\vee^{2g} S^n, X_g), inc) \cong \pi_k(\Map_{S^1}(X_g, X_g), \id).$$

Secondly, expressing $S^1 \times S^{2n-2}$ as the pushout of $(S^1 \vee S^{2n-2}) \xleftarrow{[i_1, i_{2n-2}]} S^{2n-2} \to D^{2n-1}$, where $i_1:S^1\to S^1\vee S^{2n-2}$ and $i_{2n-2}:S^{2n-2}\to S^1\vee S^{2n-2}$ are the inclusions,
there is a homotopy cartesian square
\beq
\xymatrix{
\Map_{S^1}(S^1 \times S^{2n-2}, X_g) \ar[d]\ar[r] & \Map_{*}( D^{2n-1}, X_g)\ar[d]\\
\Map_{S^1}(S^1 \vee S^{2n-2}, X_g) \ar[r]^-{w} & \Map_{*}(S^{2n-2}, X_g)
}
\eeq
whose top right-hand corner is contractible. The map $w$ is given by precomposition with the attaching map $[i_1, i_{2n-2}] : S^{2n-2} \to S^1 \vee S^{2n-2}$, so under the evident identification $\Map_{S^1}(S^1 \vee S^{2n-2}, X_g) = \Map_{*}( S^{2n-2}, X_g)$ it is given on $k$th homotopy groups by 
$$[t,-] : \pi_{k+2n-2}(X_g) \to \pi_{k+2n-2}(X_g).$$
In terms of the $\Z[\pi]$-module structure on homotopy groups this is the map $x \mapsto t\cdot x - x$. We claim that this is injective. As $X_g \simeq S^1 \vee \bigvee^{2g} S^n$, an element in the kernel of this map is represented by a map $x: S^{k+2n-2} \to \widetilde{X_g}$ whose (unbased) homotopy class is invariant with respect to the action of $\pi$ by deck transformations, but this means that $x$ is null (it is homotopic to maps lying in disjoint finite wedges of $n$-spheres; projecting to these finite summands shows it is trivial). The long exact sequence on homotopy groups therefore gives an isomorphism
$$[\pi_{k+2n-1}(X_g)]_\pi \xto{\sim} \pi_k(\Map_{S^1}(S^1 \times S^{2n-2}, X_g), \iota)$$
as required. By inspecting the vertical homotopy fibers in square above, we see that
this isomorphism is induced by acting on $\iota$ via the pinch map $p : S^1 \times S^{2n-2} \to S^1 \times S^{2n-2} \vee S^{2n-1}$. In fact the map from $\Map_*(S^{2n-1},X_g)$ to the homotopy fiber (over the inclusion) of the left vertical map, given by sending $\phi\in \Map_*(S^{2n-1},X_g)$ to the pair consisting of the composite
\beq
S^1\times S^{2n-2}\xto{p} S^1\times S^{2n-2}\vee S^{2n-1}\xto{\iota\vee\phi}X_g,
\eeq
and the constant homotopy of the inclusion $S^1\vee S^{2n-2}\hookrightarrow S^1\times S^{2n-2}\xto{\iota} X_g$, is a homotopy equivalence.

Thirdly, we determine the map $\pi_k(\rho)$. The inclusion $\iota: S^1 \times S^{2n-2} \to X_g$ has an evident nullhomotopy on $* \times S^{2n-2}$, giving a factorisation up to homotopy
$$\iota : S^1 \times S^{2n-2} \to \frac{S^1 \times S^{2n-2}}{* \times S^{2n-2}} \xleftarrow{\sim} S^1 \vee S^{2n-1} \xto{\iota_{S^1} \vee q}X_g$$
where $\iota_{S^1}$ is the restriction of the inclusion $\iota$ to $S^1\times\ast\subset S^1\times S^{2n-1}$, and $q : S^{2n-1} \to \bigvee_{j=1}^{2g} S^{n}_j$ is the inclusion of the boundary $S^{2n-1} \to W_{g,1}$, so represents the homotopy class $\omega = \sum_{i=1}^g [a_i, b_i]$. Thus the restriction map
$$\Map_{S^1}(X_g, X_g) \xto{\rho} \Map_{S^1}(S^1 \times S^{2n-2}, X_g)$$
along $\iota$ is homotopic to the map sending $f$ to
$$S^1 \times S^{2n-2} \xto{p} S^1 \times S^{2n-2} \vee S^{2n-1} \xto{pr_1}S^1 \vee S^{2n-1} \xto{\iota_{S^1} \vee q}X_g \xto{f} X_g.$$
Thus there is a factorisation
\begin{align*}
\Map_{S^1}(X_g, X_g) \simeq \Map_*(\vee^{2g} S^n, X_g) &\xto{q^*} \Map_*(S^{2n-1},X_g)\\
&\xto{-q}\Map_*(S^{2n-1},X_g)\to \Map_{S^1}(S^1 \times S^{2n-2}, X_g)
\end{align*}
where the first map is precomposition by $q$, the second map is given by translation by the inverse of $q$ using the $H$-space structure on $\Map_*(S^{2n-1},X_g)$, and the third map is given by acting on $\iota$ by the pinch map. Our calculation of the homotopy groups of $\Map_{S^1}(S^1 \times S^{2n-2}, X_g)$ shows that this pinch action induces an isomorphism on homotopy groups after taking coinvariants for the $\pi$-action on the domain.

On $\pi_0$ we now get the claimed formula: to a $\phi : \pi_n(X_g) \to \pi_n(X_g)$ representing a map $\alpha_\phi: X_g \to X_g$, $\pi_0(\rho)(\phi)$ is given by
$$S^{2n-1} \xto{\nabla} S^{2n-1} \vee S^{2n-1} \xto{S^{2n-1} \vee q}S^{2n-1} \vee X_g\xto{-\omega \vee \alpha_\phi} X_g$$
so is $\sum_{i=1}^{g} [\phi(a_i), \phi(b_i)] - [a_i, b_i] \in \pi_{2n-1}(X_g)$.

To analyse the effect on higher homotopy groups, we precompose with the equivalence
$$\Map_*(\vee^{2g} S^n, X_g) \xto{+inc} \Map_*(\vee^{2g} S^n, X_g)$$
given by translation by $inc$ using the $H$-space structure on $\Map_*(\vee^{2g} S^n, X_g)$ coming from $\vee^{2g} S^n$ being a co-$H$-space, as above. We therefore want to compute the effect on homotopy groups of the composition
\begin{equation}\label{eq:LongComp}
\Map_*(\vee^{2g} S^n, X_g) \xto{+inc} \Map_*(\vee^{2g} S^n, X_g) \xto{q^*} \Map_*(S^{2n-1},X_g)\xto{-q}\Map_*(S^{2n-1},X_g)
\end{equation}
based at the constant map. This composition sends a map $f : \vee^{2g} S^n \to X_g$ to 
\begin{align*}
S^{2n-1} \xto{\nabla} S^{2n-1} \vee S^{2n-1} &\xto{S^{2n-1} \vee q} S^{2n-1} \vee (\vee^{2g} S^n)\\
 &\xto{S^{2n-1} \vee \nabla} S^{2n-1} \vee (\vee^{2g} S^n) \vee (\vee^{2g} S^n) \xto{-q \vee inc \vee f} X_g.
\end{align*}
The first three of these maps represents the homotopy class
$$ i_{2n-1} + \sum_{i=1}^g [a_i^1 + a_i^2, b_i^1 + b_i^2] \in \pi_{2n-1}(S^{2n-1} \vee (\vee^{2g} S^n) \vee (\vee^{2g} S^n))$$
where $a_i^r$ is the map $a_i$ considered as a map to the $r$th wedge summand of $\vee^{2g} S^n$. We can expand this out as
$$i_{2n-1} + \sum_{i=1}^g [a_i^1, b_i^1] + [a_i^1, b_i^2] + [a_i^2, b_i^1] + [a_i^2, b_i^2].$$
The composition \eqref{eq:LongComp} is therefore homotopic to the map which sends $f$ to
$$-q + \sum_{i=1}^g [a_i, b_i] + [a_i, f \circ b_i] + [f \circ a_i,  b_i] + [f\circ a_i, f \circ b_i],$$
where $+$ denotes repeated cleaving maps $\nabla$. As $-q + \sum_{i=1}^g [a_i, b_i] \simeq 0$, we can remove these terms. If $f_s \in \Omega^k \Map_*(\vee^{2g} S^n, X_g)$ represents an element of $\pi_k(\Map_*(\vee^{2g} S^n, X_g))$, then $t \mapsto \sum_{i=1}^g[f_t \circ a_i, f_t \circ b_i]$ is adjoint to
$$S^k \wedge S^{2n-1} \xto{S^k \wedge q} S^{k} \wedge (\vee^{2g} S^n) \xto{f} X_g$$
but as suspensions of Whitehead products are trivial, this is nullhomotopic. This leaves
$$s \mapsto \sum_{i=0}^g [a_i, f_s \circ b_i] + [f_s \circ a_i, b_i] : S^k \to \Map_*(S^{2n-1},X_g).$$
Lemma \ref{lem:WhiteheadMap} above shows that this corresponds to
$$\sum_{i=0}^g [\phi(a_i), b_i] + (-1)^{n k}[a_i, \phi(b_i)] \in \pi_{k+2n-1}(X_g).$$

For the final statement, note that under the isomorphism $\pi_n(X_g) \otimes_{\Z[\pi]} \pi_{k+n}(X_g) \cong \Hom_{\Z[\pi]}(\pi_n(X_g), \pi_{k+n}(X_g))$ the class $b_j \otimes y$ corresponds to the function $\lambda_X(-, b_j) \cdot y$. Thus under $\pi_k(\rho)$ this maps to
$$\sum_{i=1}^g [\lambda_X(a_i, b_j) \cdot y, b_i] + (-1)^{n k}[a_i, \lambda_X(b_i, b_j) \cdot y] = [y, b_j];$$
similarly, $a_j \otimes y$ corresponds to $\lambda_X(-, a_j) \cdot y$ which maps to
\begin{align*}
\sum_{i=1}^g [\lambda_X(a_i, a_j) \cdot y, b_i] + (-1)^{nk} [a_i, \lambda_X(b_i, a_j) \cdot y] = (-1)^{nk + n}[a_j, y]= [y, a_j]
\end{align*}
As classes of the form $a_j \otimes y$ and $b_j \otimes y$ generate $\pi_n(X_g) \otimes_{\Z[\pi]} \pi_{k+n}(X_g)$, this proves the final statement.
\end{proof}

Similarly, for the manifolds $W_{g,1}$ there is a fibration sequence
\begin{equation*}
\Map_{\partial}(W_{g,1}, W_{g,1})\to\Map_{*}(W_{g,1}, W_{g,1})\xto{\rho}\Map_{*}(S^{2n-1},W_{g,1})
\end{equation*}
and an analogous, but much simpler argument proves the following.

\begin{thm}\label{thm:HtyAutWg}
There are isomorphisms
\begin{align*}
\pi_k(\Map_{*}(W_{g,1}, W_{g,1}), id) &\cong \Hom_{\Z}(\pi_n(W_{g,1}), \pi_{k+n}(W_{g,1}))\\
\pi_k(\Map_{*}(S^{2n-1}, W_{g,1}), \iota) &\cong \pi_{2n-1+k}(W_{g,1}),
\end{align*}
under which the map $\pi_0(\rho)$ is given by $$\phi \mapsto \sum_{i=1}^g [\phi(a_i), \phi(b_i)] - [a_i, b_i]$$ and, for $k>0$, the map  $\pi_k(\rho)$ is given by
$$\phi \mapsto \sum_{i=1}^g [\phi(a_i), b_i] + (-1)^{n k}[a_i, \phi(b_i)].$$ 
Under the identification 
\begin{align*}
\pi_n(W_{g,1}) \otimes_{\Z} \pi_{k+n}(W_{g,1}) &\xlongrightarrow{\sim} \Hom_{\Z}(\pi_n(W_{g,1}), \pi_{k+n}(W_{g,1}))\\
x \otimes y &\longmapsto \lambda_W(-, x) \cdot y
\end{align*}
the latter map is given by $x \otimes y \mapsto [y,x]$.\qed
\end{thm}

\subsection{The mapping class group}\label{sec:hAut-MCG}
We now restrict our attention to the groups $\pi_*(\hAut_\partial(X_g))$ and $\pi_*(\hAut_\partial(W_{g,1}))$ for degrees $* < n-1$, where the homotopy groups showing up in Theorems \ref{thm:HtyAut} and \ref{thm:HtyAutWg} are metastable. We use some tools from metastable homotopy theory to establish the calculations we will need. For the reader's convenience, we collect those tools in Appendix \ref{sec:metastable} at the end of the paper.

We begin analysing the mapping class group $\pi_0(\hAut_{\partial}(X_g))$. Restricting to those path components $\hAut_\partial(X_g) \subset \Map_\partial(X_g, X_g)$ consisting of homotopy equivalences, we have an exact sequence
$$\cdots \to \pi_0(\hAut_\partial(X_g)) \to GL_{\Z[\pi]}(\pi_n(X_g)) \xto{\phi \mapsto\sum [\phi(a_i),\phi(b_i)]-[a_i,b_i]} [\pi_{2n-1}(X_g)]_\pi,$$
showing that the image of $\pi_0 \hAut_\partial(X_g)$ in $GL_{\Z[\pi]}(\pi_n(X_g)) \cong GL_{2g}(\Z[\pi])$ are those automorphisms $\phi$  for which $\sum [\phi(a_i),\phi(b_i)] = \sum[a_i,b_i]$. The following lemma spells out what this means in concrete terms; it uses the notation $X^\dagger$ for the hermitian transpose of a matrix over $\Z[\pi]$.

\begin{lem}
Let $\phi \in GL_{2g}(\Z[\pi])$ given in block form by $\begin{psmallmatrix}
A & B \\
C & D
\end{psmallmatrix}$ with respect to the basis $a_1, a_2, \ldots, a_g, b_1, b_2, \ldots, b_g$. Then $\sum [\phi(a_i),\phi(b_i)]=\sum[a_i,b_i]$ if and only if
\begin{enumerate}[(i)]
\item $A D^\dagger + (-1)^n B C^\dagger = I$,
\item $AB^\dagger + (-1)^n (AB^\dagger)^\dagger$ and $CD^\dagger + (-1)^n (CD^\dagger)^\dagger$ vanish off the diagonal,
\item the diagonal entries of $AB^\dagger$ and $ CD^\dagger$ lie in $\Lambda_n$.
\end{enumerate}
\end{lem}
\begin{proof}
The element $\sum_{i=1}^g [a_i, b_i] = \sum_{i=1}^g [x_i, x_{g+i}]$ is supported on the second term of \eqref{eq:Hilton-X_g} in Appendix \ref{sec:Hilton-Milnor} with $k=0$. As $[x, t \cdot y] = t \cdot [t^{-1} \cdot x, y]$, we have $[x, t \cdot y]\sim [t^{-1} \cdot x, y]$ in the $\pi$-coinvariants, and hence $[ x, a \cdot y] = [\bar{a} \cdot x, y]$. Now 
\begin{align*}
\sum [\phi(a_i),\phi(b_i)] &=\sum_{i=1}^g [\sum_{j=1}^g A_{ji} a_j + C_{ji}b_j, \sum_{k=1}^g B_{ki} a_k + D_{ki}b_k]\\
 &= \sum_{r,s} \sum_i[A_{ri} a_r, B_{si} a_s] + [A_{si}a_s, D_{ri} b_r] + [C_{ri}b_r, B_{si}a_s]+ [C_{ri} b_r, D_{si} b_s]\\
&= \sum_{r,s} \sum_i[A_{ri} a_r, B_{si} a_s] + [A_{si}a_s, D_{ri} b_r] + (-1)^n [B_{si}a_s, C_{ri}b_r]+ [C_{ri} b_r, D_{si} b_s]\\
&= \sum_{r,s} \sum_i[A_{ri} \overline{B_{si}} a_r,  a_s] + [A_{si} \overline{D_{ri}}a_s,  b_r] + (-1)^n [B_{si}\overline{C_{ri}}a_s, b_r]+ [C_{ri} \overline{D_{si}} b_r,  b_s]\\
&= \sum_{r,s} [(AB^\dagger)_{rs} a_r,  a_s] + [(A D^\dagger + (-1)^n B C^\dagger)_{sr} a_s,  b_r] + [(C D^\dagger)_{rs} b_r,  b_s].
\end{align*}

The middle terms being $\sum_{i=1}^g [a_i, b_i]$ is equivalent to (i), and the outer terms vanishing for $r \neq s$ is equivalent to (ii). For the first terms with $r=s$, we need
$$0 =  [(AB^\dagger)_{rr} a_r,  a_r].$$
Writing $(AB^\dagger)_{rr} = \sum_{a \in \Z} u_a t^a \in \Z[\pi]$, this is
\begin{align*}
0= [\sum_{a \in \Z} u_a t^a a_r,  a_r] &= \sum_{a < 0} u_a [t^a a_r,  a_r] + u_0 [a_r, a_r] + \sum_{a > 0} u_a [t^a a_r,  a_r]\\
 &= \sum_{a < 0} u_a [a_r, t^{-a} a_r] + u_0 [a_r, a_r] + \sum_{a > 0} u_a [t^a a_r,  a_r]\\
 &= \sum_{a < 0} u_a(-1)^n [t^{-a} a_r,  a_r] + u_0 [a_r, a_r] + \sum_{a> 0} u_a [t^a a_r,  a_r]\\
 &= u_0 [a_r, a_r] + \sum_{a>0} (u_a + (-1)^n u_{-a}) [t^a a_r, a_r].
\end{align*}
Now $[a_r, a_r]$ has order 1 if $n=1,3,7$, order 2 if $n$ is odd otherwise, and infinite order if $n$ is even (this follows e.g.\ from the metastable EHP sequence as in Section \ref{sec:EHP} and the Hopf Invariant 1 Theorem); for $a > 0$ the element $[t^a a_r, a_r]$ has infinite order. A direct check then shows that this identity holds if and only if $(AB^\dagger)_{rr} = \sum_{a \in \Z} u_a t^a \in \Lambda_n$.

The last terms with $r=s$ shows $(CD^\dagger)_{rr} \in \Lambda_n$ in the same way.
\end{proof}

Furthermore we have $(AB^\dagger + (-1)^n (AB^\dagger)^\dagger)_{rr} = (AB^\dagger)_{rr} + (-1)^n \overline{(AB^\dagger)_{rr}}$ and as $\Lambda_n \subseteq \Lambda_n^{\mathrm{max}} := \{a \in \Z[\pi] \, | \, a + (-1)^n \bar{a} = 0\}$ it follows from (iii) that the diagonal entries of $AB^\dagger + (-1)^n (AB^\dagger)^\dagger$ are also zero; similarly those of $CD^\dagger + (-1)^n (CD^\dagger)^\dagger$. Thus we may replace (ii) by the condition
\begin{enumerate}[(i')]
\setcounter{enumi}{1}
\item $AB^\dagger + (-1)^n (AB^\dagger)^\dagger$ and $CD^\dagger + (-1)^n (CD^\dagger)^\dagger$ vanish.
\end{enumerate}

We may express the conditions (i) and (ii') as
\begin{equation*}
\begin{bmatrix}
A & B \\
C & D
\end{bmatrix} \cdot \begin{bmatrix}
0 & I \\
(-1)^n I & 0
\end{bmatrix} \cdot \begin{bmatrix}
A & B \\
C & D
\end{bmatrix}^\dagger = \begin{bmatrix}
0 & I \\
(-1)^n I & 0
\end{bmatrix},
\end{equation*}
which is the statement that $\phi$ preserves the hermitian form $\lambda_X$. Then $\phi^{-1}$ is represented by the matrix $\begin{bmatrix}
A & B \\
C & D
\end{bmatrix}^{-1} = \begin{bmatrix}
D^\dagger & (-1)^n B^\dagger \\
(-1)^n C^\dagger & A^\dagger
\end{bmatrix}$, so using the quadratic property of $q_X$ we have 
\begin{align*}
q_X(\phi^{-1}(a_i)) &= q_X(\sum_j \overline{D_{ij}} a_j + (-1)^n \overline{C_{ij}} b_j)\\
 &= (-1)^n \sum_j \overline{D_{ij}} C_{ij} = (-1)^n (C D^\dagger )_{ii}
\end{align*}
and similarly $q_X(\phi^{-1}(b_i)) 
= (-1)^n(A B^\dagger)_{ii}$. As a quadratic refinement is (freely) determined by its values on a basis, it follows that condition (iii) is equivalent---given (i) and (ii')---to the condition that $\phi^{-1}$ preserves the quadratic refinement $q_X$ of $\lambda_X$, which is turn is equivalent to $\phi$ preserving $q_X$. Thus this discussion shows that the action on $\pi_n(X_g)$ gives a surjective homomorphism
$$\pi_0( \hAut_\partial(X_{g})) \to U_g(\Z[\pi], \Lambda_n).$$

Similarly, any homotopy equivalence of $W_{g,1}$ fixing its boundary acts on $\pi_n(W_{g,1}) \cong H_n(W_{g,1};\Z)$ with its $(-1)^n$-symmetric intersection form $\lambda_W$, and also preserves the quadratic form $q_W$ (either by a parallel analysis to the above using Theorem \ref{thm:HtyAutWg}, or by \cite[Section 2.7]{BerglundMadsen1}). This gives a surjective homomorphism $\pi_0( \hAut_\partial(W_{g,1})) \to U_g(\Z, \Lambda_n)$.

\begin{thm}\label{thm:hAutComparison}
The squares
\beq
\xymatrix{
\pi_0 (\hAut_\partial(W_{g,1})) \ar[d]\ar[r] &\pi_0( \hAut_\partial(X_g))\ar[d] & \pi_0 (\hAut_\partial(X_{g})) \ar[d]\ar[r] &\pi_0( \hAut_\partial(W_{g,1}))\ar[d]\\
U_g(\Z,  \Lambda_n) \ar[r] & U_g(\Z[\pi],  \Lambda_n) & U_g(\Z[\pi], \Lambda_n) \ar[r] & U_g(\Z, \Lambda_n),
}
\eeq
induced by the codimension zero embeddings $W_{g,1} \subset X_g \subset W_{g,1}$, are cartesian. The vertical maps are surjective, and their (common) kernels are the finite abelian groups $H  \otimes \Sigma\pi_{2n}(S^n)$, where $\Sigma\pi_{2n}(S^n) := \mathrm{Im}(\Sigma : \pi_{2n}(S^n) \to \pi_{2n+1}(S^{n+1}))$.
\end{thm}
\begin{proof}
Consider the analogue of \eqref{eq:XFibn} for $W_{g,1}$, which is the fibration sequence
\beq
\Map_\partial(W_{g,1}, W_{g,1})\to\Map_\ast(W_{g,1}, W_{g,1})\to\Map_\ast(S^{2n-1}, W_{g,1}).
\eeq
This has a map to \eqref{eq:XFibn}, given in the evident way on the first two terms, and on the last given by sending a map $f : S^{2n-1} \to W_{g,1}$ to $S^1 \times S^{2n-2} \to S^1 \times S^{2n-2} \vee S^{2n-1} \xto{\iota \vee f} X_g$. The map of long exact sequences has a portion
\beq
\xymatrix{
\Hom_{\Z}(\pi_n(W_{g,1}), \pi_{n+1}(W_{g,1})) \ar[r] \ar[d]&\pi_{2n}(W_{g,1}) \ar[d]\ar[r] &\pi_0( \hAut_\partial(W_{g,1}))\ar[d] \ar[r] & GL_{2g}(\Z) \ar[d]\\
\Hom_{\Z[\pi]}(\pi_n(X_g), \pi_{n+1}(X_g)) \ar[r]^-{\pi_1(\rho)}& [\pi_{2n}(X_g)]_\pi \ar[r] & \pi_0( \hAut_\partial(X_g)) \ar[r] & GL_{2g}(\Z[\pi]),
}
\eeq
and we must show that the map between kernels of the two rightmost horizontal maps is an isomorphism; this is the same as the map between the cokernels of the two leftmost horizontal maps.

By Theorem \ref{thm:HtyAutWg}, the top leftmost map sends a $\phi \in \Hom_{\Z}(\pi_n(W_{g,1}), \pi_{n+1}(W_{g,1}))$ to $\sum_{i=1}^g [a_i, \phi(b_i)] + (-1)^{n k}[\phi(a_i), b_i]$, and so, from Appendix \ref{sec:Hilton-Milnor}, its image is spanned by the Whitehead products $[x_i, x_j \circ \eta]$. For $i \neq j$ these generate the second summand in \eqref{eq:Hilton-W_g}, and for $j=i$ we have $[x_i, x_i \circ \eta]$ in the first term. Now $[\iota_n, \iota_n \circ \eta] \in \pi_{2n}(S^n)$ generates the kernel of $\Sigma : \pi_{2n}(S^n) \to \pi_{2n+1}(S^{n+1})$, by \cite[Corollary XII.2.6]{WhiteheadBook}. Thus the cokernel of the top leftmost map is identified with $H \otimes \Sigma \pi_{2n}(S^n)$. (See \cite[Theorem 8.14]{Baues} for another derivation of this.)

Similarly, taking $\pi$-coinvariants in equation \eqref{eq:Hilton-X_g} gives
$$[\pi_{2n}(X_g)]_\pi \cong \bigoplus_{1 \leq i \leq 2g} \pi_{2n}(S^n)\{x_i\} \oplus \bigoplus_{\substack{1 \leq i , j \leq 2g\\a\in\Z\\(a,i) < (0,j)}}\pi_{2n}(S^{2n-1})\{t^a x_i \otimes x_j\}.$$
The description of $\pi_1(\rho)$ in Theorem \ref{thm:HtyAut} shows that the image of $\pi_1(\rho)$ is spanned by the (equivalence classes of the) Whitehead products $[t^a x_i,  t^b x_j \circ \eta]$. Using  \eqref{eq:WhiteheadCompFormula} from Appendix \ref{sec:WhiteheadComposition} and the fact that $\eta : S^{2n} \to S^{2n-1}$ is an $n$-fold suspension, we can write these as $[t^a x_i,  t^b x_j]\circ \eta$. As above for $j \neq i$ these span the second summand, and for $i=j$ they give $[\iota, \iota\circ\eta]\otimes x_i$, so the cokernel of $\pi_1(\rho)$ is again $H  \otimes \Sigma\pi_{2n}(S^n)$. Under these identifcations the map between these cokernels can then be seen to be the identity map. Finally, $\Sigma\pi_{2n}(S^n)$ is a finite group, by a theorem of Serre.
\end{proof}

\begin{rem}
In particular the extension
$$0 \to H \otimes \Sigma \pi_{2n}(S^n) \to \pi_0(\hAut_\partial(X_{g})) \to U_g(\Z[\pi], \Lambda_n) \to 0$$
is pulled back from the analogous extension for $\pi_0(\hAut_\partial(W_{g,1}))$. For $n$ either odd or 2 or 6, the latter extension has been shown to be split by Baues \cite[Theorem 8.14]{Baues} (see also \cite[Corollary F]{KrannichMCG}), observing that $\pi_0( \hAut_\partial(W_{g,1})) \to \pi_0 (\hAut^+(W_{g}))$ is an isomorphism.
\end{rem}

\subsection{Higher homotopy groups}\label{sec:hAutHigherHty}

We also wish to compute $\pi_{k-1}(\hAut_{\partial}(X_g),id)$ in degrees $k \leq n-1$, where the homotopy theory involved is metastable. In order to express the answer as a representation of $\pi_0(\hAut_{\partial}(X_g))$, we define
\begin{align*}
S^+_X &:= \langle x \otimes x \, | \, x \in \pi_n(X_g) \rangle_\Z \subset \pi_n(X_g)^{\otimes_{\Z[\pi]} 2},\\
S^-_X &:= \langle x \otimes y - y \otimes x \, | \, x, y \in \pi_n(X_g) \rangle_\Z \subset \pi_n(X_g)^{\otimes_{\Z[\pi]} 2}.
\end{align*}
These are subrepresentations as they are the invariants of the involutions $x \otimes y \mapsto y \otimes x$ and $x\otimes y \mapsto - y \otimes x$ respectively, which are morphisms of representations\footnote{But are not $\Z[\pi]$-linear: $S_X^+$ and $S_X^-$ are not $\Z[\pi]$-modules.}.
\begin{prop}\label{prop:pi_khAut}
For all $2 \leq k < n-1$, there are extensions of $\pi_0(\hAut_\partial(X_g))$-modules
\beq
0\to \frac{\pi_{2n-1+k}(S^n)}{[\iota_n, \pi_{n+k}(S^n)]} \otimes H \to\pi_{k-1}(\hAut_{\partial}(X_g),id)\to Q_X \to 0,
\eeq
and
$$0 \to K \otimes S^+_X \to Q_X \to \frac{\pi_{n+k-1}(S^n)}{K} \otimes S^-_X \to 0,$$
where $K := \mathrm{Ker}([\iota_n, -] : \pi_{n+k-1}(S^n) \to \pi_{2n-1+k-1}(S^n))$.
\end{prop}
 
\begin{proof}
By the description of $\pi_{m}(\rho)$ in Theorem \ref{thm:HtyAut}, we need to understand the kernel and cokernel of the map
\begin{align*}
\pi_m(\rho) :\pi_n(X_g) \otimes_{\Z[\pi]} \pi_{n+m}(X_g) &\to [\pi_{2n+m-1}(X_g)]_\pi\\
x \otimes y &\mapsto [y,x]
\end{align*}
for $m=k-1$ and $m=k$ respectively. The homotopy groups of $X_g$ may be described by the Hilton--Milnor theorem, and in the metastable range they are given by \eqref{eq:Hilton-X_g} in Section \ref{sec:Hilton-Milnor}. Thus we can write the domain as
\begin{equation}\label{eq:Source}
\bigoplus_{\substack{1 \leq i\leq 2g}}\pi_n(S^n)\{x_i\} \otimes \pi_{n+m}(S^n)\{x_i\} \oplus\bigoplus_{\substack{1 \leq i , j \leq 2g\\a \in \Z\\(a,i) \neq (0,j)}} \pi_n(S^n)\{t^a x_i\} \otimes \pi_{n+m}(S^n)\{x_j\}.
\end{equation}
On the other hand, taking $\pi$-coinvariants in \eqref{eq:Hilton-X_g} gives
\begin{equation}\label{eq:Target}
[\pi_{2n+m-1}(X_g)]_\pi \cong \bigoplus_{\substack{1 \leq i \leq 2g}} \pi_{2n+m-1}(S^n)\{x_i\} \oplus  \bigoplus_{\substack{1 \leq i , j \leq 2g\\a\in\Z\\(a,i) < (0,j)}} \pi_{2n+m-1}(S^{2n-1})\{t^a x_i \otimes x_j\}.
\end{equation}
We will use the following notational convention to refer to elements of these groups: we denote the element $[f] \in \pi_l(S^n)$ considered as the summand $\pi_l(S^n)\{t^a x_i\}$ by $[t^a x_i \circ f]$, because the composition $S^l \xto{f} S^n \xto{t^a x_i} X_g$ indeed represents the corresponding element of $\pi_l(X_g)$.

The map $\pi_m(\rho)$ sends the terms of \eqref{eq:Source} to the corresponding terms of \eqref{eq:Target}. On the first terms of \eqref{eq:Source} it sends $x_i \otimes x_i \circ f$, for $f\in\pi_{2n+m-1}(S^n)$, to 
\begin{align*}
[x_i \circ f, x_i] &= (-1)^{n(n+m)}[x_i, x_i \circ f],\\
&= (-1)^{n(n+m)}([ x_i , x_i]\circ \Sigma^{n-1} f ),
\end{align*}
lying in the first terms of \eqref{eq:Target}. On the second terms of \eqref{eq:Source}, it sends $t^a x_i \otimes x_j \circ g$, for $g\in\pi_{2n+m-1}(S^{2n-1})$, to 
\begin{align*}
[x_j \circ g, t^a x_i] &= (-1)^{n(n+m)}[t^a x_i, x_j \circ g]\\
 &= \begin{cases}
 (-1)^{n(n+m)}([t^a x_i , x_j]\circ \Sigma^{n-1} g) & \text{if $(a,i) < (0,j)$}\\
 (-1)^{nm} ([t^{-a} x_j , x_i]\circ \Sigma^{n-1} g) & \text{if $(a,i) > (0,j)$},
\end{cases}
\end{align*}
lying in the second terms of \eqref{eq:Target}, where in the second case we have used $[t^a x_i , x_j] = t^a[ x_i , t^{-a}x_j] = (-1)^n t^a [t^{-a} x_j, x_i]$. Here we have applied \eqref{eq:WhiteheadCompFormula} from Appendix \ref{sec:metastable}, using that $g \in \pi_{n+m}(S^n)$ is a suspension as $m \leq k < n-1$.

The cokernel of $\pi_k(\rho)$ is therefore
$$\bigoplus_{1 \leq i \leq 2g} \frac{\pi_{2n-1+k}(S^n)}{[\iota_n, \pi_{n+k}(S^n)]}\{x_i\} \oplus \bigoplus_{\substack{1 \leq i , j \leq 2g\\a\in\Z\\(a,i) < (0,j)}} \frac{\pi_{2n+k-1}(S^{2n-1})}{\Sigma^{n-1} \pi_{n+k}(S^n)}\{t^a x_i \otimes x_j\},$$
and the second term vanishes as $\Sigma^{n-1} : \pi_{n+k}(S^n) \to \pi_{2n+k-1}(S^{2n-1})$ is an epimorphism by our assumption that $k < n-1$. One easily verifies that as a representation of $\pi_0(\hAut_\partial(X_g))$ this is $\frac{\pi_{2n-1+k}(S^n)}{[\iota_n, \pi_{n+k}(S^n)]} \otimes H$.

To understand the kernel of $\pi_{k-1}(\rho)$, first observe  the kernel on the first terms is 
$$\bigoplus_{\substack{1 \leq i \leq 2g}} K \{ x_i\otimes x_i\},$$
for $K := \mathrm{Ker}([\iota_n, -] : \pi_{n+k-1}(S^n) \to \pi_{2n-1+k-1}(S^n))$. For the second terms, observe that for each triple $(i,j,a)$ the map
\begin{align*}
\pi_n(S^n)\{t^a x_i\} \otimes \pi_{n+k-1}(S^n)\{x_j\} &\to \pi_{2n-1+k-1}(S^{2n-1})\{t^a x_i \otimes x_j\}\\
t^a x_i \otimes x_j \circ g & \mapsto \begin{cases}
 (-1)^{n(n+k-1)}([t^a x_i , x_j]\circ \Sigma^{n-1} g) & \text{if $(a,i) < (0,j)$}\\
 (-1)^{n(k-1)}([t^{-a} x_j , x_i]\circ \Sigma^{n-1} g) & \text{if $(a,i) > (0,j)$},
\end{cases}
\end{align*}
is an isomorphism, by our assumption that $k < n-1$. Thus on the second terms the kernel is the span of the elements
$$t^a x_i \otimes x_j \circ g -(-1)^n t^{-a} x_j \otimes x_i \circ g \in \bigoplus_{\substack{1 \leq i < j \leq 2g\\a \in \Z}} \pi_n(S^n)\{t^a x_i\} \otimes \pi_{n+k-1}(S^n)\{x_j\}.$$
Now if $n$ is even then $K$ consists of elements of order 2 by Lemma \ref{lem:Ord2} so, whatever the parity of $n$, for $g \in K$ we have
$$t^a x_i \otimes x_j \circ g -(-1)^n t^{-a} x_j \otimes x_i \circ g = t^a x_i \otimes x_j \circ g + t^{-a} x_j \otimes x_i \circ g.$$
Thus the terms
$$\bigoplus_{\substack{1 \leq i \leq 2g}} K \{x_i\otimes x_i\} \oplus \bigoplus_{\substack{1 \leq i, j \leq 2g\\a \in \Z\\(a,i) < (0,j)}} K\{t^a x_i \otimes x_j  + t^{-a} x_j \otimes x_i \}$$
form a subrepresentation of the kernel, which is isomorphic to $K \otimes S^+_X$. Quotienting out by this subrepresentation leaves the terms
$$\bigoplus_{\substack{1 \leq i< j \leq 2g\\a \in \Z}} \frac{\pi_{n+k-1}(S^n)}{K}\{t^a x_i \otimes x_j -(-1)^n t^{-a} x_j \otimes x_i\}.$$
If $n$ is odd then $\frac{\pi_{n+k-1}(S^n)}{K}$ consists of elements of order 2 by Lemma \ref{lem:Ord2} so, whatever the parity of $n$, for $g \in \frac{\pi_{n+k-1}(S^n)}{K}$ we have
$$t^a x_i \otimes x_j \circ g -(-1)^n t^{-a} x_j \otimes x_i \circ g = t^a x_i \otimes x_j \circ g - t^{-a} x_j \otimes x_i \circ g.$$
Thus we identify this representation as $\frac{\pi_{n+k-1}(S^n)}{K} \otimes S^-_X$.
\end{proof}
An analogous but much simpler argument using Theorem \ref{thm:HtyAutWg} establishes the corresponding result for the homotopy groups of $\hAut(W_{g,1})$. Define 
\begin{align*}
S^+_W &:= \langle x \otimes x \, | \, x \in \pi_n(W_{g,1}) \rangle_\Z \subset \pi_n(W_{g,1})^{\otimes 2},\\
S^-_W &:= \langle x \otimes y - y \otimes x \, | \, x, y \in \pi_n(W_{g,1}) \rangle_\Z \subset \pi_n(W_{g,1})^{\otimes 2}.
\end{align*}

\begin{prop}\label{cor:pi_khAutWg}
For all $2 \leq k < n-1$, there are extensions of $\pi_0(\hAut_\partial(W_{g,1}))$-modules
\beq
0\to \frac{\pi_{2n-1+k}(S^n)}{[\iota_n, \pi_{n+k}(S^n)]} \otimes H \to\pi_{k-1}(\hAut_{\partial}(W_{g,1}),id)\to Q_W \to 0,
\eeq
and
$$0 \to K \otimes S^+_W \to Q_W \to \frac{\pi_{n+k-1}(S^n)}{K} \otimes S^-_W \to 0,$$
where $K = \mathrm{Ker}([\iota_n, -] : \pi_{n+k-1}(S^n) \to \pi_{2n-1+k-1}(S^n))$.\qed
\end{prop}

\begin{example}\label{ex:oddprimes}
Work $p$-locally at an odd prime $p$. By the discussion in Section \ref{sec:QuadApprox} there is a short exact sequence
$$0 \to \begin{cases}
0 & \text{ if $n$ is odd}\\
\pi_{k}^s & \text{ if $n$ is even}\end{cases} \to \pi_{2n+k-1}(S^n) \to \pi_{n+k-1}^s \to 0,$$
where the right-hand map is stabilisation, and the left-hand map is 
$\pi_k^s \cong \pi_{k+n}(S^n) \xto{[\iota_n, -]} \pi_{2n+k-1}(S^n)$. Thus for $n$ both odd or even we have $\frac{\pi_{2n-1+k}(S^n)}{[\iota_n, \pi_{n+k}(S^n)]} \cong \pi_{n+k-1}^s$, and we see, in agreement with Lemma \ref{lem:Ord2}, that localised at an odd prime
$$K = \mathrm{Ker}([\iota_n, -] : \pi_{n+k-1}(S^n) \to \pi_{2n-1+k-1}(S^n)) = \begin{cases}
\pi_{k-1}^s & \text{ if $n$ is odd}\\
0 & \text{ if $n$ is even}.
\end{cases}$$

Therefore localised at an odd prime we have short exact sequences
\beq
0\to \pi_{n+k-1}^s \otimes H \to\pi_{k-1}(\hAut_{\partial}(X_g),id)\to \begin{cases}
\pi_{k-1}^s \otimes S^+_X & \text{ if $n$ is odd}\\
\pi_{k-1}^s \otimes S^-_X & \text{ if $n$ is even}.
\end{cases}\to 0.
\eeq
\end{example}

\begin{example}\label{ex:pi_1hAut}
To calculate $\pi_{1}(\hAut_{\partial}(X_g),id)$, we need $K = \mathrm{Ker}([\iota_n, -] : \pi_{n+1}(S^n) \to \pi_{2n}(S^n))$ and $\mathrm{Coker}([\iota_n, -] : \pi_{n+2}(S^n) \to \pi_{2n+1}(S^n))$. The second part of Lemma \ref{lem:Hopf} identifies this cokernel with $\Sigma \pi_{2n+1}(S^n)$, and the first part of that lemma shows that $K$ is $\Z/2$ if $n \equiv 3 \mod 4$ or $n=6$, and is zero otherwise. In the first case we have $Q_X \cong \Z/2 \otimes S^+_X$, and in the second we have $Q_X \cong \Z/2 \otimes S^-_X$. Thus there is an extension
$$0 \to \Sigma \pi_{2n+1}(S^n) \otimes H \to \pi_{1}(\hAut_{\partial}(X_g),id) \to \begin{cases}
\Z/2 \otimes S^+_X & \text{ if $n$ is 6 or is $3\!\!\!\! \mod 4$}\\
\Z/2 \otimes S^-_X & \text{ otherwise }
\end{cases} \to 0.$$
\end{example}

\section{The Weiss sequence, disjunction, and surgery}\label{sec:weiss+disjunction}
The following space of embeddings is central to our arguments.
\begin{deftn}\label{def:emb}
Let $\Emb^{\cong}_{\partial/2}(X_g)$ be the topological monoid of embeddings $X_g\to X_g$ (with the $C^{\infty}$-topology), which fix a neighbourhood of $S^1\times D^{2n-2}_-=1/2\partial X_g\subset\partial X_g$ pointwise, and are isotopic (through embeddings with the same boundary condition) to a diffeomorphism of $X_g$ which restricts to the identity on a neighbourhood of $\partial X_g$.
\end{deftn}
From now on, we will freely use the notions of ``block'' embeddings and automorphisms as defined in \cite{BLR}.

\begin{deftn}
Let $\blEmb_{\partial/2}(X_g)$ be the semi-simplicial monoid of block self-embeddings of $X_g$ fixing a neighbourhood of $1/2\partial X_g$, and $\blEmb_{\partial/2}^{\cong}(X_g)$ be the maximal sub-semi-simplicial monoid containing the $0$-simplices which belong to $\Emb_{\partial/2}^{\cong}(X_g)$.
\end{deftn}

In this section we will relate the homotopy groups of $B\Emb^{\cong}_{\partial/2}(X_g)$ to those of $B\hAut^{\cong}_{\partial}(X_g)$. Our strategy will be based on the map of fibration sequences
\begin{equation}\label{eq:WeissMap}
\begin{gathered}
\xymatrix{
B\Diff_{\partial}(S^1\times D^{2n-1}) \ar[r] \ar[d] & B\Diff_{\partial}(X_g)\ar[r] \ar[d] & B\Emb_{\partial/2}^{\cong}(X_g) \ar[d]\\
B\blDiff_{\partial}(S^1\times D^{2n-1}) \ar[r] & B\blDiff_{\partial}(X_g)\ar[r] & B\blEmb_{\partial/2}^{\cong}(X_g),
}
\end{gathered}
\end{equation}
which we now explain.

\subsection{The Weiss fibre sequence}\label{sec:WeissFib}

A fibration sequence analogous to the top row of \eqref{eq:WeissMap}, but with leftmost term $B\Diff_\partial(D^d)$ instead, appears first in work of Weiss \cite[Remark 2.1.3]{weiss-dalian}, and has been developed in detail by Kupers \cite[Section 4]{kupers-finiteness}, who showed that this sequence may be delooped. Their arguments may be straightforwardly generalised to the following situation: $M$ is a compact $d$-manifold with boundary, and $N \subset \partial M$ is a codimension 0 submanifold, then the delooped fibration sequence then takes the form
\beq
B\Diff_{\partial}(M)\to B\Emb^{\cong}_{\partial M\setminus{\rm{int }}(N)}(M)\to B^2\Diff_{\partial}(N\times I),
\eeq
where the subscript $\partial M\setminus\mathrm{int }(N)$ in the middle term refers to those self-embedding which fix a neighbourhood of $\partial M\setminus\mathrm{int }(N)$ pointwise.
There is an analogous fibration sequence for block diffeomorphisms and block embeddings, or for (block) homeomorphisms and topological (block) embeddings. We will be interested in the case when $M=X_g$ and $N=S^1\times D^{2n-2}_+\subset\partial X_g$, which gives the map of fibration sequences displayed in \eqref{eq:WeissMap}. 

It will be convenient for calculations to work not just with diffeomorphisms and embeddings but with those preserving a stable framing. Recall that a tangential structure is a fibration $\theta : B \to B\On(2n)$, and a $\theta$-structure on a $2n$-manifold $M$ is a fibrewise lienar isomorphism $\ell:TM\to\theta^*\gamma$, where $\gamma$ is the canonical $2n$-plane bundle over $B\On(2n)$. We define a tangential structure $\mathrm{sfr}_{2n}$ as the pullback
\begin{equation}\label{eq:sfr}
\begin{gathered}
\xymatrix{
B \ar[r] \ar[d]_{\mathrm{sfr}_{2n}} & E\On \ar[d]\\
B\On(2n) \ar[r] & B\On,
}
\end{gathered}
\end{equation}
so that $B \simeq \SOn/\SOn(2n)$. Choose a boundary condition $\ell_{\partial X_g} : TX_g \vert_{\partial X_g} \to \mathrm{sfr}_{2n}^*\gamma$, and its restriction $\ell_{1/2\partial X_g} : TX_g \vert_{1/2\partial X_g} \to \mathrm{sfr}_{2n}^*\gamma$. Observe that $\Diff_{\partial}(X_g)$ and $\Emb^{\cong}_{\partial/2}(X_g)$ act, via the derivative, on the spaces $\Bun_\partial(TX_g, \mathrm{sfr}_{2n}^*\gamma; \ell_{\partial X_g})$ and $\Bun_{\partial/2}(TX_g, \mathrm{sfr}_{2n}^*\gamma; \ell_{{1/2\partial X_g}})$ of bundle maps extending the appropriate boundary condition. We define
\begin{align*}
B\Diff^\mathrm{sfr}_\partial(X_g ; \ell_{\partial X_g}) & := \Bun_\partial(TX_g, \mathrm{sfr}_{2n}^*\gamma; \ell_{\partial X_g}) \sslash \Diff_\partial(X_g)\\
B\Emb^{\cong, \mathrm{sfr}}_{\partial/2 }(X_g ; \ell_{\partial/2 X_g}) & := \Bun_{\partial/2}(TX_g, \mathrm{sfr}_{2n}^*\gamma; \ell_{{1/2\partial X_g}}) \sslash \Emb^{\cong}_{\partial/2}(X_g).
\end{align*}
For a stable framing $\ell$ of $X_g$ we write $B\Diff_{\partial}^{\sfr}(X_g;\ell_{\partial X_g})_\ell$ and $B\Emb_{\partial/2}^{\cong,\sfr}(X_g;\ell_{\partial/2})_\ell$ for the path-components of $\ell$.

As in \cite[Proposition 8.8]{KR-WAlg}, there is a delooped stably framed Weiss fibre sequence
\beq
B\Diff_{\partial}^{\sfr}(X_g;\ell_{\partial X_g})_\ell\to B\Emb_{\partial/2}^{\cong,\sfr}(X_g;\ell_{\partial/2})_\ell \to B\left(B\Diff_{\partial}^{\sfr}(S^1\times D^{2n-1};\ell_{\partial})_L\right),
\eeq
where $\ell_{\partial}:T(S^1\times D^{2n-1})|_{\partial(S^1\times D^{2n-1})}\to\sfr_{2n}^*\gamma$ is given by the restriction of the standard stable framing of $\R^{2n}$ to $\partial(S^1\times D^{2n-1})$, 
\beq
B\Diff^\mathrm{sfr}_\partial(S^1\times D^{2n-1} ; \ell_{\partial}):= \Bun_\partial(T(S^1\times D^{2n-1}), \mathrm{sfr}_{2n}^*\gamma; \ell_{\partial}) \sslash \Diff_\partial(S^1\times D^{2n-1}),
\eeq
and $L \subset \pi_0(B\Diff^\mathrm{sfr}_\partial(S^1\times D^{2n-1} ; \ell_{\partial}))$ is the ``inertia" subgroup: those stably-framed manifolds $(S^1 \times D^{2n-1}, \ell')$ which when glued to $(X_g, \ell)$ give a stably framed manifold equivalent to $(X_g, \ell)$, up to diffeomorphism and homotopy of stable framings.

As the tangential structure $\mathrm{sfr}_{2n}$ is defined using the stable tangent bundle, it also makes sense for block diffeomorphisms and block embeddings, giving analogous spaces $B\blDiff^\mathrm{sfr}_\partial(X_g ; \ell_{\partial X_g})$ and $B\blEmb^{\mathrm{sfr}, \cong}_{\partial/2 }(X_g ; \ell_{\partial/2 X_g})$. See e.g.\ Sections 1.7-1.9 of \cite{krannich-concordances} for a detailed discussion of how this may be implemented.

\subsection{Disjunction and surgery}

The main simplification we make in this paper is to replace the study of $B\Emb_{\partial/2}^{\cong}(X_g)$ by its block analogue $B\blEmb_{\partial/2}^{\cong}(X_g)$, which is valid in a range of degrees by the following.

\begin{prop}\label{prop:EMCG up to concordance}
Let $n\geq 3$. Then the map $B\Emb_{\partial/2}^{\cong}(X_g)\to B\blEmb_{\partial/2}^{\cong}(X_g)$ is $(n-1)$-connected.
\end{prop}
\begin{proof}
By taking vertical homotopy fibres in \eqref{eq:WeissMap} the homotopy fibre $F$ of this map fits into a fibration sequence
$$\frac{\blDiff_{\partial}(S^1\times D^{2n-1})}{\Diff_{\partial}(S^1\times D^{2n-1})} \to \frac{\blDiff_{\partial}(X_g)}{\Diff_{\partial}(X_g)} \to F.$$
By Morlet's lemma of disjunction \cite[Theorem 3.1]{BLR}, the space $F$ is $(n-2)$-connected, so the map in question is $(n-1)$-connected.
\end{proof}
\begin{rem}\label{rem:morlet-Wg}
By mapping the Weiss fibre sequence
\beq
B\Diff_{\partial}(D^{2n})\to B\Diff_{\partial}(W_{g,1})\to B\Emb_{\partial/2}^{\cong}(W_{g,1})
\eeq
to its block version, the same argument shows that the map $B\Emb_{\partial/2}^{\cong}(W_{g,1})\to B\blEmb_{\partial/2}^{\cong}(W_{g,1})$ is $(2n-3)$-connected. 
\end{rem}

We now define maps
\begin{equation}\label{eq:blEmb to hAut}
B\blEmb_{\partial/2}^{\cong}(X_g)\to B\hAut_{\partial}^{\cong}(X_g)\ \ \ \text{and}\ \ \ B\blEmb_{\partial/2}^{\cong}(W_{g,1})\to B\hAut_{\partial}^{\cong}(W_{g,1})
\end{equation}  
as follows: forgetting the smoothness gives a map between Weiss' fibre sequences
\begin{equation*}
\xymatrix{
B\blDiff_{\partial}(S^1\times D^{2n-1})\ar[d] \ar[r]& B\blDiff_{\partial}(X_g)\ar[r]\ar[d] & B\blEmb_{\partial/2}^{\cong}(X_g)\ar[d] \\
B\blTop_{\partial}(S^1\times D^{2n-1}) \ar[r] & B\blTop_{\partial}(X_g)\ar[r] & B\blEmb_{\partial/2}^{\cong,{\rm{TOP}}}(X_g)
}
\end{equation*}
where the decoration $\mathrm{TOP}$ on the right-hand term of the lower fibration sequence refers to locally flat topological embeddings. By \cite[Corollary 2.3]{burghelea-blockTop} there is a decomposition
$$\blTop_{\partial}(S^1\times D^{2n-1}) \simeq \blTop_{\partial}(D^{2n-1}) \times \Omega \blTop_{\partial}(D^{2n-1})$$
which is therefore contractible by the Alexander trick. Thus the map $B\blTop_{\partial}(X_g) \to B\blEmb_{\partial/2}^{\cong,{\rm{TOP}}}(X_g)$ is an equivalence and so we have a canonical (up to homotopy) lift $B\blEmb_{\partial/2}^{\cong}(X_g)\to B\blTop_{\partial}(X_g)$ which we compose with $B\blTop_{\partial}(X_g)\to B\hAut_{\partial}(X_g)$. The resulting map has image in $B\hAut^{\cong}_{\partial}(X_g)$, giving the first map in \eqref{eq:blEmb to hAut}. The second map in \eqref{eq:blEmb to hAut} is obtained similarly, using that $\blTop_{\partial}(D^{2n})$ is contractible.

\begin{prop}\label{prop:BEmb to BhAut}
For all $n\geq 3$, the squares
\beq
\xymatrix{
B\blEmb_{\partial/2}^{\cong}(W_{g,1})\ar[r]\ar[d] & B\blEmb_{\partial/2}^{\cong}(X_{g})\ar[d] &  B\blEmb_{\partial/2}^{\cong}(X_{g})\ar[r]\ar[d] & B\blEmb_{\partial/2}^{\cong}(W_{g,1})\ar[d]\\
B\hAut^{\cong}_{\partial}(W_{g,1})\ar[r] & B\hAut^{\cong}_{\partial}(X_g) & B\hAut^{\cong}_{\partial}(X_{g})\ar[r] & B\hAut^{\cong}_{\partial}(W_{g,1})
}
\eeq
given by the codimension zero embeddings $W_{g,1} \subset X_g \subset W_{g,1}$, are homotopy cartesian. 
\end{prop}
\begin{proof}
It suffices to consider the first square, as the square obtained by gluing them is certainly homotopy cartesian (and all the spaces are connected).

Consider the diagrams
\beq
\xymatrix{
B\blDiff_{\partial}(D^{2n})\ar[r]\ar[d] &
 B\blDiff_{\partial}(W_{g,1})\ar[r]\ar[d]  &
 B\blEmb_{\partial/2}^{\cong}(W_{g,1})\ar[d] \\
B\hAut_{\partial}^{\cong}(D^{2n})\ar[r] &
 B\hAut_{\partial}^{\cong}(W_{g,1})\ar[r]  &
 B\hAut_{\partial}^{\cong}(W_{g,1})
}
\eeq 
and
\beq
\xymatrix{
B\blDiff_{\partial}(S^1\times D^{2n-1})\ar[r]\ar[d] &
 B\blDiff_{\partial}(X_g)\ar[r]\ar[d]  &
 B\blEmb_{\partial/2}^{\cong}(X_g)\ar[d] \\
B\hAut_{\partial}^{\cong}(S^1\times D^{2n-1})\ar[r] &
 B\hAut_{\partial}^{\cong}(X_g)\ar[r]  &
 B\hAut_{\partial}^{\cong}(X_g)
}
\eeq 
given by maps of homotopy fibre sequences, where the first diagram maps to the second. We wish to show that the square given by the right-hand vertical maps is homotopy cartesian: as the homotopy fibres of these maps are connected by definition, as we only take those components of $\hAut_{\partial}(X_g)$ or $\hAut_{\partial}(W_{g,1})$ represented by diffeomorphisms, it is equivalent to show that the cube
\beq
\xymatrix{
B\blDiff_{\partial}(D^{2n})\ar[rr]\ar[dd] \ar[rd] & &
 B\blDiff_{\partial}(W_{g,1}) \ar[dd]|\hole  \ar[rd] \\
 & B\blDiff_{\partial}(S^1 \times D^{2n-1}) \ar[rr] \ar[dd] & & B\blDiff_{\partial}(X_{g}) \ar[dd]\\
B\hAut_{\partial}^{\cong}(D^{2n})\ar[rr]|\hole \ar[rd] & &
 B\hAut_{\partial}^{\cong}(W_{g,1}) \ar[rd]\\
& B\hAut_{\partial}^{\cong}(S^1 \times D^{2n-1})\ar[rr] & &
 B\hAut_{\partial}^{\cong}(X_{g})
}
\eeq
is homotopy cartesian.

The $s$-cobordism theorem identifies $\frac{\hAut_{\partial}^{\cong}(X_g)}{\blDiff_{\partial}(X_g)}$ with the component of the identity map in the block simple structure space $\widetilde{\mathcal{S}}^s(X_g)$, and Quinn's geometric formulation of surgery theory \cite{Quinn, Nicas} places this space into a homotopy fibre sequence
$$\widetilde{\mathcal{S}}^s(X_g) \xto{\nu} \Map_{\partial}(X_g, \mathrm{G}/\On) \xto{\sigma} \Omega^{\infty+2n}L^s(\Z[\pi]).$$
Similarly, $\frac{\hAut_{\partial}^{\cong}(S^1\times D^{2n-1})}{\blDiff_{\partial}(S^1\times D^{2n-1})}$ is the identity component of $\widetilde{\mathcal{S}}^s(S^1 \times D^{2n-1})$, fitting in a homotopy fibre sequence
$$\widetilde{\mathcal{S}}^s(S^1 \times D^{2n-1}) \xto{\nu} \Map_{\partial}(S^1 \times D^{2n-1}, \mathrm{G}/\On) \xto{\sigma} \Omega^{\infty+2n} L^s(\Z[\pi]),$$
which maps to the sequence above by naturality. The map on $L$-theory terms is an equivalence, giving a homotopy cartesian square
\beq
\xymatrix{
\frac{\hAut_{\partial}^{\cong}(S^1\times D^{2n-1})}{\blDiff_{\partial}(S^1\times D^{2n-1})} \ar[r]\ar[d] &
\frac{\hAut_{\partial}^{\cong}(X_g)}{\blDiff_{\partial}(X_g)} \ar[d]\\
\Map_{\partial}(S^1\times D^{2n-1}, \mathrm{G}/\On)_0 \ar[r]  &  
\Map_{\partial}(X_g, \mathrm{G}/\On)_0
}
\eeq
Similarly, there is a homotopy cartesian square
\beq
\xymatrix{
\frac{\hAut_{\partial}^{\cong}(D^{2n})}{\blDiff_{\partial}(D^{2n})} \ar[r]\ar[d] &
\frac{\hAut_{\partial}^{\cong}(W_{g,1})}{\blDiff_{\partial}(W_{g,1})} \ar[d]\\
\Map_{\partial}(D^{2n}, \mathrm{G}/\On)_0 \ar[r]  &  
\Map_{\partial}(W_{g,1}, \mathrm{G}/\On)_0.
}
\eeq
The argument is completed by observing that
\beq
\xymatrix{
\Map_{\partial}(D^{2n}, \mathrm{G}/\On)_0 \ar[r] \ar[d]  &  
\Map_{\partial}(W_{g,1}, \mathrm{G}/\On)_0 \ar[d]\\
\Map_{\partial}(S^1\times D^{2n-1}, \mathrm{G}/\On)_0 \ar[r]  &  
\Map_{\partial}(X_g, \mathrm{G}/\On)_0
}
\eeq
is homotopy cartesian, as both rows extend to fibre sequence by taking the natural map (to the right) to the component of the constant map in the space $\Map_{\partial/2}(W_{g,1}, \mathrm{G}/\On)$ of mappings $W_{g,1}\to \mathrm{G}/\On$ which are constant on $1/2\partial W_{g,1}$. 
\end{proof}

\begin{rem}
This argument also identifies the vertical homotopy fibres in the square of Proposition \ref{prop:EMCG up to concordance} to be $\Map_{\partial/2}(W_{g,1}, \mathrm{G}/\On)_0$, so that there is a fibration sequence
\begin{equation*}
\Map_{\partial/2}(W_{g,1}, \mathrm{G}/\On)_0 \to B\blEmb_{\partial/2}^{\cong}(X_{g}) \to B\hAut_{\partial}^{\cong}(X_{g}).
\end{equation*}
\end{rem}

\section{Self-embeddings of $X_g$}\label{sect:Emb mcg}
Recall from Section \ref{sect:mflds-quad} that $U_g(\Z,\Lambda_n)$ and $U_g(\Z[\pi],\Lambda_n)$ are the unitary groups associated to the genus $g$ hyperbolic quadratic modules $(\pi_n(W_{g,1}),\lambda_W,q_W)$ and $(\pi_n(X_g),\lambda_X,q_X)$ over $\Z$ and $\Z[\pi]$ respectively with the appropriate form parameters.

In this section we study the map 
$$\pi_0(\Emb_{\partial/2}^{\cong}(X_g))\to U_g(\Z[\pi], \Lambda_n)$$ 
which takes an isotopy class of embeddings to the induced automorphism on $\pi_n(X_g)$. We will show that for large enough $g$ the image of this map is a certain subgroup $\Omega_g$ of index $|bP_{2n}|$, the order of the (finite) group of homotopy $(2n-1)$-spheres which bound a parallelisable manifold. Using this we determine the group $\pi_0(\Emb_{\partial/2}^{\cong}(X_g))$ up to an extension. 

Throughout we use the conventions of Section \ref{sect:mflds-quad}.

\subsection{Realising automorphisms of a quadratic module by self-embeddings}
We denote by $\Emb_{\partial/2}(X_g)$ the space of all embeddings $X_g\to X_g$ whose restriction to $\frac{1}{2}\partial X_g$ is the identity. Recall that $\Emb^{\cong}_{\partial/2}(X_g)\subset \Emb_{\partial/2}(X_g)$ denotes the union of the path components of $\Emb_{\partial/2}(X_g)$ in the image of the map $\pi_0(\Diff_{\partial}(X_g))\to\pi_0(\Emb_{\partial/2}(X_g))$.

We will often implicitly use that $\pi = \Z$ has trivial Whitehead group \cite[Theorem 15]{Higman}, to avoid any discussion of ordinary versus simple homotopy equivalences.

\begin{lem}\label{lem:emb-are-he}
Every embedding $e\in\Emb_{\partial/2}(X_g)$ is a (simple) homotopy equivalence.
\end{lem}
\begin{proof}
We use the following general fact: if $f : A\to B$  is an isometry of quadratic left $R$-modules for some ring $R$, and $A$ is nonsingular, then $f$ is a split monomorphism. This is because the isomorphism $A\to A^{\vee}$ induced by the  form on $A$ factors as
\beq
A\xto{f} B\to B^{\vee}\xto{f^{\vee}}A^{\vee}.
\eeq
In our situation, the embedding $e$ induces an isometry on $\pi_n(X_g)=\pi_n(X_g,\ast)$ (i.e. it preserves the bilinear form $\lambda_X$). Hence, by the general fact, $e$ induces a split injection on $\pi_n$, and so a splitting provides a surjective endomorphism  of the finitely generated $\Z[\pi]$-module $\pi_n(X_g)$. Because $\Z[\pi]$ is a commutative ring, this endomorphism is an isomorphism by Nakayama's lemma. Thus by the Hurewicz and Whitehead theorems, the lift $\widetilde{e}:\widetilde{X}_g\to \widetilde{X}_g$ of $e$ to the universal cover is a homotopy equivalence. The result follows by noticing that 
$e$ also induces an isomorphism on $\pi_1$, and applying Whitehead's theorem once more.
\end{proof}
A consequence of the previous lemma is that every self-embedding of $X_g$ which is the identity on $1/2\partial X_g$ induces an \emph{automorphism} of $\pi_n(X_g)$ which preserves the quadratic module structure. In other words we have a homomorphism of monoids
\beq
\pi_0(\Emb_{\partial/2}(X_g))\to U_g(\Z[\pi], \Lambda_n).
\eeq
In the rest of this section we will show that this map is surjective, and characterise the subgroup $\pi_0(\Emb_{\partial/2}^{\cong}(X_g)) \subset \pi_0(\Emb_{\partial/2}(X_g))$ of embeddings isotopic to diffeomorphisms in terms of this map.

\begin{prop}\label{prop:realizing}
The map $\pi_0(\Emb_{\partial/2}(X_g))\to U_g(\Z[\pi], \Lambda_n)$ is surjective, provided $n\geq 3$.
\end{prop}
\begin{proof}
As in Section \ref{sect:mflds-quad}, we write $a_i, b_i$ for the homotopy classes of $S^n\times\{\ast\}$ and $\{\ast\}\times S^n$ inside the $i$th copy of $S^n\times S^n\subset W_{g,1}\subset X_g$, so that the pairs $(a_i,b_i)$ are orthogonal with respect to  $\lambda_W$ and $\lambda_X$, and hence form a $\Z$-basis for $\pi_n(W_{g,1})$ and a $\Z[\pi]$-basis for $\pi_n(X_g)$. Let $\rho\in U_g(\Z[\pi], \Lambda_n)$. The images $\rho(a_i), \rho(b_i) \in \pi_n(X_g)$ may be represented by embeddings $S^n \times D^n \hookrightarrow X_g$ (together with a framed path to the base point $\ast\in\partial X_g$) by Lemma \ref{lem:QuadDetectsEmb}, as $n\geq 3$ and $q_X(\rho(a_i)) = q_X(a_i)=0=q_X(b_i) = q_X(\rho(b_i)))$. 
Furthermore, since the pairs $(a_i,b_i)$ are orthogonal, we can use the Whitney trick to isotope the embeddings $\rho(a_i)$ and $\rho(b_i)$ so that so that their cores intersect transversely in exactly one point, and are disjoint from the other embeddings. Plumbing them together gives an embedding $e:W_{g,1}\to X_g$ which maps $\frac{1}{2}\partial W_{g,1}$ into $\frac{1}{2}\partial X_g$. To extend this embedding to an embedding $X_g\to X_g$ with the required boundary condition, it suffices to add a collar of $1/2\partial X_g$ to $e(W_{g,1})$.
\end{proof}

\subsection{Realising automorphisms of a quadratic module by diffeomorphisms}

We now investigate the subgroup of $U_g(\Z[\pi], \Lambda_n)$ of those automorphisms realised by elements of $\Emb^{\cong}_{\partial/2}(X_g)$. We will do this using tools from surgery theory which we recall briefly. We remind the reader that we use the fact that the Whitehead group of $\pi \cong \Z$ is trivial to neglect any discussion of bases or simple maps.

\subsubsection{The structure set}
The (smooth) structure set $\calS_{\partial}(S^1\times D^{2n-1})$ has elements equivalence classes of pairs $[N,f]$, where $N$ is a compact smooth $2n$-manifold, and $f:(N,\partial N)\to (S^1\times D^{2n-1},S^1\times S^{2n-2})$ is a homotopy equivalence whose restriction to the boundary is a diffeomorphism. Two such pairs $[N_0,f_0]$ and $[N_1,f_1]$ are equivalent if there exists an $h$-cobordism $W$ between $N_0$ and $N_1$ with $\partial W=N_0\cup V\cup N_1$, for some smooth $n$-manifold $V$ such that $\partial V=\partial N_0\sqcup \partial N_1$, together with a map $F:W\to S^1\times D^{2n-1}\times [0,1]$ such that $F|_{N_j}=f_j$, for $j=0,1$, and $F|_{V}$ maps $V$ diffeomorphically onto $S^1\times S^{2n-2}\times I$.
\subsubsection{L-groups}
The even dimensional $L$-groups $L_{2n}(\Z[\pi])$ are defined as the quotient of the monoid of equivalence classes of finitely generated free $(-1)^n$-quadratic modules over $\Z[\pi]$ by the submonoid of hyperbolic forms. The monoid operation is orthogonal sum, and this quotient is in fact a group.

The odd dimensional $L$-groups $L_{2n+1}(\Z[\pi])$ are given by the quotient 
\beq
U(\Z[\pi])/RU(\Z[\pi])
\eeq
where $U(\Z[\pi])=\varinjlim U_g(\Z[\pi], \Lambda_n^{\min})$ is the stabilisation of the unitary groups described in Section \ref{subsec:quadMod} with---crucially---form parameter $\Lambda_n^{\min}$. Here $RU(\Z[\pi])$ denotes the subgroup of $U(\Z[\pi])$ generated by the commutator subgroup $[U(\Z[\pi]),U(\Z[\pi])]$ and the element $\sigma$ represented by the matrix $\begin{bsmallmatrix}
0 & 1\\
(-1)^n & 0
\end{bsmallmatrix} \in U_1(\Z[\pi], \Lambda_n)$.

By Wall's realisation theorem \cite[Theorem 6.5]{wall-book} to each $x\in L_{2n+1}(\Z[\pi])$ there is associated a manifold triad $(W,\partial_-W,\partial_+W)$ with a degree $1$ normal map $\Phi_x$ to
\beq
(S^1\times D^{2n-1}\times I,S^1\times D^{2n-1}\times\{0\}\; \cup\; \partial (S^1\times D^{2n-1})\times I,S^1\times D^{2n-1}\times\{1\})
\eeq
such that $\Phi_x$ is the identity on $\partial_-W$ and restricts to a homotopy equivalence on $\partial_+W$. This gives a map $$\partial:L_{2n+1}(\Z[\pi])\to\calS_\partial(S^1\times D^{2n-1})$$ defined by
\beq
\partial (x):=[\partial_+W,\Phi_x|_{\partial W_+}].
\eeq

\subsubsection{The complement of a self-embedding}
We define a map 
\begin{equation}\label{eq:kappa}
\kappa:\pi_0(\Emb_{\partial/2}(X_g))\to\calS_{\partial}(S^1\times D^{2n-1})
\end{equation} 
as follows. Let $e:X_g\to X_g$ be an embedding which fixes $S^1\times D^{2n-2}_+\subset\partial X_g$, and let $C$ denote the closure of the complement of $e(X_g)$. The boundary $\partial C$ has a decomposition
\beq
\partial C=S^1\times D^{2n-2}_-\cup e(S^1\times D^{2n-2}_-),
\eeq
which, after smoothing corners, is diffeomorphic to $S^1\times S^{2n-2}$. Thus we have an identification $e(X_g) \cup_{e(S^1 \times D^{2n-2}_-)} C = X_g$. By the Seifert--van Kampen theorem the natural map
$$\pi_1(e(X_g)) *_{\pi_1(e(S^1 \times D_-^{2n-2}))} \pi_1(C) \to \pi_1(X_g)$$
is an isomorphism; as $\pi_1(e(S^1 \times D_-^{2n-2})) \to \pi_1(e(X_g))$ is an isomorphism, we deduce that $\pi_1(C) \to \pi_1(X_g)$ is an isomorphism as well, and so $\pi_1(C)\cong\pi$ too. In particular, there is a map $C \to S^1$ extending the projection map $\partial C = S^1 \times S^{2n-2} \to S^1$, giving a map of pairs
$$\sigma_e : (C, \partial C) \to (S^1 \times D^{2n-1}, S^1 \times S^{2n-2})$$
extending the identification on the boundaries, well-defined up to homotopy. We claim that $\sigma_e$ is a homotopy equivalence of pairs: as it is a homeomorphism on the boundary, it suffices to show that $C \to S^1$ is an equivalence, or equivalently that the universal cover $\widetilde{C}$ is contractible. But we have 
$$e(\widetilde{X}_g) \cup_{\R \times D_-^{2n-2}} \widetilde{C} \simeq \widetilde{X}_g$$
and $e : X_g \to X_g$ is an equivalence, so $\widetilde{C}$ is indeed contractible. The map $\kappa$ sends the isotopy class of the embedding $e$ to the structure $[C,\sigma_e]\in\calS_{\partial}(S^1\times D^{2n-1})$.

\begin{prop}\label{prop:ses Emb-structure set}
The map \eqref{eq:kappa} is a monoid homomorphism with kernel $\pi_0(\Emb_{\partial/2}^{\cong}(X_g))$.
\end{prop}
\begin{proof}
If $\kappa(e)=[S^1\times D^{2n-1},\id]$, then, by the $s$-cobordism theorem and the fact that the Whitehead group of $\pi$ is trivial, $\sigma_e$ is homotopic (relative to the boundary) to a diffeomorphism $\varphi:C\to S^1\times D^{2n-1}\cong S^1\times D^{2n-2}\times [-1,0]$ which maps $S^1\times D^{2n-2}_-\subset\partial C$ identically onto $S^1\times D^{2n-2}\times\{0\}$, and $e(S^1\times D^{2n-2}_-)=(S^1\times D^{2n-2}\times\{-1\})\cup\partial(S^1\times D^{2n-2})\times [-1,0]$. With this we define a $1$-parameter family of submanifolds of $X_g$:
\beq
e(X_g)\cup_{e(S^1\times D^{2n-2}_-)} \varphi^{-1}\left(\partial (S^1\times D^{2n-2}_-)\times [-1,0]\cup S^1\times D^{2n-2}_-\times[-1,t]\right)\subset X_g.
\eeq
which gives an isotopy from $e$ to a diffeomorphism that is the identity on the boundary of $X_g$. This isotopy is through embeddings that are the identity on $S^1\times D^{2n-2}_+\subset\partial X_g$. Thus the kernel of $\kappa$ is contained in $\pi_0(\Emb^{\cong}_{\partial/2}(X_g))$.

For the other inclusion, if $e$ is a diffeomorphism that restricts to the identity on $\partial X_g$, we can push $S^1\times D^{2n-1}_-$ slightly to the interior of $X_g$ to get an isotopy to an embedding whose complement is diffeomorphic to $S^1\times D^{2n-1}$. 
\end{proof}

\subsubsection{The map $\kappa$ and $L$-theory} 

We shall now identify the image of the map $\kappa$. Our analysis will be based on the following proposition. 
\begin{prop}\label{prop:kappa}
If $n \geq 3$ but $n \neq 3,7$, then the following diagram commutes
\begin{equation}\label{eq:Emb-surgery}
\begin{gathered}
\xymatrix{
\pi_0(\Emb_{\partial/2}(X_g)) \ar[r]^-{\kappa}\ar[d]_{\pi_n} &  \calS_{\partial}(S^1\times D^{2n-1})\\
U_g(\Z[\pi], \Lambda_n^{\min})\ar[r]^{s}& L_{2n+1}(\Z[\pi])\ar[u]_{\partial}
}
\end{gathered}
\end{equation}
where the map $s$ is stabilisation and projection.
\end{prop}

We emphasise that here the unitary group is defined with form parameter $\Lambda_n^{\min}$, as required in the definition of $L$-theory, so if $n=3,7$ then the left-hand vertical map would not exist in any natural way, as the quadratic modules constructed in Section \ref{subsec:quadMod} have a slightly larger form parameter in these cases. 

\begin{proof}[Proof of Proposition \ref{prop:kappa}]
Let $e : X_g \to X_g$ be an embedding which fixes half the boundary, $C$ be the closure of its complement and $\sigma_e : (C, \partial C) \to (S^1 \times D^{2n-1}, S^1 \times S^{2n-2})$ be the structure map constructed above. We will show how to construct a manifold triad $(W,\partial_-W,\partial_+W)$, a degree $1$ map $\Psi$ from this triad to
$$(S^1\times D^{2n-1}\times I, S^1\times D^{2n-1}\times\{0\}\; \cup\; \partial(S^1\times D^{2n-1})\times I,S^1\times D^{2n-1}\times\{1\}),$$
and a trivialisation $F$ of $\Psi^\ast(\nu\times I)\oplus TW$, where $\nu$ is the stable normal bundle of $S^1\times D^{2n-1}$, such that: 
\begin{enumerate}[(i)]
\item $\Psi$ is the identity on $\partial_-W$ and $F$ is the standard trivialisation here, 
\item $\Psi$ is a homotopy equivalence on $\partial_+W$,
\item $\Psi$ is $n$-connected,
\item $[\partial_+W,\Psi|_{\partial_+W}]=[C,\sigma_e]\in\calS_{\partial}(S^1\times D^{2n-1})$, and
\item its surgery obstruction is $s(\pi_n(e)) \in  L_{2n+1}(\Z[t,t^{-1}])$.
\end{enumerate}
By the description of Wall realisation, such a manifold triad shows that the square commutes when evaluated on the embedding $e$, as required.

The manifold $W$ will be obtained as a composition of cobordisms $W_i$. As $\nu$ is trivial, $F$ is the same as a stable framing of the tangent bundle of $W$: this will be obtained by inductively extending stable framings over each $W_i$. Firstly we trivially attach $2g$-handles to the interior of $S^1\times D^{2n-1}$ so that the resulting manifold is $X_g$. The trace of these surgeries is the cobordism $W_1$, with 
\begin{align*}
\partial_-W_1 &= (S^1\times D^{2n-1}\times\{0\}) \cup(\partial (S^1\times D^{2n-1})\times [0,1])\\
\partial_+W_1 &=X_g\times\{1\},
\end{align*}
and the standard stable framing of $S^1 \times D^{2n-1}$ extends over this cobordism, as the handles were attached trivially. Secondly, by definition of the complement $C$ there is a diffeomorphism $e(X_g) \cup_{e(S^1\times D^{2n-1}_-)} C \to X_g$, and we let $W_2$ be the cobordism given by  the trace of the inverse of this dffeomorphism, with
\begin{align*}
\partial_-W_2 &=(X_g\times\{1\})\cup(\partial (S^1\times D^{2n-1}) \times [1,2])\\
\partial_+W_2 &=(e(X_g)\cup_{e(S^1\times D^{2n-1}_-)} C) \times \{2\}.
\end{align*}
As this is the trace of a diffeomorphism, the stable framing on $\partial_+ W_1$ extends over it. This induces a stable framing $F'$ on $W_{g,1} \cong e(W_{g,1}) \subset e(X_g) \subset \partial_+ W_2$. 

This stable framing induces (after destabilising) a trivialisation of the normal bundles of each of the cores $S^n \times \{*\}, \{*\} \times S^n \subset W_{g,1} = D^{2n} \# (S^n \times S^n)^{\# g}$, but these trivialisations need not be the same as the standard trivialisations. However, we claim that there is a diffeomorphism $\varphi : W_{g,1} \to W_{g,1}$, which is the identity on the boundary and acts as the identity on $\pi_n(W_{g,1})$, such that the stable framing $\varphi_* F'$ induces the standard trivialisation of the normal bundles of each of the cores. This follows from Kreck's description of the mapping class group $\pi_0(\Diff_\partial(W_{g,1}))$, which in particular shows that for any homomorphism $\psi : \pi_n(W_{g,1}) \to \pi_n(\SOn(n))$ there is a diffeomorphism of $W_{g,1}$ (trivial on the boundary and acting trivially on $\pi_n(W_{g,1})$) which re-trivialises the normal bundle of each embedded $n$-sphere $a$ with trivial normal bundle by $\psi([a])$. We extend this diffeomorphism from $W_{g,1} \cong e(W_{g,1})$ to $e(X_g)$ and hence to $\partial_+W_2 =e(X_g)\cup_{e(S^1\times D^{2n-1}_-)} C$, and its trace gives a cobordism $W_3$ with
\begin{align*}
\partial_-W_3 &=(e(X_g)\cup_{e(S^1\times D^{2n-1}_-)} C)\cup(\partial (S^1\times D^{2n-1})\times [2,3])\\
\partial_+W_3 &=(e(X_g)\cup_{e(S^1\times D^{2n-1}_-)} C) \times \{3\},
\end{align*}
and the stable framing on $\partial_+ W_2$ extends over it. By construction, the induced stable framing on $W_{g,1} \cong e(W_{g,1}) \subset e(X_g) \subset \partial_+ W_3$ induces the standard normal framing of the cores, so we may do surgery along each of the $S^n \times \{*\}$ to obtain a stably framed cobordism from $W_{g,1}$ to $D^{2n}$. Extending this trivially, we obtain a cobordism $W_4$ with
\begin{align*}
\partial_-W_4 &=(e(X_g)\cup_{e(S^1\times D^{2n-1}_-)} C)\cup(\partial (S^1\times D^{2n-1})\times [3,4])\\
\partial_+W_4 &= C \times \{4\},
\end{align*}
and equipped with a stable framing extending that on $\partial_+ W_3$. We then take
$$W=W_1\cup W_2\cup W_3 \cup W_4$$
with the combined stable framing. This cobordism has
\begin{align*}
\partial_-W &=(S^1\times D^{2n-1}\times\{0\})\cup(\partial(S^1\times D^{2n-1})\times [0,4])\\
\partial_+W &= C \times \{4\}.
\end{align*}

By construction $W$ has a degree $1$ $n$-connected map $\Psi:(W,\partial_-W,\partial_+W)\to (S^1\times D^{2n-1}\times [0,4], S^1\times D^{2n-1}\times\{0\}\cup\partial(S^1\times D^{2n-1})\times [0,4], S^1\times D^{2n-1}\times\{4\})$, and the discussion of stable framings shows how it is covered by a trivialisation $F$ of $\Psi^\ast(\nu\times I)\oplus TW$ extending the standard trivalisation on $\partial_- W$. Also $[\partial_+W,\Psi|_{\partial_+W}] = [C,\sigma_e]$. Thus $\Psi$ satisfies (i), (ii), (iii), and (iv).

It remains to determine the surgery obstruction of $\Psi$. By definition the stably framed cobordism $W_2\cup W_3 \cup W_4$ is obtained as the trace of the surgery on the normal map $X_g \times \{1\} \to S^1 \times D^{2n-1} \times \{1\}$ given by the lagrangian $L := \mathrm{Ker}(\pi_n(X_g) \to \pi_n(W_2 \cup W_3 \cup W_4))$ in the surgery kernel
$$\pi_{n+1}(S^1 \times D^{2n-1} \times \{1\}, X_g \times \{1\}) \cong \pi_n(X_g).$$
As $W_4$ is the trace of the surgery along the standard lagrangian $L_0$ in $\pi_n(X_{g}) \cong \pi_n(e(X_g)) \cong \pi_n(e(X_g)\cup_{e(S^1\times D^{2n-1}_-)} C)$, the diffeomorphism $\varphi$ used to build $W_3$ acts trivially on $\pi_n(W_{g,1})$, and the diffeomorphism used to build $W_2$ acts as $e_*^{-1}$ on $\pi_n(X_g)$, it follows that $L$ is the image of $L_0$ under $e_* : \pi_n(X_g) \to \pi_n(X_g)$. By the description \cite[Theorem 6.5]{wall-book} of Wall realisation, it follows that $\Psi$ has surgery obstruction $s(\pi_n(e))$, proving (v).
\end{proof}

\subsubsection{The image of the map $\kappa$}
\begin{prop}\label{prop:kappa37}
If $n=3,7$ the map
$$\kappa : \pi_0(\Emb_{\partial/2}(X_g)) \to  \calS_{\partial}(S^1\times D^{2n-1})$$
is trivial.
\end{prop}
\begin{proof}
Calculating with the surgery exact sequence, one finds that there is a bijection
\begin{align*}
\Theta_{2n-1} \times \Theta_{2n} &\to \calS_{\partial}(S^1\times D^{2n-1}) \\
(\Sigma_A, \Sigma_B) &\mapsto (S^1 \times (D^{2n-1} \# \Sigma_A)) \# \Sigma_B.
\end{align*}
Suppose that $\kappa(e) = (\Sigma_A, \Sigma_B)$ under this bijection.

By closing the manifold $X_g$ off by gluing in $S^1 \times D^{2n-1}$ we find that there is a diffeomorphism
$$(S^1 \times \Sigma_A) \# \Sigma_B \# (S^n \times S^n)^{\# g} \cong (S^1 \times S^{2n-1}) \# (S^n \times S^n)^{\# g}.$$
Take universal covers, and consider disjoint lifts of $\{*\} \times \Sigma_A$ from the left-hand side and $\{*\} \times S^{2n-1}$ from the right-hand side. As this manifold is stably parallelisable, the compact region bounded by these gives a stably parallelisable cobordism from $\Sigma_A$ to $S^{2n-1}$, which therefore shows that $\Sigma_A \in bP_{2n} \leq	\Theta_{2n-1}$. Now $bP_6 = bP_{14}=0$, so under the given assumptions $\Sigma_A \cong S^{2n-1}$.

On the other hand, by closing the manifold $X_g$ off by gluing in $D^2 \times S^{2n-1}$ we find that there is a diffeomorphism $\Sigma_B \#  (S^n \times S^n)^{\# g} \cong  (S^n \times S^n)^{\# g}$. The manifold $ (S^n \times S^n)^{\# g}$ has trivial inertia group, by \cite[Theorem 3.1]{Kosinski}, \cite{WallInertia}, so $\Sigma_B \cong S^{2n}$.
\end{proof}

\begin{cor}\label{cor:Emb-bP}
For $n \geq 3$ there is a monoid homomorphism
$$\bar{\kappa} : \pi_0(\Emb_{\partial/2}(X_g)) \to bP_{2n}$$
whose kernel is the group $\pi_0(\Emb^{\cong}_{\partial/2}(X_g))$.

If $n$ is even then $\bar{\kappa}$ is surjective as long as $g \geq 8$; if $n$ is odd then it is surjective as long as $g \geq 2$.
\end{cor}
\begin{proof}
If $n=3, 7$ then $bP_{2n}=0$ so we take $\bar{\kappa}=0$, and the statement follows from Propositions \ref{prop:ses Emb-structure set} and \ref{prop:kappa37}.

If $n \neq 3,7$ then we consider Propositions  \ref{prop:ses Emb-structure set} and \ref{prop:kappa} instead, which show that it is enough to show that the image of $\partial:L_{2n+1}(\Z[t,t^{-1}])\to\calS_{\partial}(S^1\times D^{2n-1})$ is isomorphic to $bP_{2n}$. Using the fact that $L_{2n+1}(\Z)=0$ (and that the Whitehead group of $\pi$ vanishes), work of Shaneson \cite[Theorem 5.1]{Shaneson} gives an isomorphism
$$L_{2n}(\Z) \xto{\sim} L_{2n+1}(\Z[\pi]).$$
In terms of normal cobordisms Shaneson shows that this is given by taking products with the circle, so it interacts with Wall realisation to give a commutative square
\beq
\xymatrix{
L_{2n}(\Z)\ar[r]^-{\partial}\ar[d]_{\cong} & \calS_{\partial}(D^{2n-1})\ar[d]^{S^1\times}\\
L_{2n+1}(\Z[\pi])\ar[r]^-{\partial} & \calS_{\partial}(S^1\times D^{2n-1}).
}
\eeq
Now $\calS_{\partial}(D^{2n-1}) = \Theta_{2n-1}$, the top map is a surjection onto $bP_{2n} \leq \Theta_{2n-1}$, and it follows from the $h$-cobordism theorem that the right-hand map is injective. Proposition \ref{prop:kappa} therefore gives a map
$$\bar{\kappa} : \pi_0(\Emb_{\partial/2}(X_g)) \to bP_{2n}$$
whose kernel is the same as that of $\kappa$, which is $\pi_0(\Emb^{\cong}_{\partial/2}(X_g))$ by Proposition \ref{prop:ses Emb-structure set}.

For the surjectivity statement, first recall that $L_{2n}(\Z)$ is isomorphic to $\Z$ (generated by the $(+1)$-quadratic form $E_8$ of rank 8) if $n$ is even and to $\Z/2$ (generated by the $(-1)$-quadratic form $K=\left(\Z\{e,f\}, \lambda(e,f)=1, q(e)=q(f)=1\right)$ of rank 2) if $n$ is odd. Shaneson's isomorphism $L_{2n}(\Z) \xto{\sim} L_{2n+1}(\Z[\pi])$ sends a $(-1)^n$-quadratic form $M$ over $\Z$ to the automorphism $\mathrm{Id} \oplus t$ of the $(-1)^n$-quadratic form
$$M \otimes\Z[t^{\pm 1}] \oplus -{M} \otimes \Z[t^{\pm 1}] \cong (M \oplus -{M}) \otimes \Z[t^{\pm 1}]$$
over $\Z[t^{\pm 1}]$, which is a hyperbolic form. When $M$ is $E_8$ this quadratic form is isomorphic to $(\pi_n(X_8), \lambda_X, q_X)$; when $M$ is $K$ it is isomorphic to $(\pi_n(X_2), \lambda_X, q_X)$. Thus for $g$ in the indicated range the map $U_g(\Z[\pi], \Lambda_n^{\min}) \to L_{2n+1}(\pi)$ is surjective. This, together with Proposition \ref{prop:realizing}, and the square \eqref{eq:Emb-surgery} give the surjectivity of the map $\bar{\kappa}:\pi_0(\Emb_{\partial/2}(X_g))\to bP_{2n}$.
\end{proof}

\begin{rem}
As a submonoid of a finite group is a group, Corollary \ref{cor:Emb-bP} implies that the monoid $\pi_0(\Emb_{\partial}(X_g))$, under composition of self-embeddings, is in fact a group.
\end{rem}

\subsubsection{Realising automorphisms by diffeomorphisms}
From diagram \eqref{eq:Emb-surgery} for $n \neq 3,7$ we have a map $U_g(\Z[\pi], \Lambda_n)\to\calS_{\partial}(S^1\times D^{2n-1})$ whose image, by Corollary \ref{cor:Emb-bP}, is isomorphic to $bP_{2n}$ for large enough $g$. By Proposition \ref{prop:kappa37} we also have such a map for $n=3,7$, trivially. Let $\Omega_g$ denote the kernel of this map, giving for large enough $g$ an exact sequence
\begin{equation}\label{eq:Ug-bP}
0\to\Omega_g\to U_g(\Z[\pi], \Lambda_n)\to bP_{2n}\to 0.
\end{equation}
Combining Proposition \ref{prop:realizing}, Corollary \ref{cor:Emb-bP}, and the sequence \eqref{eq:Ug-bP}, we obtain the following.
\begin{cor}\label{cor:Realising}
Let $g\geq 8$ if $n$ is even, and $g \geq 2$ if $n$ is odd. There is a map of short exact sequences
\beq
\xymatrix{
0\ar[r] &\pi_0(\Emb_{\partial/2}^{\cong}(X_g))\ar[r]\ar@{->>}[d]
&\pi_0(\Emb_{\partial/2}(X_g))\ar[r]\ar@{->>}[d] 
& bP_{2n}\ar[r]\ar@{=}[d]
& 0\\
0\ar[r] &\Omega_g\ar[r]
& U_g(\Z[\pi], \Lambda_n) \ar[r]
& bP_{2n}\ar[r]
&0
}
\eeq
whose vertical maps are surjections.\qed
\end{cor}

\subsection{Isotopy classes of embeddings}\label{sec:MCG}

With the results of the last section we may determine the group $\pi_0(\Emb_{\partial/2}^{\cong}(X_g))$ up to an extension, and it is surprisingly simple. Write $S\pi_n\SOn(n) :={\rm{im}}(S: \pi_n\SOn(n)\to\pi_n\SOn(n+1))$.

\begin{cor}
For $n\geq 3$ there is an extension
$$1 \to \Hom_\Z(H,S\pi_n\SOn(n)) \to \pi_0({\Emb}_{1/2\partial}^{\cong}(X_g))\to \Omega_g \to 1.$$
\end{cor}
\begin{proof}
Consider the commutative diagram
\beq
\xymatrix{
\pi_0(\widetilde{\Emb}_{1/2\partial}^{\cong}(X_g)) \ar[r]\ar[d] &
\pi_0(\widetilde{\Emb}_{1/2\partial}^{\cong}(W_{g,1}))\ar[d]\\
\pi_0(\hAut_{\partial}^{\cong}(X_g)) \ar[r]\ar[d] &
\pi_0(\hAut_{\partial}^{\cong}(W_{g,1}))\ar[d]\\
\Omega_g \ar[r] & U_g(\Z, \Lambda_n).
}
\eeq
By Proposition \ref{prop:BEmb to BhAut} the top square is cartesian, and its vertical maps are surjective (by definition of the decoration $\cong$). By Theorem \ref{thm:hAutComparison} and the results of the last section the bottom square is cartesian and the vertical maps are surjective\footnote{The surjectivity of the right vertical maps is a consequence of \cite[Theorem 2]{kreckisotopy}.}. Thus the outer square is cartesian with surjective vertical maps. Furthermore, by Proposition \ref{prop:EMCG up to concordance} and Remark \ref{rem:morlet-Wg} it remains so when the spaces of block embedding are replaced with spaces of embeddings.

 In \cite[eq.\ (5)]{KR-WAlg}, using the work of Kreck \cite{kreckisotopy} the kernel of the right-hand vertical map is identified with $\Hom_\Z(H,S\pi_n\SOn(n))$. This gives the claimed extension.
\end{proof}

\begin{rem}
It follows from \cite[Theorem A]{KrannichMCG} that this extension is split for $n$ odd when $n \neq 3,7$, and is not split when $n=3,7$ unless $g \leq 1$.
\end{rem}

\section{Stably framed self-embeddings of $X_g$}\label{sec:sfr}
As we have described in Section \ref{sec:WeissFib}, given a boundary condition $\ell_{\partial X_g} : TX_g\vert_{\partial X_g} \to \sfr_{2n}^*\gamma$ there is a space $B\blEmb^{\mathrm{sfr}, \cong}_{\partial/2}(X_g ; \ell_{\partial/2 X_g})$ fitting into a fibration sequence
\begin{equation*}
\Bun_{\partial/2}(TX_g, \mathrm{sfr}_{2n}^*\gamma; \ell_{1/2\partial X_g}) \to B\blEmb^{\mathrm{sfr}, \cong}_{\partial/2}(X_g ; \ell_{1/2\partial X_g}) \to B\blEmb^{\cong}_{\partial/2}(X_g).
\end{equation*} 
Let $\Map_{\partial/2}(X_g,\SOn)$ denote the space of mappings $X_g\to \SOn$ which are constant on $1/2\partial X_g$. Choosing a stable framing $\ell$ which extends $\ell_{\partial X_g}$ gives an identification
$$\Bun_{\partial/2}(TX_g, \mathrm{sfr}_{2n}^*\gamma; \ell_{1/2\partial X_g})  \simeq \Map_{\partial/2}(X_g, \SOn).$$
There is an analogous fibration sequence, and identification, for $W_{g,1}$.

\begin{prop}\label{prop:BEmbSfr to BhEmb}
The square
\beq
\xymatrix{
B\blEmb^{\mathrm{sfr}, \cong}_{\partial/2}(X_g ; \ell_{1/2\partial X_g}) \ar[r] \ar[d]& B\blEmb^{\mathrm{sfr}, \cong}_{\partial/2}(W_{g,1} ; \ell_{1/2\partial W_{g,1}}) \ar[d]\\
B\blEmb^{\cong}_{\partial/2}(X_g) \ar[r] & B\blEmb^{\cong}_{\partial/2}(W_{g,1})
}
\eeq
is homotopy cartesian.
\end{prop}
\begin{proof}
The map on vertical homotopy fibres is $\Map_{\partial/2}(X_g, \SOn) \to \Map_{\partial/2}(W_{g,1}, \SOn)$ which is an equivalence.
\end{proof}

\subsection{Higher homotopy groups}\label{sec:higherpisfrEmb}

Combining Proposition  \ref{prop:BEmbSfr to BhEmb} with Proposition \ref{prop:BEmb to BhAut} gives the following relation between spaces of stably framed embeddings and of homotopy automorphisms.

\begin{cor}
For all $n\geq 3$, the square
\begin{equation}\label{eq:SfrEmbPullback}
\begin{gathered}
\xymatrix{
B\blEmb^{\mathrm{sfr}, \cong}_{\partial/2}(X_g ; \ell_{1/2\partial X_g}) \ar[r] \ar[d]& B\blEmb^{\mathrm{sfr}, \cong}_{\partial/2}(W_{g,1} ; \ell_{1/2\partial W_{g,1}}) \ar[d]\\
B\hAut^{\cong}_{\partial}(X_g) \ar[r] & B\hAut^{\cong}_{\partial}(W_{g,1})
}
\end{gathered}
\end{equation}
is homotopy cartesian, and has vertical homotopy fibres identified with
$$\Map_{\partial/2}(X_g, \mathrm{G})_J \xto{\sim} \Map_{\partial/2}(W_{g,1}, \mathrm{G})_J,$$
where the subscript $J$ denotes those path components of maps which factor up to homotopy through $J : \On \to \mathrm{G}$ (or, equivalently, which are null when composed with $\mathrm{G} \to \mathrm{G}/\On$). \qed
\end{cor}

In Propositions \ref{prop:pi_khAut} and  \ref{cor:pi_khAutWg} we have calculated $\pi_k(B\hAut_{\partial}^{\cong}(X_g))$ and $\pi_k(B\hAut_{\partial}^{\cong}(W_{g,1}))$ for $2 \leq k < n-1$, and we also have
$$\pi_{k-1}(\Map_{\partial/2}(X_g, \mathrm{G})_J) \cong \pi_{n+k-1}^s \otimes H.$$
The following lemma relates these calculations.

\begin{lem}\label{lem:ConnectingMapNormalInv}
For $2 \leq k < n-1$ the composition
$$\frac{\pi_{2n-1+k}(S^n)}{[\iota_n, \pi_{n+k}(S^n)]} \otimes H \to\pi_{k}(B\hAut_{\partial}^{\cong}(W_{g,1})) \xto{\partial}\pi_{k-1}(\Map_{\partial/2}(W_{g,1}, \mathrm{G})_J) \cong \pi_{n+k-1}^s \otimes H$$
is induced by the stabilisation map on the first factor. Similarly with $W_{g,1}$ replaced by $X_{g}$.
\end{lem}
\begin{proof}
By naturality using \eqref{eq:SfrEmbPullback}, the case of $X_g$ follows from that of $W_{g,1}$.

Theorem \ref{thm:HtyAut} and Proposition \ref{prop:pi_khAut} show that the first map is induced by the pinch action $\pi_{2n-1+k}(W_{g,1}) \to \pi_{k-1}(\hAut_{\partial}^{\cong}(W_{g,1}),\id)$. 
Inspecting the proofs of Propositions \ref{prop:BEmb to BhAut} and \ref{prop:BEmbSfr to BhEmb}, we see that the connecting homomorphism 
\beq
\partial:\pi_{k-1}(\hAut_{\partial}^{\cong}(W_{g,1}),\id)\to\pi_{k-1}(\Map_{\partial/2}(W_{g,1}, \mathrm{G})_J)
\eeq
of the long exact sequence for the leftmost fibre sequence in \eqref{eq:SfrEmbPullback} is given by the composite
\beq
\pi_{k-1}(\hAut_{\partial}^{\cong}(W_{g,1}),\id)\xto{\nu^t}\pi_{k-1}(\Map_{\partial}(W_{g,1}, \mathrm{G})_J)\to\pi_{k-1}(\Map_{\partial/2}(W_{g,1}, \mathrm{G})_J)
\eeq
where the map $\nu^t$ is given by assigning to a self-homotopy equivalence $W_{g,1}\to W_{g,1}$ its \textit{tangential normal invariant}, and the second by relaxing the boundary condition. By \cite[Theorem 4.7 and Lemma 4.8]{madsen-taylor-williams} the composition
\beq
\pi_{2n-1+k}(W_{g,1})\to\pi_{k-1}(\hAut_{\partial}^{\cong}(W_{g,1}),\id) \to \pi_{k-1}(\Map_{\partial/2}(W_{g,1}, \mathrm{G})_J) \cong\pi_{2n+k-1}^s(W_{g,1}),
\eeq
where the isomorphism is by Atiyah duality, agrees with the stabilisation map. The result follows from the commutativity of the square
\beq
\xymatrix{
\pi_{2n-1+k}(W_{g,1})\ar[r]\ar[d] & \pi^s_{2n-1+k}(W_{g,1})\ar[d]^-{\cong}\\
\frac{\pi_{2n-1+k}(S^n)}{[\iota_n,\pi_{n+k(S^n)}]}\otimes H \ar[r] & \pi_{n+k-1}^s\otimes H
}
\eeq
where the lower horizontal map is the stabilisation on the first factor, and the leftmost vertical map is 
$$\pi_{2n-1+k}(W_{g,1}) \to\mathrm{Coker}(\pi_k(\rho))\cong \frac{\pi_{2n-1+k}(S^n)}{[\iota_n,\pi_{n+k(S^n)}]}\otimes H$$ 
where the last isomorphism arises from a proof of Proposition \ref{cor:pi_khAutWg} analogous to that of Proposition \ref{prop:pi_khAut}.
\end{proof}
\begin{example}\label{ex:EmbeddingsOddprimes}
Working $p$-locally at an odd prime, combining this with Example \ref{ex:oddprimes} we have 
$$\pi_k(B\blEmb^{\mathrm{sfr}, \cong}_{\partial/2}(X_g ; \ell_{1/2\partial X_g})) \cong_{(p)} \begin{cases}
\pi_{k-1}^s \otimes S^+_X & \text{ if $n$ is odd}\\
\pi_{k-1}^s \otimes S^-_X & \text{ if $n$ is even}
\end{cases}$$
 for $2 \leq k < n-1$. In particular the lowest non-trivial $p$-local homotopy group is
$$\pi_{2p-2}(B\blEmb^{\mathrm{sfr}, \cong}_{\partial/2}(X_g ; \ell_{1/2\partial X_g}))_{(p)} \cong \Z/p \otimes S^{(-1)^{n+1}}_X.$$
\end{example}

\begin{example}\label{ex:EmbeddingsPi2}
Combining this with Example \ref{ex:pi_1hAut} we have extensions
$$0 \to \mathrm{Coker}(\Sigma \pi_{2n+2}(S^n) \to \pi_{n+2}^s) \otimes H\to \pi_2(B\blEmb^{\mathrm{sfr}, \cong}_{\partial/2}(X_g ; \ell_{1/2\partial X_g}), \ell) \to Y \to 0$$
and 
$$0 \to \mathrm{Ker}(\Sigma \pi_{2n+1}(S^n) \to \pi_{n+1}^s) \otimes H \to Y \to \begin{cases}
\Z/2 \otimes S^+_X & \text{ if $n$ is 6 or is $3\!\!\!\! \mod 4$}\\
\Z/2 \otimes S^-_X & \text{ otherwise }
\end{cases} \to 0.$$
\end{example}

\subsection{Isotopy classes of stably framed embeddings}\label{sec:sfrMCG}
We wish to describe the group
$$\check{\Xi}_g^{\sfr, \ell} := \pi_1(B\widetilde{\Emb}_{1/2\partial}^{\mathrm{sfr}, \cong}(X_g ; \ell_{1/2\partial X_g}), \ell)$$
up to extensions. To express the result, note that a stable framing $\ell$ can be uniquely destabilised, relative to ${1/2\partial X_g}$, to a framing, and hence following the discussion in Section \ref{subsec:quadMod} it gives a map
$$\ell_* : \pi_n(X_g) \to \pi_n(\mathrm{Fr}^+(TX_g)) \cong \mathrm{Imm}_n^{\fr}(X_g).$$
We let $q^\ell_X$ be the restriction of $q^{\fr}$ along this map: then $(\pi_n(X_g), \lambda_X, q^\ell_X)$ is a quadratic module over $\Z[\pi]$ with form parameter $\Lambda_n^{\min}$. Base changing along $\Z[\pi] \to \Z$ gives a quadratic module $(\pi_n(W_{g,1}), \lambda_W, q^\ell_W)$ over $\Z$. 

\begin{lem}\label{lem:HypBdyCond}
The boundary condition $\ell_{1/2\partial X_g}$ and stable framing $\ell$ may be chosen so that $q_X^\ell(a_i) = q_X^\ell(b_i)=0$.
\end{lem}
\begin{proof}
As $X_g$ is the boundary connect-sum of $S^1 \times D^{2n-1}$ and $g$ copies of $W_{1,1}$, it suffices to give a framing on $W_{1,1} \cong (S^n \times D^n_-) \cup (D^n_- \times S^n)$ for which the associated quadratic form vanishes on the cores $S^n \times \{0\}$ and $\{0\} \times S^n$. Recall from Section \ref{subsec:quadMod} that the identification $\pi_n(\mathrm{Fr}^+(T -)) \cong \mathrm{Imm}_n^{\fr}(-)$ via Hirsch--Smale theory depends on a choice of framing of $S^n \times D^n$: if we take this framing on each of $S^n \times D^n_-$ and $D^n_- \times S^n \cong S^n \times D^n_-$ (perhaps combined with a reflection in $D^n_-$ on the latter so that the framings agree up to homotopy on $D^n_- \times D^n_-$) then tautologically the associated quadratic form indeed vanishes on the cores.
\end{proof}

\emph{From now on we make this choice}, so that $(\pi_n(X_g), \lambda_X, q^\ell_X)$ and $(\pi_n(W_{g,1}), \lambda_W, q^\ell_W)$ are hyperbolic quadratic modules.

If $n \neq 3, 7$ then $\Lambda_n^{\min} = \Lambda_n$ and so $q^\ell_X = q_X$ (and $q^\ell_W = q_W$), but if $n = 3, 7$ then $\Lambda_n^{\min}$ is properly contained in $\Lambda_n$ and so $q_X^\ell$ is a slight refinement of $q_X$. In either case we have
$$ \mathrm{Aut}(\pi_n(X_g), \lambda_X, q^{\ell}_X) = U_g(\Z[\pi], \Lambda_n^{\min}) \leq  \mathrm{Aut}(\pi_n(X_g), \lambda_X, q_X) = U_g(\Z[\pi], \Lambda_n)$$
and the composition with $U_g(\Z[\pi], \Lambda_n) \to bP_{2n}$ is surjective (because if $n \neq 3,7$ then this inclusion is an equality, and if $n=3,7$ then $bP_{2n}=0$). Define the group $\Omega_g^{\min}$ to be the kernel of this map, so there is an extension
\begin{equation}\label{eq:DefnOmegaMin}
0 \to \Omega_g^{\min} \to U_g(\Z[\pi], \Lambda_n^{\min}) \to bP_{2n} \to 0.
\end{equation}

\begin{prop}\label{prop:sfrMCG}
Let $g \geq 8$ if $n$ is even and $g \geq 2$ if $n$ is odd. For the choice of stable framing $\ell$ from Lemma \ref{lem:HypBdyCond} there are extensions
$$1 \to L^{\sfr,\ell}_g \to \check{\Xi}_g^{\sfr, \ell} \to \Omega_g^{\min} \to 1$$
and
$$0 \to \begin{cases}
\mathrm{Hom}(H, \Z/4) & \text{$n=6$}\\
0 & \text{$n \neq 6$}
\end{cases} \to L^{\sfr,\ell}_g \to \begin{cases}
0 & \text{$n$ odd}\\
\mathrm{Hom}(H, \Z/2) & \text{$n$ even}
\end{cases} \to 0.$$
\end{prop}

\begin{rem}
It would be interesting to resolve the extension describing $L_g^{\sfr, \ell}$ in the case $n=6$, though it is not necessary for our argument.
\end{rem}

\begin{proof}
Using the square \eqref{eq:SfrEmbPullback} and the fact that its horizontal maps are (compatibly) split surjections up to homotopy, we obtain a pullback of groups
\begin{equation*}
\xymatrix{
\pi_1(B\blEmb^{\mathrm{sfr}, \cong}_{\partial/2}(X_g ; \ell_{1/2\partial X_g}), \ell) \ar[r] \ar[d]& \pi_1(B\blEmb^{\mathrm{sfr}, \cong}_{\partial/2}(W_{g,1} ; \ell_{1/2\partial W_{g,1}}), \ell_{W_{g,1}}) \ar[d]\\
\pi_0(\hAut^{\cong}_{\partial}(X_g)) \ar[r] & \pi_0(\hAut^{\cong}_{\partial}(W_{g,1}))
}
\end{equation*}
which we may combine with that of Theorem \ref{thm:hAutComparison} to obtain a pullback of groups
\begin{equation*}
\xymatrix{
\pi_1(B\blEmb^{\mathrm{sfr}, \cong}_{\partial/2}(X_g ; \ell_{1/2\partial X_g}), \ell) \ar[r] \ar[d]& \pi_1(B\blEmb^{\mathrm{sfr}, \cong}_{\partial/2}(W_{g,1} ; \ell_{1/2\partial W_{g,1}}), \ell_{W_{g,1}}) \ar[d]\\
\Omega_g \ar[r] & U_g(\Z, \Lambda_n).
}
\end{equation*}
The square
\begin{equation*}
\xymatrix{
U_g(\Z[\pi], \Lambda_n^{\min}) \ar[r] \ar@{^(->}[d] & U_g(\Z, \Lambda_n^{\min}) \ar@{^(->}[d]\\
U_g(\Z[\pi], \Lambda_n) \ar[r] & U_g(\Z, \Lambda_n).
}
\end{equation*}
is also a pullback of groups, by direct inspection (as $\Lambda_n/\Lambda^{\min}_n$ is the same for $\Z$ and $\Z[\pi]$).

Therefore the claims in the proposition are equivalent to claims purely about the manifolds $W_{g,1}$, namely that the map
$$\pi_1(B\blEmb^{\mathrm{sfr}, \cong}_{\partial/2}(W_{g,1} ; \ell_{1/2\partial W_{g,1}}), \ell_{W_{g,1}}) \to U_g(\Z, \Lambda_n)$$ 
\begin{enumerate}[(i)]
\item has image the subgroup $\mathrm{Aut}(\pi_n(W_{g,1}), \lambda_W, q_W^\ell) = U_g(\Z, \Lambda_n^{\min})$ determined by the stable framing $\ell$,
\item has kernel $L^{\sfr,\ell}_g$ fitting in to the claimed extension.
\end{enumerate}
Consider the tangential structure $\mathrm{fr}:E\On(2n)\to B\On(2n)$. The maps
$$B\Emb^{\fr, \cong}_{\partial/2}(W_{g,1} ; \ell_{1/2\partial W_{g,1}}) \to B\Emb^{\mathrm{sfr}, \cong}_{\partial/2}(W_{g,1} ; \ell_{1/2\partial W_{g,1}}) \to B\blEmb^{\mathrm{sfr}, \cong}_{\partial/2}(W_{g,1} ; \ell_{1/2\partial W_{g,1}})$$
are respectively $(n-1)$-connected (as $\Map_{\partial/2}(W_{g,1}, \Omega \SOn/\SOn(2n))$ is $(n-2)$-connected) and $(2n-4)$-connected (by Morlet's lemma of disjunction). We will exploit the corresponding map of fibration sequences
\begin{equation}\label{eq:FrToSfr}
\begin{gathered}
\xymatrix{
\Map_{\partial/2}(W_{g,1}, \SOn(2n)) \ar[r] \ar[d]& \Map_{\partial/2}(W_{g,1}, \mathrm{G})_J \ar[d]\\
B\Emb^{\fr, \cong}_{\partial/2}(W_{g,1} ; \ell_{1/2\partial W_{g,1}}) \ar[r] \ar[d]& B\blEmb^{\mathrm{sfr}, \cong}_{\partial/2}(W_{g,1} ; \ell_{1/2\partial W_{g,1}}) \ar[d]\\
B\Emb^{\cong}_{\partial/2}(W_{g,1}) \ar[r] & B\hAut^{\cong}_{\partial}(W_{g,1})
}
\end{gathered}
\end{equation}
and the analysis of framings of $W_{g,1}$ given in \cite{KR-WFram} and \cite[Section 3.2]{KR-WDisc}.

In the notation of those papers, we have $\check{\Lambda}_g^{\fr, \ell}  = \pi_1(B\Emb^{\fr, \cong}_{\partial/2}(W_{g,1} ; \ell_{1/2\partial W_{g,1}}), \ell)$ and ${\Lambda}_g^{\fr, [[\ell]]} = \mathrm{Im}(\check{\Lambda}_g^{\fr, \ell} \to \pi_1(B\Emb^{\cong}_{\partial/2}(W_{g,1})))$, and the image of these groups in $U_g(\Z, \Lambda_n)$ is denoted $G_g^{\fr, [[\ell]]}$. These groups are identified with $U_g(\Z, \Lambda_n^{\min})$  in \cite[Proposition 3.5]{KR-WFram}, which proves (i). 

For (ii) we revisit the proof of \cite[Lemma 3.10]{KR-WDisc}. The long exact sequence for the left-hand column of \eqref{eq:FrToSfr} gives an exact sequence
$$\pi_1(\Emb^{\cong}_{\partial/2}(W_{g,1})) \xto{\alpha} \mathrm{Hom}(H, \pi_{n+1}(\SOn(2n))) \xto{\beta}\check{\Lambda}_g^{\fr, \ell} \to {\Lambda}_g^{\fr, [[\ell]]} \to 1.$$
Combining Remark 2.4 and Section 3.4 of \cite{KR-WFram} gives an extension
$$0 \to \begin{cases}
0 & \text{$n$ odd}\\
\mathrm{Hom}(H, \Z/2) & \text{$n$ even}
\end{cases} \to{\Lambda}_g^{\fr, [[\ell]]} \to G_g^{\fr, [[\ell]]} \to 0.$$
Thus to prove (ii) it remains to show that the map $\alpha$ is surjective for $n \neq 6$, and has cokernel $\mathrm{Hom}(H, \Z/4)$ for $n=6$. For this we consider the portion of the map of long exact sequences given by the map of fibrations \eqref{eq:FrToSfr}, which has the form
\begin{equation*}
\xymatrix{
\pi_1(\Emb^{\cong}_{\partial/2}(W_{g,1})) \ar[r]^-{\alpha} \ar[d] & \mathrm{Hom}(H, \pi_{n+1}(\SOn(2n))) \ar[r]^-{\beta} \ar[d]^{J}& \check{\Lambda}_g^{\fr, \ell} \ar@{=}[d]\\
\pi_1(\hAut_\partial^{\cong}(W_{g,1})) \ar[r]^-{\gamma} & \mathrm{Hom}(H, \pi_{n+1}(\mathrm{G})) \ar[r]^-{\delta} & \pi_1(B\blEmb^{\mathrm{sfr}, \cong}_{\partial/2}(W_{g,1} ; \ell_{1/2\partial W_{g,1}}), \ell_{W_{g,1}}).
}
\end{equation*}

By Lemma \ref{lem:ConnectingMapNormalInv} the composition
$$\Sigma \pi_{2n+1}(S^n) \otimes H \cong \frac{\pi_{2n+1}(S^n)}{[\iota_n, \pi_{n+2}(S^n)]} \otimes H \to \pi_1(\hAut_\partial^{\cong}(W_{g,1})) \xto{\gamma} \mathrm{Hom}(H, \pi_{n+1}(\mathrm{G}))$$
is, after identifying $\mathrm{Hom}(H, \pi_{n+1}(\mathrm{G})) \cong \pi_{n+1}^s \otimes H$ using the duality on $H$, given by  the stabilisation map $\Sigma \pi_{2n+1}(S^n) \to \pi_{n+1}^s$ tensored with $H$. By Lemma \ref{lem:Epi} this stabilisation map is surjective for $n \neq 6$, and has cokernel $\Z/4$ for $n=6$. Thus $\beta=0$ for $n \neq 6$, and for $n=6$ we use that $J : \pi_7(\SOn(12)) \cong \pi_7(\SOn) \to \pi_7^s$ is onto to deduce that the image of $\beta$ is $\mathrm{Hom}(H, \Z/4)$.
\end{proof}

\section{Proofs of the main theorems}\label{sec:proofs}

\subsection{Strategy}\label{sec:Strategy}

Let $\ell$ be the stable framing from Lemma \ref{lem:HypBdyCond}, so that $(\pi_n(X_g),\lambda_X,q_X^{\ell})$ is a hyperbolic quadratic module. The stably framed Weiss sequence from Section \ref{sec:WeissFib} is a fibration sequence
\begin{equation*}
B\Diff_{\partial}^{\sfr}(X_g; \ell_{\partial X_g})_\ell \to B\Emb_{\partial/2}^{\cong, \sfr}(X_g ; \ell_{1/2\partial X_g})_\ell \to B \left(B\Diff_{\partial}^{\sfr}(S^1\times D^{2n-1}; \ell_{\partial})_L\right),
\end{equation*}
for some $L \subset \pi_0(B\Diff_{\partial}^{\sfr}(S^1\times D^{2n-1}; \ell_{\partial}))$ (the argument will not require anything about the structure of this group $L$). The base of this fibration has abelian fundamental group, so the fibration is plus-constructible by \cite[Theorem 1.1 (a)]{Berrick}. After plus-constructing it (always with respect to the maximal perfect subgroups) we may loop the fibration back up to give
\begin{equation}\label{eq:sfr-weiss-plus}
B\Diff_{\partial}^{\sfr}(S^1\times D^{2n-1}; \ell_{\partial})_L \to B\Diff_{\partial}^{\sfr}(X_g; \ell_{\partial X_g})_\ell^+ \to B\Emb_{\partial/2}^{\cong, \sfr}(X_g ; \ell_{1/2\partial X_g})_\ell^+.
\end{equation}
Our strategy will be to analyse the homotopy groups of the middle and rightmost spaces in degrees $< n-1$. The leftmost space differs only slightly from the space $B\Diff_{\partial}(S^1\times D^{2n-1})$ we wish to study, via the homotopy fibre sequence
\begin{equation}\label{eqn:sfr vs nosfr diff}
\Map_{\partial}(S^1\times D^{2n-1}, \SOn)\to B\Diff_{\partial}^{\sfr}(S^1\times D^{2n-1};\ell_{\partial})\to B\Diff_{\partial}(S^1\times D^{2n-1}).
\end{equation}

In Section \ref{sec:GRW} we will use the work of Galatius and the second-named author \cite{GRWActa}, to identify the middle space of \eqref{eq:sfr-weiss-plus} in the limit as $g \to \infty$ as the infinite loop spaces associated to certain Thom spectrum.

For the rightmost space of \eqref{eq:sfr-weiss-plus} we proceed as follows. It follows from Proposition \ref{prop:EMCG up to concordance} that the map
$$B\Emb^{\mathrm{sfr}, \cong}_{\partial/2}(X_g ; \ell_{1/2\partial X_g}) \to B\blEmb^{\mathrm{sfr}, \cong}_{\partial/2}(X_g ; \ell_{1/2\partial X_g})$$
is $(n-1)$-connected, so it remains so after taking the path-component of $\ell$ and plus-constructing. In Section \ref{sec:sfrMCG} we have described the fundamental group $\check{\Xi}_g^{\sfr,\ell}$ of the space $B\blEmb_{\partial/2}^{\cong, \sfr}(X_g ; \ell_{1/2\partial X_g})_\ell$, and in Section \ref{sec:higherpisfrEmb} we have described the higher homotopy groups in degrees $< n-1$.  Define the space $\Delta$ by the fibration
\begin{equation}\label{eq:Delta}
\Delta\to
\hocolim\limits_{g \to \infty}B\blEmb^{\mathrm{sfr}, \cong}_{\partial/2}(X_g ; \ell_{1/2\partial X_g})_\ell^+\to \hocolim\limits_{g \to \infty}(B\check{\Xi}_g^{\sfr,\ell})^+.
\end{equation} 
Let us write
$$\check{\Xi}_\infty^{\sfr,\ell} := \colim_{g \to \infty} \check{\Xi}_g^{\sfr,\ell} \quad \text{ and }\quad \Omega_\infty^{\min} := \colim_{g \to \infty} \Omega_g^{\min}.$$
We will study the homotopy groups of $(B\check{\Xi}_\infty^{\sfr,\ell})^+$ using  Proposition \ref{prop:sfrMCG}, which gives a surjection (for $g$ large enough) $\check{\Xi}_g^{\sfr,\ell} \to \Omega_g^{\min}$. By \eqref{eq:DefnOmegaMin} the space $(B\Omega_\infty^{\min})^+$ is closely related to the Hermitian $K$-theory of $\Z[\pi]$, and in Section \ref{sec:HermKThy} below we will explain how to access its homotopy groups.

For those of $\Delta$ we will use the following general lemma of Hurewicz flavour.
\begin{lem}\label{lem:PlusConst}
Let $X$ be a path-connected space with fundamental group $\Gamma$, $S \subset \Z$ be a multiplicative subset, and suppose that $\pi_i(X) \otimes \Z[S^{-1}]=0$ for $2 \leq i < d$. Let $\calP(\Gamma)$ denote the maximal perfect subgroup of $\Gamma$. Then the homotopy fibre $F$ of $X^+ \to B\Gamma^+$  is simply-connected, satisfies $\pi_i(F) \otimes \Z[S^{-1}]=0$ for $2 \leq i < d$, and there is a map
$$H_0(\mathcal{P}(\Gamma); \pi_d(X)) \to \pi_d(F)$$
which becomes an isomorphism on applying $-\otimes \Z[S^{-1}]$.
\end{lem}
\begin{proof}
Consider the fibration $X \to B\Gamma$, with fibre the universal cover $\widetilde{X}$. By performing fibrewise $\Z[S^{-1}]$-localisation \cite{MayFibrewise} of $X$ over $B\Gamma$ we may suppose without loss of generality that $\widetilde{X}$ is $(d-1)$-connected and that $\pi_d(\widetilde{X})\cong \pi_d(X)$ is a $\Z[S^{-1}]$-module. 

Form the pullback $X' \to B\mathcal{P}(\Gamma)$ of this fibration to the maximal perfect subgroup of $\Gamma$, whose fibre is still $\widetilde{X}$. The plus-construction of this fibration is the universal cover of $X^+ \to B\Gamma^+$, so consider the map of fibrations
\beq
\xymatrix{
\widetilde{X} \ar[r] \ar[d]& X' \ar[r] \ar[d]& B\mathcal{P}(\Gamma) \ar[d]\\
F \ar[r] & \widetilde{X^+} \ar[r]& \widetilde{B\Gamma^+}
}
\eeq
where the two rightmost vertical maps are plus-constructions with respect to the full fundamental group. A comparison of $\Z[S^{-1}]$-homology Serre spectral sequences, using that $\widetilde{B\Gamma^+}$ is simply-connected, then gives the result.
\end{proof}

\subsection{The diffeomorphism group of $X_g$}\label{sec:GRW}

In Section \ref{sec:WeissFib} we defined the tangential structure $\sfr_{2n} : B \to B\SOn(2n)$ as the pullback in the square \eqref{eq:sfr}, so $B\simeq \SOn/\SOn(2n)$ which is a $(2n-1)$-connected space. The manifold $X_g$ is parallelisable, thus it admits a $\sfr_{2n}$-structure $\ell: TX_g\to\sfr_{2n}^*\gamma_{2n}$, where $\gamma_{2n}\to B\SOn(2n)$ denotes the universal oriented $2n$-plane bundle. This $\sfr_{2n}$-structure induces a map $\ell: X_g\to B$ whose Moore--Postnikov decomposition has $n$th stage
\beq
X_g\xto{\ell'} S^1\times B\xto{u} B,
\eeq
where the first map is given by the generator of $H^1(X_g;\Z) \cong \Z$ and $\ell$, and the second map is projection to $B$. Writing $\theta := \sfr_{2n} \circ u : S^1 \times B \to B\SOn(2n)$, there is a forgetful map
$$B\Diff_{\partial}^{\theta}(X_g;\ell'_{\partial})_{\ell'} \to B\Diff_{\partial}^{\sfr}(X_g;\ell_{\partial})_{\ell},$$
which is an equivalence as the pair $(X_g, \partial X_g)$ is $(n-1)$-connected and the map $u : S^1 \times B \to B$ is $1$-co-connected (cf.\ \cite[Lemma 9.2]{GRWHomStabII} or \cite[Lemma 4.15]{GRW-handbook}). The Pontrjagin--Thom construction gives rise to a map
\begin{equation}\label{eq:PT-map}
B\Diff_{\partial}^{\theta}(X_g;\ell_{\partial}')_{\ell'}\to
\Omega_0^{\infty}\mathrm{MT}\theta
\end{equation}
which by \cite[Theorem 1.8]{GRWActa} becomes a homology equivalence after passing to the homotopy colimit as $g\to\infty$. Here $\mathrm{MT}\theta$ denotes the Madsen--Tillmann spectrum associated to the tangential structure $\theta$, i.e.\ the Thom spectrum of $-\theta^*\gamma_{2n}$. Thus after plus constructing there is a weak homotopy equivalence
\begin{equation}\label{eq:PTmap+}
\left(\hocolim_{g \to \infty}B\Diff^{\sfr}_{\partial}(X_g;\ell_{\partial})_\ell\right)^+\to\Omega_0^{\infty}\mathrm{MT}\theta.
\end{equation}
From this we extract the following consequences.

\begin{lem}\label{lem:GRW}\hfill
\begin{enumerate}[(i)]
\item For all $k$, the groups $\pi_k(\left(\hocolim_{g \to \infty}B\Diff^{\sfr}_{\partial}(X_g;\ell_{\partial})_\ell\right)^+)$ are finitely generated.
\item In degrees $k < 4n+3$ the groups $\pi_k(\left(\hocolim_{g \to \infty}B\Diff^{\sfr}_{\partial}(X_g;\ell_{\partial})_\ell\right)^+)\otimes\Q$ are isomorphic to $\Q$ for $k=1, 2n+3, 2n+4, 2n+7, 2n+8,\ldots$ and are trivial otherwise.
\end{enumerate}
\end{lem}
\begin{proof}
We use the equivalence \eqref{eq:PTmap+}. By definition the bundle $\theta^*\gamma_{2n}$ is stably trivial, so the Madsen--Tillmann spectrum is a suspension spectrum.
\beq
\mathrm{MT}\theta\simeq \Sigma^{\infty-2n}(S^1\times \SOn/\SOn(2n))_+
\eeq
and so $\Omega_0^{\infty}\mathrm{MT}\theta\simeq\Omega_0^{2n}Q((S^1\times \SOn/\SOn(2n))_+)$. Therefore 
\beq
\pi_k(\Omega_0^{\infty}\mathrm{MT}\theta)\cong \pi_{k+{2n}}^s((S^1\times \SOn/\SOn(2n))_+)
\eeq
These groups are finitely generated for all $k$, because the homology groups of $S^1\times \SOn/\SOn(2n)$ are; this shows part (i). 

For part (ii) we observe that
\beq
\pi_*^s((S^1\times \SOn/\SOn(2n))_+)\otimes\Q\cong H_*(S^1\times \SOn/\SOn(2n);\Q)\cong H_*(S^1;\Q)\otimes H_*(\SOn/\SOn(2n);\Q).
\eeq
Now the rational homology of the space $\SOn/\SOn(2n)$ can be obtained (by the Universal Coefficient Theorem) from its cohomology, which is given by
\beq
H^*(\SOn/\SOn(2n);\Q)\cong\Lambda[x_{4n+3},x_{4n+7,\ldots}]\otimes\Q[e]/(e^2)
\eeq
where the $x_i$'s have degree $i$, and $e$ has degree $2n$. This can be seen using the Eilenberg--Moore spectral sequence
\beq
E_2^{*,*}=\mathrm{Tor}^{*,*}_{H^*(B\SOn;\Q)}(\Q,H^*(B\SOn(2n);\Q))\Longrightarrow H^*(\SOn/\SOn(2n);\Q)
\eeq
of the fibration $\SOn/\SOn(2n)\to B\SOn(2n)\xto{\rho} B\SOn$ where $H^*(B\SOn;\Q) \cong \Q[p_1, p_2, \ldots]$ acts on $H^*(B\SOn(2n);\Q) \cong \Q[e, p_1, p_2, \ldots, p_{n-1}]$ through $\rho^*$, and on $\Q$ through the augmentation.
\end{proof}

\subsection{Hermitian $K$-theory}\label{sec:HermKThy}

We shall need to know something about the homology of the groups $U_g(\Z[\pi], \Lambda_n^{\min})$ in the limit as $g \to \infty$. Using that finitely-generated projective $\Z[t, t^{-1}]$-modules are stably free (e.g.\ by the Bass--Heller--Swan theorem), recent work of Hebestreit--Steimle \cite[Theorem A]{HebestreitSteimle} identifies the plus-construction $BU_\infty(\Z[\pi], \Lambda_n^{\min})^+$ with the basepoint component $\Omega^\infty_0 \mathrm{GW}(\Z[\pi], \Qoppa^{(-1)^n gq})$ of the infinite loop space of the Grothendieck--Witt spectrum of $\Z[\pi]$ with respect to a certain quadratic functor $\Qoppa^{(-1)^n gq}$, the ``(skew) genuine quadratic" functor. On the other hand, recent work of Calmès--Dotto--Harpaz--Hebestreit--Land--Moi--Nardin--Nikolaus--Steimle \cite{CDHHLMNNSII} describes this Grothendieck--Witt spectrum in terms of a fibration sequence
\begin{equation}\label{eq:CDHHLMNNS}
\mathrm{K}(\Z[\pi])_{hC_2} \to \mathrm{GW}(\Z[\pi], \Qoppa^{\pm gq}) \to \mathrm{L}(\Z[\pi], \qoppa^{\pm gq})
\end{equation}
where the left-hand term denotes the $C_2$-homotopy orbits for a certain involution on the $K$-theory spectrum $\mathrm{K}(\Z[\pi])$, and the right-hand term denotes an appropriate form of $L$-theory. 

\begin{lem}\label{lem:KThy}\mbox{}
\begin{enumerate}[(i)]
\item The fibration \eqref{eq:CDHHLMNNS} admits a splitting after inverting 2.

\item The groups $\pi_d(\mathrm{K}(\Z[\pi])_{hC_2}) \otimes \Q$ are $\Q$ for $d=0,1$ and zero otherwise.
\item The map
$$\Q = \pi_1(\mathrm{K}(\Z[\pi])_{hC_2})\otimes \Q \to \pi_1(\mathrm{GW}(\Z[\pi], \Qoppa^{(-1)^n gq}))\otimes \Q $$
may be split by the map which to a class represented by $A \in U_g(\Z[\pi], \Lambda_n^{\min})$ assigns half the exponent of $t$ in $\det(A) \in \Z[\pi]^\times \cong\{\pm t^i \,| \, i \in \Z\}$.

\item The groups $\pi_d(\mathrm{GW}(\Z[\pi], \Qoppa^{(-1)^n gq}))$ are finitely-generated.

\item The groups $\pi_d(\mathrm{L}(\Z[\pi], \Qoppa^{(-1)^n gq}))\otimes\Q$ are $\Q$ if $d+2n \equiv 0, 1 \mod 4$ and zero otherwise.
\end{enumerate}
\end{lem}
\begin{proof}
Item (i) is part of the Main Theorem of \cite{CDHHLMNNSII}: briefly, the forgetful map from Grothendieck-Witt theory to $K$-theory factors through $\mathrm{K}(\Z[\pi])^{hC_2}$, and the composition
$$\mathrm{K}(\Z[\pi])_{hC_2} \to \mathrm{GW}(\Z[\pi], \Qoppa^{\pm gq}) \to \mathrm{K}(\Z[\pi])^{hC_2}$$
is the norm map, which is split after inverting $2$ by the composite
$$\mathrm{K}(\Z[\pi])^{hC_2}\xto{\text{forget}}\mathrm{K}(\Z[\pi])\to\mathrm{K}(\Z[\pi])_{hC_2}.$$

For (ii), it follows from the Bass--Heller--Swan theorem \cite{BHS} that the algebraic $K$-theory assembly map $B\pi_+\wedge \mathrm{K}(\Z)\to \mathrm{K}(\Z[\pi])$ is an equivalence, and this map is also equivariant with respect to the involution which acts by the identity on $B\pi$ and is the usual involution on $\mathrm{K}(\Z)$. By work of Borel we have
$$K_d(\Z) \otimes \Q \cong \begin{cases}
\Q & d=0 \text{ and } d=5,9,13,\ldots\\
0 & \text{ otherwise}.
\end{cases}$$
The involution on these groups is known from Farrell--Hsiang \cite[Lemma 2.4]{farrell-hsiang}: it is trivial on $K_0(\Z)\otimes\Q$, and it acts by $-1$ on the higher rational $K$-groups. Thus we have
\beq
\pi_d(\mathrm{K}(\Z[\pi])_{hC_2})\otimes\Q\cong
\begin{cases}
\Q &\text{if}\ d=0,1\\
0  &\text{otherwise}. 
\end{cases}
\eeq

For (iii), the determinant map $\det :  U_g(\Z[\pi], \Lambda_n^{\min}) \to \Z[\pi]^\times $ is a homomorphism and commutes with stabilisation, so gives a map
$$\det : \pi_1(\mathrm{GW}(\Z[\pi], \Qoppa^{(-1)^n gq})) \to \Z[\pi]^\times \cong\{\pm t^i \,| \, i \in \Z\} \cong \Z^\times \times \Z.$$
The class in $\pi_1(\mathrm{K}(\Z[\pi]))$ represented by $(t) \in GL_1(\Z[\pi])$ has image in $\pi_1(\mathrm{GW}(\Z[\pi], \Qoppa^{(-1)^n gq}))$ represented by the unitary matrix $(t) \oplus (t^{-1})^\dagger$, whose determinant is $t^2$. This shows that the claimed map provides a rational splitting.

For (iv), we first observe that the homotopy groups of $\mathrm{K}(\Z)$ are finitely-generated by a theorem of Quillen \cite{QuillenFG}, and hence those of $\mathrm{K}(\Z)_{hC_2}$ are too: thus it suffices to show that those of $\mathrm{L}(\Z[\pi], \Qoppa^{\pm gq})$ are. We will do so by first expressing these homotopy groups in terms of those of $\mathrm{L}(\Z[\pi], \Qoppa^{gs})$. We apply Corollary R.10 of \cite{CDHHLMNNSIII} twice to identify $\mathrm{L}(\Z[\pi], \Qoppa^{gq}) \simeq S^4 \wedge \mathrm{L}(\Z[\pi], \Qoppa^{gs})$. For the skew case, we first apply the same result once to identify $\mathrm{L}(\Z[\pi], \Qoppa^{-gq}) \simeq S^2 \wedge \mathrm{L}(\Z[\pi], \Qoppa^{ge})$. Then, we observe that the group-ring involution on $\Z[\pi] \cong \Z[t, t^{-1}]$ satisfies $\widehat{H}^{-1}(C_2 ; \Z[\pi])=0$ and so by Remark R.4 of \cite{CDHHLMNNSIII} we have $\Qoppa^{ge} = \Qoppa^{gq}$, whence a further application of Corollary R.10 gives $\mathrm{L}(\Z[\pi], \Qoppa^{-gq}) \simeq S^6 \wedge \mathrm{L}(\Z[\pi], \Qoppa^{gs})$. Thus 
it is enough to show that the groups $\pi_d(\mathrm{L}(\Z[\pi], \Qoppa^{gs}))$ are finitely-generated for all $d \in \Z$. By Theorem 1.2.18 of \cite{CDHHLMNNSIII} these are the classical (non-periodic) symmetric $L$-groups $L^d(\Z[\pi])$ of Ranicki \cite{ranicki-foundations}. As finitely-generated projective $\Z[t, t^{-1}]$-modules are stably free, and the Whitehead group of $\pi \cong \langle t \rangle$ vanishes too, there is no difference between based, free, and projective $L$-theory. In particular Proposition 19.2 (ii) of \cite{ranicki-lower} (see also Theorem 4.1 of \cite{milgram-ranicki}) gives a Shaneson splitting
$$\pi_d(\mathrm{L}(\Z[\pi], \Qoppa^{gs})) \cong L^d(\Z[\pi]) \cong L^d(\Z) \oplus L^{d-1}(\Z).$$
Now we have (see \cite[Proposition 4.3.1]{RanickiLES}, \cite[p.\ 21]{HLN})
$$\pi_d(\mathrm{L}(\Z, \Qoppa^{gs})) \cong L^d(\Z) \cong \begin{cases}
\Z & d \equiv 0 \mod 4\\
\Z/2 & d \equiv 1 \mod 4 \text{ and } d > 0\\
\Z/2 & d \equiv 2 \mod 4 \text{ and } d < -4\\
0 & \text{ otherwise}.
\end{cases}$$
With the splitting above this shows the groups $\pi_d(\mathrm{L}(\Z[\pi], \Qoppa^{gs}))$ are finitely-generated, proving (iv); it also gives the calculation in (v) upon inverting $2$.
\end{proof}

\subsection{Cyclotomic structure}

The group $\mathrm{O}(2)$ acts smoothly on $S^1$ and hence on $S^1 \times D^{2n-1}$, and therefore it acts by conjugation on $\Diff_\partial(S^1 \times D^{2n-1})$ and hence on $B\Diff_\partial(S^1 \times D^{2n-1})$. For each cyclic group $C_d \leq \mathrm{SO}(2) \leq \mathrm{O}(2)$ we have an identification of the $C_d$-fixed points
$$B\Diff_\partial(S^1 \times D^{2n-1})^{C_d} \simeq B\Diff_\partial(S^1/C_d \times D^{2n-1})$$
which in turn may be identified with $B\Diff_\partial(S^1 \times D^{2n-1})$ via $[z] \mapsto z^d :S^1/C_d \xto{\sim} S^1$. This endows $B\Diff_\partial(S^1 \times D^{2n-1})$ with the structure of a cyclotomic space in the sense of \cite[Definition 4.3]{Schlichtkrull}. In particular there are commuting endomorphisms
$$F_d : B\Diff_\partial(S^1 \times D^{2n-1}) \to B\Diff_\partial(S^1 \times D^{2n-1})$$
corresponding under this identification to the inclusion of the $C_d$-fixed points. In terms of diffeomorphisms $F_d$ is given by associating to a diffeomorphism $\varphi$ of $S^1 \times D^{2n-1}$ the induced diffeomorphism of the $d$-fold cyclic cover $f_d \times D^{2n-1}: S^1 \times D^{2n-1} \to S^1 \times D^{2n-1}$. These maps endow the homotopy (or homology) groups of $B\Diff_\partial(S^1 \times D^{2n-1})$ with the structure of a $\Z[\mathbb{N}^\times]$-module.

\begin{deftn}
Say that a $\Z[\mathbb{N}^\times]$-module $M$ is \emph{tame} if the operation $F_d$ acts invertibly on $M \otimes \Z[d^{-1}]$ for all $d \in \mathbb{N}$.
\end{deftn}

\begin{lem}\label{lem:tame-criterion}\mbox{}
\begin{enumerate}[(i)]
\item Let $A$ be a finitely-generated abelian group and $F : A \to A$ be an endomorphism. If $F \otimes \Z[d^{-1}]$ is an epimorphism or a split monomorphism, then it is an isomorphism.

\item Let  $0 \to A \to B \to C \to 0$ be a short exact sequence of $\Z[\mathbb{N}^\times]$-modules.
\begin{enumerate}[(a)]
\item If $A$ and $C$ are tame then so is $B$.
\item If $B$ is tame and $C$ is a finitely-generated abelian group, then $A$ and $C$ are tame.
\end{enumerate}

\end{enumerate}
\end{lem}
\begin{proof}
For (i), we may as well consider the more general situation of a finitely-generated $\Z[d^{-1}]$-module $A'$ and a $\Z[d^{-1}]$-module endomorphism $F' : A' \to A'$. If $F'$ is an epimorphism, then by Nakayama's lemma it is an isomorphism. If $F'$ is a split monomorphism then we choose a splitting for it, which is an epimorphism, and apply the previous case: then use that a one-sided inverse to an isomorphism is an isomorphism.



For (ii), part (a) follows from the 5-lemma. For part (b) observe that as $B$ is tame the map $F_d \otimes \Z[d^{-1}] : C\otimes \Z[d^{-1}] \to C \otimes \Z[d^{-1}]$ is an epimorphism and so, by Nakayama's lemma, an isomorphism and so $C$ is tame. It then follows from the induced maps of short exact sequences that $F_d \otimes \Z[d^{-1}] : A\otimes \Z[d^{-1}] \to A \otimes \Z[d^{-1}]$ is also an isomorphism so $A$ is tame.
\end{proof}

We will now explain how analogous $\Z[\mathbb{N}^\times]$-module structures may be defined on the spaces in \eqref{eq:sfr-weiss-plus} and \eqref{eq:Delta}. The $d$-fold cyclic cover of $X_g$ may be identified with $X_{dg}$, so lifting homotopy equivalences, or diffeomorphisms, or embeddings, along this cover gives maps
\begin{align*}
F_d : B\hAut_\partial(X_g) &\to B\hAut_\partial(X_{dg})\\
F_d : B\Diff_\partial(X_g) &\to B\Diff_\partial(X_{dg})\\
F_d : B\Emb_{\partial/2}(X_g) &\to B\Emb_{\partial/2}(X_{dg}),
\end{align*}
fitting together to give a map of Weiss fibre sequences. Similarly with block embeddings and when equipped with stable framings. 

\begin{lem}\label{lem:PT-transfer}
The following diagram commutes up to homotopy
\beq
\xymatrix{
B\Diff_{\partial}^{\sfr}(X_g;\ell_{\partial X_g})_\ell \ar[d]_-{F_d} &B\Diff_{\partial}^{\theta}(X_g;\ell'_{\partial X_g})_{\ell'} \ar[l]_-{\sim}\ar[r]\ar[d]_-{F_d} & \Omega_0^{\infty}\mathrm{MT}\theta\ar[d]\\
B\Diff_{\partial}^{\sfr}(X_{dg};\ell_{\partial X_{dg}})_\ell & B\Diff_{\partial}^{\theta}(X_{dg};\ell'_{\partial X_{dg}})_{\ell'} \ar[l]_-{\sim}\ar[r] & \Omega_0^{\infty}\mathrm{MT}\theta
}
\eeq
where, under the identification $\mathrm{MT}\theta\simeq\Sigma^{\infty}S^1_+\wedge\Sigma^{\infty-2n}\SOn/\SOn(2n)_+$, the right vertical map is given by $\mathrm{trf}_{f_d}\wedge\Sigma^{\infty-2n}\SOn/\SOn(2n)_+$, whose first factor is the Becker--Gottlieb transfer for the cyclic $d$-fold covering $f_d:S^1\to S^1$.
\end{lem}

\begin{proof}
We use the description of $B\Diff_{\partial}^{\sfr}(X_g;\ell_{\partial X_g})_{\ell} \simeq B\Diff_{\partial}^{\theta}(X_g;\ell'_{\partial X_g})_{\ell'}$ as a space of \textit{fatly embedded} $\theta$-submanifolds \cite[Remark 1.10]{GRWActa}. 
Let $(W\subset (-\infty,0]\times\R^{q-1},\ell')$ be a fatly embedded $\theta$-submanifold. Let $f_d:S^1\to S^1$ be the cyclic $d$-fold covering, and $(f_d, h) : S^1\to S^1\times\R^2$ be an embedding representing it. The $\theta$-structure $ \ell'=(c, \ell) : W \to S^1 \times B$ includes the data of a classifying map $c : W \to S^1$ for the universal cover of $W$. Define $\widehat{W}$ to be the pull-back
\beq
\xymatrix{
\widehat{W}\ar[r]^-{p}\ar[d]^-{\hat{c}} & W\ar[d]^-{c}\\
S^1\ar[r]^-{f_d} & S^1
}
\eeq
The map $F_d$ takes $W$ to the fatly embedded $\theta$-submanifold given by the image of
$$\widehat{W} \xto{p \times \hat{c}} W \times S^1 \xto{\mathrm{inc} \times h} (-\infty,0]\times\R^{q-1} \times \R^2$$
equipped with the $\theta$-structure $\ell' \circ Dp : T\widehat{W} \to TW \to \theta^*\gamma_{2n}$. As both the map $\eqref{eq:PT-map}$ and the Becker--Gottlieb transfer are defined by means of the Pontrjagin--Thom construction, the commutativity of the diagram is then straightforward.
\end{proof}

Consider then the maps
\begin{equation*}
\xymatrix{
B\Diff_{\partial}^{\sfr}(X_g; \ell_{\partial X_g})_\ell^+ \ar[r] \ar[d]^-{F_d} & B\blEmb_{\partial/2}^{\cong, \sfr}(X_g ; \ell_{1/2\partial X_g})_\ell^+ \ar[d]^-{F_d} \ar[r] & (B\check{\Xi}_g^{\sfr,\ell})^+ \ar[r] \ar[d]^-{F_d} & (B\Omega_g^{\min})^+ \ar[d]^-{F_d}\\
B\Diff_{\partial}^{\sfr}(X_{dg}; \ell'_{\partial X_{dg}})_{\ell'}^+ \ar[r] & B\blEmb_{\partial/2}^{\cong, \sfr}(X_{dg} ; \ell'_{1/2\partial X_{dg}})_{\ell'}^+ \ar[r] & (B\check{\Xi}_{dg}^{\sfr,\ell})^+ \ar[r] & (B\Omega_{dg}^{\min})^+.
}
\end{equation*}

In the limit as $g \to \infty$ the left- and two rightmost vertical maps give endomorphisms
\begin{align*}
F_d : \Omega^\infty_0 \mathrm{MT}\theta &\to \Omega^\infty_0 \mathrm{MT}\theta\\
F_d : (B\check{\Xi}_{\infty}^{\sfr,\ell})^+ &\to (B\check{\Xi}_{\infty}^{\sfr,\ell})^+\\
F_d : (B\Omega_{\infty}^{\min})^+ &\to (B\Omega_{\infty}^{\min})^+.
\end{align*}

\begin{lem}\label{lem:tame}
These maps make the homotopy groups of $\Omega^\infty_0 \mathrm{MT}\theta$, $(B\check{\Xi}_{\infty}^{\sfr,\ell})^+$ and $(B\Omega_{\infty}^{\min})^+$ into tame $\Z[\mathbb{N}^\times]$-modules.
\end{lem}
\begin{proof}
Recall that $\mathrm{MT}\theta \simeq \Sigma^\infty S^1_+ \wedge \Sigma^{\infty -2n} \SOn/\SOn(2n)_+$. By Lemma \ref{lem:PT-transfer} the first of these maps is homotopic to $\mathrm{trf}_{f_d} \wedge \Sigma^{-2n} \SOn/\SOn(2n)_+$, whose first factor is the Becker--Gottlieb transfer for the cyclic $d$-fold covering $f_d : S^1 \to S^1$. Functoriality of the transfer then shows that the homotopy groups of $\Omega^\infty_0 \mathrm{MT}\theta$ are $\Z[\mathbb{N}^{\times}]$-modules. Furthermore, it follows from \cite[Theorem 5.5]{BG} that $f_d \circ \mathrm{trf}_{f_d}$ is an equivalence after inverting $d$, and $f_d$ is too, so $F_d$ induces a split monomorphism on homotopy groups after inverting $d$, and as the homotopy groups of $\mathrm{MT}\theta$ are finitely-generated Lemma \ref{lem:tame-criterion} (i) finishes the argument.

The classifying spaces of the groups $\check{\Xi}_{g}^{\sfr,\ell}$ and $\Omega_g^{\min}$ both assemble into $E_\infty$-algebras
$$\coprod_{g \geq 0} B\check{\Xi}_{g}^{\sfr,\ell} \quad\quad\text{ and }\quad\quad \coprod_{g \geq 0} B\Omega_g^{\min}$$
so that $(B\check{\Xi}_{\infty}^{\sfr,\ell})^+$ and $(B\Omega_{\infty}^{\min})^+$ are (the basepoint components of) their group-completions. In both cases the maps $F_d$ are defined before group-completion, and are given by multiplication by $d$ using the $E_\infty$-structure. Thus after group-completion these maps are given by multiplication by $d$ using the loop-space structure, so induce multiplication by $d$ on homotopy groups. Thus these homotopy groups are $\Z[\mathbb{N}^{\times}]$-modules, which are tame simply because multiplication by $d$ is an isomorphism after inverting $d$. 
\end{proof}

\subsection{$p$-local homology: Theorem \ref{thm:B}}

In this section we will prove Theorem \ref{thm:B}, formulated again below as Theorem \ref{thm:Tors}. For a prime number $p$ let $\mathscr{C}_p$ denote the Serre class of abelian groups $A$ such that $A\otimes\Z_{(p)}$ is finitely generated as a $\Z_{(p)}$-module. 

\begin{thm}\label{thm:Tors}
For $n \geq 3$, all primes $p$, and $k< \min(2p-3, n-2)$ we have
\beq
\pi_k(B\Diff_{\partial}(S^1\times D^{2n-1}))\in\mathscr{C}_p
\eeq
and if $2p-3 < n-2$ then there is a injective map
$$\bigoplus_{a> 0} \Z/p\{t^a - t^{-a}\} \to \pi_{2p-3}(B\Diff_{\partial}(S^1\times D^{2n-1}))$$
whose cokernel is in $\mathscr{C}_p$. If $p=2$ then the cokernel is a finitely-generated abelian group.
\end{thm}

The following lemma refers to the space $\Delta$ defined in \eqref{eq:Delta}.

\begin{lem}\label{lem:delta}
We have $\pi_i(\Delta)_{(p)}=0$ for $i < \min(2p-2,n-1)$ and if $2p-2 < n-1$ there is an injective map
$$\bigoplus_{a> 0} \Z/p\{t^a - t^{-a}\}\to \pi_{2p-2}(\Delta)$$
whose cokernel is $p$-locally trivial if $p$ is odd, and is $\Z/2$ if $p=2$. The $\Z[\mathbb{N}^\times]$-module structure on the left-hand term is given by the rule
\begin{equation}\label{eq:AlgFrob}
F_d(t^a - t^{-a}) := \begin{cases}
t^{a/d} - t^{-a/d} & \text{ if $d$ divides $a$}\\
0 & \text{ otherwise,}
\end{cases}
\end{equation}
and linearity.
\end{lem}
\begin{proof}
We apply Lemma \ref{lem:PlusConst} to the unstable analogue
$$\Delta_g \to B\blEmb^{\mathrm{sfr}, \cong}_{\partial/2}(X_{g} ; \ell_{1/2\partial X_{g}})_\ell^+ \to (B\check{\Xi}_g^{\sfr,\ell})^+$$
of the fibration sequence \eqref{eq:Delta}.

If $p$ is odd then by Example \ref{ex:EmbeddingsOddprimes} the groups $\pi_k(B\blEmb^{\mathrm{sfr}, \cong}_{\partial/2}(X_{g} ; \ell_{1/2\partial X_{g}}))_{(p)}$ vanish for $2 \leq k < \min(2p-2,n-1)$ and if $2p-2 < n-1$ we have
$$\pi_{2p-2}(B\blEmb^{\mathrm{sfr}, \cong}_{\partial/2}(X_g ; \ell_{1/2\partial X_g}))_{(p)} \cong \Z/p \otimes S^{(-1)^{n+1}}_X.$$
Thus by Lemma \ref{lem:PlusConst} we have $\pi_k(\Delta_g)_{(p)}=0$ for $1 \leq k < \min(2p-2,n-1)$, and if $2p-2 < n-1$ then the map
$$H_0(\check{\Xi}_g^{\sfr,\ell} ; \Z/p \otimes S^{(-1)^{n+1}}_X) \to \pi_{2p-2}(\Delta_g)_{(p)}$$
is an isomorphism. The action of $\check{\Xi}_g^{\sfr,\ell}$ on $\Z/p \otimes S^{(-1)^{n+1}}_X$ is via the surjection $\check{\Xi}_g^{\sfr,\ell} \to \Omega_g^{\min}$, so the coinvariants are the same as those for the latter group, which we compute in Lemma \ref{lem:coinvariants U_g} of Appendix \ref{sec:Coinv}. Taking colimits as $g \to \infty$ we find that $\pi_k(\Delta)_{(p)}=0$ for $1 \leq k < \min(2p-2,n-1)$, and if $2p-2 < n-1$ then
$$ \bigoplus_{a> 0} \Z/p\{t^a - t^{-a}\} \xto{\sim}\pi_{2p-2}(\Delta)_{(p)}.$$

If $p=2$ then we instead use the two extensions described in Example \ref{ex:EmbeddingsPi2}, using again that $\check{\Xi}_g^{\sfr,\ell}$ acts on these groups via the surjection $\check{\Xi}_g^{\sfr,\ell} \to \Omega_g^{\min}$ so that we may compute coinvariants for the latter group. The terms
$$\mathrm{Coker}(\Sigma \pi_{2n+2}(S^n) \to \pi_{n+2}^s) \otimes H \quad\text{ and }\quad \mathrm{Ker}(\Sigma \pi_{2n+1}(S^n) \to \pi_{n+1}^s) \otimes H$$
have trivial $\Omega_g^{\min}$-coinvariants by Lemma \ref{lem:coinvariantsH}, so by Lemma \ref{lem:coinvariants U_g} there is an isomorphism
$$H_0(\Omega_g^{\min} ; \pi_2(B\blEmb^{\mathrm{sfr}, \cong}_{\partial/2}(X_g ; \ell_{1/2\partial X_g}), \ell)) \cong \Z/2 \otimes \begin{cases}
\bigoplus_{a> 0} \Z\{t^a + (-1)^n t^{-a}\} \\
\bigoplus_{a> 0} \Z\{t^a - (-1)^n t^{-a}\} 
\end{cases} \oplus \Z/2 $$
with the same two cases as  in Example \ref{ex:EmbeddingsPi2}. In either case we can write the first term as $\bigoplus_{a> 0} \Z/2\{t^a - t^{-a}\}$, because it is tensored with $\Z/2$. As above, taking colimits as $g \to \infty$ we find an isomorphism
\begin{equation*}
\left(\bigoplus_{a> 0} \Z/2\{t^a - t^{-a}\}\right) \oplus \Z/2 \xto{\sim} \pi_2(\Delta)_{(2)}.
\end{equation*}

We determine the $\Z[\mathbb{N}^\times]$-module structure as follows. It is induced by the map on coinvariants of $S^{\pm}_X$. The map of $S^\pm_X$'s induced by the cyclic cover $X_{dg} \to X_g$ comes from the map on homotopy automorphisms, and in terms of our description of the homotopy groups of $\hAut_\partial(X_g)$ in Section \ref{sec:hAutHigherHty} it is induced by the isomorphism 
$$\tau_d: \pi_n(X_g) \to \pi_n(X_{dg})$$
inverse to that given by viewing $X_{dg}$ as a covering space of $X_g$. Identifying $\pi_1(X_{dg})$ as the subgroup of $\pi_1(X_g) = \langle t \rangle$ generated by $t^d$, $\tau_d$ is the map of $\Z[(t^d)^{\pm}]$-modules given in terms of the hyperbolic bases by
$$\tau_d(t^e a_i) = a_{i+eg} \quad\quad \tau_d(t^e b_i) = b_{i+eg}$$
for $0 \leq e < d$. Now the coinvariant $t^a-t^{-a} \in \Z[\pi]$ of $S^{(-1)^{n+1}}_X$ may be represented by $\xi := t^a a_1 \otimes b_1 - (-1)^{n} t^{-a} b_1 \otimes a_1 \in S^{(-1)^{n+1}}_X$. If $d$ does not divide $a$, so $a = kd+r$ with $0 < r < d$, then by the formula for $\tau_d$ above the first term of $\xi$ is sent to $(t^{d})^k a_{1+rg} \otimes b_1$, which vanishes under the equivariant intersection form, as $1+rg \neq 1$. The second term vanishes for a similar reason, so $F_d(t^a-t^{-a})=0$. On the other hand if $a=kd$ then $\xi$ is sent to
$$(t^{d})^k a_{1} \otimes b_1 - (-1)^{n} (t^{d})^{-k} b_{1} \otimes a_1$$
which under the equivariant intersection form goes to $(t^{d})^k - (t^{d})^{-k}$. Bearing in mind that the generator of the fundamental group of $X_{dg}$ is $t^d$ for the purposes of this calculation, identifying it with $t$ this element corresponds to $t^k - t^{-k}$. Thus the map is indeed given by the formula \eqref{eq:AlgFrob}.
\end{proof}

\begin{lem}\label{lem:epsilon}\mbox{}
\begin{enumerate}[(i)]
\item The map 
$\pi_i((B\check{\Xi}_{\infty}^{\sfr,\ell})^+)\otimes\Z[\tfrac{1}{2}]\to \pi_i((B\Omega_{\infty}^{\min})^+)\otimes\Z[\tfrac{1}{2}]$
is an isomorphism for all $i$, and
\item for $i\leq 2$, the groups $\pi_i((B\check{\Xi}_{\infty}^{\sfr,\ell})^+)$ are finitely-generated abelian groups.
\end{enumerate}
\end{lem}
\begin{proof}
For (i), the description of $\check{\Xi}_g^{\sfr,\ell}$ in Proposition \ref{prop:sfrMCG} shows that the kernel of its map to $\Omega_g^{\min}$ is a finite group of order a power of 2. Thus the map $(B\check{\Xi}_{\infty}^{\sfr,\ell})^+\to (B\Omega_{\infty}^{\min})^+$ is an isomorphism on $H_*(-;\Z[\tfrac{1}{2}])$, and as these spaces are loop-spaces it is then also an isomorphism on $\pi_*(-)\otimes \Z[\tfrac{1}{2}]$.

For (ii), it is enough to show that the groups $\colim_{g \to \infty}H_i(\check{\Xi}_g^{\sfr,\ell};\Z)$ are finitely-generated for $i \leq 2$. 

We first claim that these groups have homological stability. By the description of $\check{\Xi}_g^{\sfr,\ell}$ in Proposition \ref{prop:sfrMCG}, and using Shapiro's lemma, there is a natural spectral sequence
$$E^2_{p,q} = H_p(U_g(\Z[\pi], \Lambda_n^{\min}) ; H_i(L_g^{\sfr, \ell};\Z) \otimes \Z[bP_{2n}]) \Rightarrow H_{p+q}(\check{\Xi}_g^{\sfr,\ell};\Z).$$
The $H_i(L_g^{\sfr, \ell};\Z)$ are polynomial coefficient systems, and $\Z[bP_{2n}]$ is an abelian coefficient system, so in total this is the kind of coefficient system to which the second part of Theorem 4.20 of \cite{RWW} applies. As described in Section 5.4 of \cite{RWW} that theorem applies in this case as long as the ring $\Z[\pi]$ has finite ``unitary stable rank", which it does by \cite[Theorem 7.3]{CrowleySixt}.

It therefore suffices to show that the groups $H_i(\check{\Xi}_g^{\sfr,\ell};\Z)$ are finitely-generated for $i \leq 2$ and all large enough $g$. The Serre spectral sequence  for the extension in Proposition \ref{prop:sfrMCG} shows that $H_1(\check{\Xi}_g^{\sfr,\ell};\Z)$ is assembled out of subquotients of
$$H_0(\Omega_g^{\min} ; H_1(L_g^{\sfr,\ell};\Z))\text{ and } H_1(\Omega_g^{\min};\Z),$$
and that $H_2(\check{\Xi}_g^{\sfr,\ell};\Z)$ is assembled out of subquotients of
$$H_0(\Omega_g^{\min} ; H_2(L_g^{\sfr,\ell};\Z)), H_1(\Omega_g^{\min} ; H_1(L_g^{\sfr,\ell};\Z)) \text{ and } H_2(\Omega_g^{\min};\Z).$$
Of these, the first terms are finitely-generated as $L_g^{\sfr, \ell}$ is. Since $\Omega_g^{\min}$ is a finite index subgroup of $U_g(\Z[\pi],\Lambda_n^{\min})$, the last terms are finitely-generated by Lemma \ref{lem:KThy} (iv) (and homological stability for the groups $\Omega_g^{\min}$, parallel to the above). The remaining term $H_1(\Omega_g^{\min} ; H_1(L_g^{\sfr,\ell};\Z))$ is finitely-generated because in addition the group $\Omega_g^{\min}$ is finitely-generated for large enough $g$. This may be seen by combining the facts that:
\begin{enumerate}[(i)]
\item the subgroup $EU_g(\Z[\pi], \Lambda_n^{\min})$ of elementary unitary matrices (cf.\ \cite[\S 5.3A]{HO}) is finitely-generated by \cite[Theorem 9.2.9]{HO}, 
\item we have $U_g(\Z[\pi], \Lambda_n^{\min})/EU_g(\Z[\pi], \Lambda_n^{\min}) 
\cong H_1(U_g(\Z[\pi], \Lambda_n^{\min});\Z)$ for $g \geq 3$ by \cite[5.4.2 and 5.4.3]{HO}, which in turn agrees with $\pi_1(GW(\Z[\pi], \Qoppa^{gq}_{\pm}))$ for large enough $g$ by homological stability, and this is finitely-generated by Lemma \ref{lem:KThy} (iv),
\item $\Omega_g^{\min} \leq U_g(\Z[\pi], \Lambda_n^{\min})$ has finite index.\qedhere
\end{enumerate}
\end{proof}

\begin{proof}[Proof of Theorem \ref{thm:Tors}]
We will use \eqref{eq:sfr-weiss-plus}, \eqref{eqn:sfr vs nosfr diff} and \eqref{eq:Delta}. We suppose that $2p-2 < n-1$, leaving the reader to keep track of the inequalities necessary for the general argument.

Suppose first that $p$ is odd. From Lemma \ref{lem:delta} we have that $\pi_i(\Delta)_{(p)}=0$ for $i < 2p-2$, and that there is an isomorphism
$$\bigoplus_{a> 0} \Z/p\{t^a - t^{-a}\} \xto{\sim}\pi_{2p-2}(\Delta)_{(p)}.$$
Furthermore, the cyclotomic structure on this group is given by \eqref{eq:AlgFrob}. As this has no tame submodules, and the homotopy groups of $\hocolim_{g \to \infty}(B\check{\Xi}_g^{\sfr,\ell})^+$ are finitely-generated after $p$-localisation (by Lemma \ref{lem:epsilon} (i) and Lemma \ref{lem:KThy} (iv)) 
and tame (by Lemma \ref{lem:tame}), it follows from \eqref{eq:Delta}  that the induced map
$$\bigoplus_{a> 0} \Z/p\{t^a - t^{-a}\} \to \pi_{2p-2}\left(\hocolim\limits_{g \to \infty}B\blEmb^{\mathrm{sfr}, \cong}_{\partial/2}(X_g ; \ell_{1/2\partial X_g})_\ell^+\right)_{(p)}$$
is injective and has finitely-generated and tame cokernel, and that the lower homotopy groups of the target are finitely-generated and tame. Using Proposition \ref{prop:EMCG up to concordance} the same conclusion holds with block-embeddings replaced by embeddings, again using $2p-2 < n-1$.

 The same argument using the stabilisation of the fibration sequence \eqref{eq:sfr-weiss-plus} and the fact that the homotopy groups of $\hocolim_{g \to \infty}B\Diff_{\partial}^{\sfr}(X_g; \ell_{\partial X_g})_\ell^+ \simeq \Omega^\infty_0 \mathrm{MT}\theta$ are finitely-generated (by Lemma \ref{lem:GRW} (i)) and tame (by Lemma \ref{lem:tame}) shows that the composition of the map above with the connecting homomorphism
$$\partial : \pi_{2p-2}\left(\hocolim\limits_{g \to \infty}B\Emb^{\mathrm{sfr}, \cong}_{\partial/2}(X_g ; \ell_{1/2\partial X_g})_\ell^+\right)_{(p)} \to \pi_{2p-3}(B\Diff_{\partial}^{\sfr}(S^1\times D^{2n-1}; \ell_{\partial_0})_L)_{(p)}$$
is still injective and has cokernel, and all lower homotopy groups, finitely-generated and tame. Finally, the same argument with \eqref{eqn:sfr vs nosfr diff} using that the homotopy groups of $\Map_{\partial}(S^1\times D^{2n-1},O)$ are also finitely-generated and tame (with respect to the evident cyclotomic structure given by taking cyclic covers) gives the claimed result.

If $p=2$ then we modify this argument slightly as follows. We have an exact sequence 
\beq
\pi_3((B\check{\Xi}_\infty^{\sfr,\ell})^+) \to \pi_2(\Delta)\to\pi_{2}\left(\hocolim\limits_{g \to \infty}B\Emb^{\mathrm{sfr}, \cong}_{\partial/2}(X_g ; \ell_{1/2\partial X_g})_\ell^+\right)\to\pi_2((B\check{\Xi}_\infty^{\sfr,\ell})^+) 
\eeq
whose leftmost term is tame by Lemma \ref{lem:tame}. By Lemma \ref{lem:delta} we have a map $\bigoplus_{a> 0} \Z/2\{t^a - t^{-a}\} \to \pi_{2}(\Delta)$ which is injective with finitely-generated and tame cokernel (in fact the cokernel is $\Z/2$). The group $\pi_2((B\check{\Xi}_\infty^{\sfr,\ell})^+)$ is tame (by Lemma \ref{lem:tame}) and finitely-generated (by Lemma \ref{lem:epsilon} (ii)). It then follows that the induced map
$$\bigoplus_{a> 0} \Z/2\{t^a - t^{-a}\} \to \pi_{2}\left(\hocolim\limits_{g \to \infty}B\blEmb^{\mathrm{sfr}, \cong}_{\partial/2}(X_g ; \ell_{1/2\partial X_g})_\ell^+\right)$$
is injective and has finitely-generated and tame cokernel. From here we proceed identically to the case $p$ odd.
\end{proof}


\subsection{Rational homology: Theorem \ref{thm:A}}
 
In this section we shall prove Theorem \ref{thm:A}, namely that the rational homotopy groups of $B\Diff_{\partial}(S^1\times D^{2n-1})$ vanish in degrees $* < n-2$. The strategy will be to show that the other two spaces in the homotopy fibre sequence \eqref{eqn:sfr vs nosfr diff} have the same rational homotopy groups in this range: we start by calculating those of $B\Diff_{\partial}^{\sfr}(S^1\times D^{2n-1};\ell_{\partial})$.

\begin{lem}\label{lem:SfrRatCalc}
For $n\geq 3$ and $0 < k<n-2$ we have
\begin{equation*}
\pi_k(B\Diff^{\sfr}_{\partial}(S^1\times D^{2n-1};\ell_{\partial}))\otimes\Q \cong \begin{cases}
\Q & 2n+k \equiv 0,3 \mod 4\\
0 & \text{ else}.
\end{cases}
\end{equation*}
\end{lem}
\begin{proof}
We follow the strategy outlined in Section \ref{sec:Strategy}, using the plus-constructed fibration \eqref{eq:sfr-weiss-plus} stabilised in the limit $g \to \infty$.

By Example \ref{ex:EmbeddingsOddprimes}, in degrees $2 \leq k < n-1$ we have $\pi_k(B\blEmb^{\mathrm{sfr}, \cong}_{\partial/2}(X_g ; \ell_{1/2\partial X_g}), \ell) \otimes \Q =0$, and by an application of Lemma \ref{lem:PlusConst} we have $\pi_i(\Delta) \otimes \Q=0$ for $i < n-1$. Combined with Proposition \ref{prop:EMCG up to concordance} and Lemma \ref {lem:epsilon} (i) it follows that the map
$$B\Emb^{\mathrm{sfr}, \cong}_{\partial/2}(X_g ; \ell_{1/2\partial X_g})_\ell^+ \to (B\Omega_g^{\min})^+$$
is an isomorphism on rationalised homotopy groups in degrees $< n-1$.

By definition of $\Omega_g^{\min}$, after stabilising and plus-constructing there is a fibration sequence of infinite loop spaces
$$ (B\Omega_\infty^{\min})^+ \to BU_\infty(\Z[\pi], \Lambda_n^{\min})^+ \to BbP_{2n},$$
so as $bP_{2n}$ is finite the left-hand map induces an isomorphism on rationalised homotopy groups. Thus by Lemma \ref{lem:KThy} we have the calculation
$$\pi_k\left((B\Omega_\infty^{\min})^+\right) \otimes \Q \cong \begin{cases}
\Q & k=1\\
0 & \text{ else}
\end{cases} \oplus \begin{cases}
\Q & 2n+k \equiv 0,1 \mod 4\\
0 & \text{ else},
\end{cases}$$
where the first summand is given, as described in Lemma \ref{lem:KThy} (iii), by the determinant. 

On the other hand, by Lemma \ref{lem:GRW} (ii) we have
$$\pi_k(\hocolim_{g \to \infty}B\Diff^{\sfr}_\partial(X_g; \ell_{\partial X_g})_\ell^+) \otimes \Q \cong \begin{cases}
\Q & k=1\\
0 & \text{ else}\end{cases}$$
in degrees $0< * < 2n+3$.

To complete the argument we claim that the map
$$\pi_1(B\Diff^{\sfr}_\partial(X_g; \ell_{\partial X_g})_\ell^+) \to \pi_1((B\Omega_g^{\min})^+)$$
is stably and rationally injective, whence the long exact sequence on rational homotopy groups for the fibration \eqref{eq:sfr-weiss-plus} gives the claim in the statement of this lemma. To show this it is enough to exhibit a stably framed diffeomorphism whose induced map on $\pi_n(X_g)$ has determinant of infinite order in $\Z[\pi]^\times$. Such a diffeomorphism is obtained by ``dragging a $W_{1,1}$ around a loop in $X_g$''. More precisely, restrict the stable framing $\ell$ to $X_{g-1}\subset X_g$ and let $Y$ denote the homotopy fibre of the map
$$B\Diff_{\partial}^{\sfr}(X_{g-1}\setminus\mathrm{int} (D^{2n});\ell_{\partial X_{g}})_{\ell}\to B\Diff_{\partial}^{\sfr}(X_{g-1};\ell_{\partial X_{g}})_{\ell}.$$
By comparing stably framed diffeomorphisms with diffeomorphisms, we see that $Y$ is equivalent to $\Emb^{\sfr}(D^{2n}, X_{g-1})$ so there is a fibration $\Omega (\SOn/\SOn(2n)) \to Y \to X_{g-1}$. In particular $\pi_1(Y) \cong \pi_1(X_{g-1}) = \langle t \rangle$. Form the composition
\beq
Y\to B\Diff_{\partial}^{\sfr}(X_{g-1}\setminus\mathrm{int} (D^{2n});\ell_{\partial X_{g}})_{\ell}\to B\Diff_{\partial}^{\sfr}(X_{g};\ell_{\partial X_{g}})_{\ell}
\eeq
where the last map is extension by the identity on $W_{1,1}$, giving on fundamental groups 
\beq
\langle t \rangle \cong \pi_1(Y) \to \pi_1(B\Diff_{\partial}^{\sfr}(X_{g};\ell_{\partial X_{g}})_{\ell}).
\eeq
Under this map, the generator $t$ is sent to a stably framed diffeomorphism which induces the automorphism of $\pi_n(X_g)$ represented by the matrix $\id_{2g-2}\oplus(t)\oplus (t)$, whose determinant is $t^2\in\Z[\pi]^{\times}$.
\end{proof}

\begin{proof}[Proof of Theorem \ref{thm:A}]
We consider the homotopy fibre sequence \eqref{eqn:sfr vs nosfr diff}. As the infinite orthogonal group $\SOn$ is an infinite loop space, we easily calculate $\pi_k(\Map_\partial(S^1\times D^{2n-1},\SOn)) \cong \pi_{2n+k}(\SOn) \oplus \pi_{2n-1+k}(\SOn)$ and hence see that when rationalised these groups are 1-dimensional if $2n+ k \equiv 0,3 \mod 4$ and zero otherwise, i.e.\
\begin{equation*}
\pi_k(\Map_\partial(S^1\times D^{2n-1},\SOn))\otimes\Q \cong \begin{cases}
\Q & 2n+k \equiv 0,3 \mod 4\\
0 & \text{ else}.
\end{cases}
\end{equation*}
This agrees with the calculation in Lemma \ref{lem:SfrRatCalc}, so we will be finished if we show that for $n\geq 3$ and $k \geq 0$ the connecting map
\begin{equation*}
\pi_k(\Diff_{\partial}(S^1\times D^{2n-1}))\otimes\Q \to \pi_k(\Map_\partial(S^1\times D^{2n-1},\SOn))\otimes\Q
\end{equation*}
is zero.

This map is induced by the action of $\Diff_{\partial}(S^1\times D^{2n-1})$ on $\Bun(T(S^1\times D^{2n-1}),\sfr^*\gamma_{2n};\ell_{\partial})$ by precomposition with the derivative of a diffeomorphism. Thus we have a commutative diagram
\beq
\xymatrix{
\Diff_{\partial}(S^1\times D^{2n-1})\ar[r]\ar[d] &  \Map_{\partial}(S^1\times D^{2n-1},\SOn)\ar[d]^-{\simeq}\\
\blDiff_{\partial}(S^1\times D^{2n-1})\ar[r]\ar[d] &  \widetilde{\Map}_{\partial}(S^1\times D^{2n-1},\SOn)\ar[d]\\
\blTop_{\partial}(S^1\times D^{2n-1})\ar[r] &  \widetilde{\Map}_{\partial}(S^1\times D^{2n-1},\mathrm{STop})
}
\eeq
where the horizontal arrows are induced by the action through the derivative map. The bottom left corner is contractible by \cite[Corollary 2.3]{burghelea-blockTop} and the second right-hand vertical map is a rational equivalence as $\mathrm{STop}/\mathrm{SO}$ is rationally contractible.
\end{proof}

\subsection{Homeomorphism groups}\label{sec:Homeo}

For $2n \geq 6$ smoothing theory provides a fibration
$$\Map_{\partial}(S^1\times D^{2n-1},\tfrac{\Top(2n)}{\mathrm{O}(2n)}) \to B\Diff_\partial(S^1 \times D^{2n-1}) \to B\Top_\partial(S^1 \times D^{2n-1}).$$

Kupers \cite{kupers-finiteness} has shown that the homotopy groups of $\mathrm{Top}(2n)/\On(2n)$ are finitely-generated for $2n \geq 6$, so the homotopy groups of $\Map_{\partial}(S^1\times D^{2n-1}, \mathrm{Top}(2n)/\On(2n))$ are finitely-generated and tame with respect to the cyclotomic structure given by taking cyclic covers. Thus Theorem \ref{thm:B} holds with diffeomorphisms replaced by homeomorphisms.

On the other hand, the rational homotopy groups of $\mathrm{Top}(2n)/\On(2n)$ vanish in degrees $* \leq 4n-2$ by \cite[Theorem A]{KR-WDisc}, so those of $\Map_{\partial}(S^1\times D^{2n-1}, \mathrm{Top}(2n)/\On(2n))$ vanish in degrees $* \leq 2n-2$. Thus the statement of Theorem \ref{thm:A} also holds with diffeomorphisms replaced by homeomorphisms.

\appendix
\section{Some recollections on metastable homotopy theory}\label{sec:metastable}

We collect some tools from metastable homotopy theory, which were used in Sections \ref{sec:hAut-MCG} and \ref{sec:hAutHigherHty}. For definitions and properties of Whitehead products we follow \cite{WhiteheadBook}.

\subsection{The Hilton--Milnor theorem}\label{sec:Hilton-Milnor}
As $W_{g,1}\simeq\bigvee^{2g}S^n$ and $X_g\simeq S^1\vee\bigvee^{2g}S^n$, the Hilton--Milnor theorem (\cite[\S XI.6]{WhiteheadBook}) provides, for all $0\leq k<n-1$, an identification
\begin{equation}\label{eq:Hilton-W_g}
\pi_{2n+k-1}(W_{g,1}) \cong \bigoplus_{\substack{1 \leq i \leq 2g}} \pi_{2n+k-1}(S^n)\{x_i\} \oplus \bigoplus_{\substack{1 \leq i , j \leq 2g\\i < j}} \pi_{2n+k-1}(S^{2n-1})\{x_i \otimes x_j\}.
\end{equation}
The first term is given by $x_i \circ f$ for $f \in \pi_{2n+k-1}(S^n)$, and the second is given by the Whitehead products $[x_i, x_j]\circ g$ for $i < j$, and $g\in\pi_{2n+k-1}(S^{2n-1})$.

Similarly
\begin{equation}\label{eq:Hilton-X_g}
\pi_{2n+k-1}(X_{g}) \cong \bigoplus_{\substack{1 \leq i \leq 2g\\a \in \Z}}  \pi_{2n+k-1}(S^n) \{t^a x_i\}\oplus \bigoplus_{\substack{1 \leq i , j \leq 2g\\a,b\in\Z\\(a,i) < (b,j)}} \pi_{2n+k-1}(S^{2n-1})\{t^a x_i \otimes t^b x_j\},
\end{equation}
where the first term is given by $t^a x_i \circ f$ for $f \in \pi_{2n+k-1}(S^n)$, and the second is given by $[t^a x_i, t^b x_j] \circ g$ for $(a,i) < (b,j)$ and $g \in \pi_{2n+k-1}(S^{2n-1})$.

\subsection{The composition formula for Whitehead products}\label{sec:WhiteheadComposition}
We often use the following special case of the composition formula for Whitehead products: if $\alpha \in \pi_{p}(X)$, $\beta \in \pi_q(X)$, and $\gamma \in \Sigma\pi_{r-1}(S^{q-1}) \subset \pi_r(S^q)$, then
\begin{equation}\label{eq:WhiteheadCompFormula}
[\alpha, \beta \circ \gamma] = [\alpha, \beta] \circ \Sigma^{p-1} \gamma \in \pi_{p+r-1}(X).
\end{equation}
This may be found as \cite[Theorem X.8.18]{WhiteheadBook}.

\subsection{The metastable EHP sequence}\label{sec:EHP}

 The stabilisation map $E: S^n \to \Omega S^{n+1}$ has a homotopy fibre $F$, and considering the Serre spectral sequence for this fibration one sees that there is a $(3n-2)$-connected map $S^{2n-1} \to F$. One then considers the resulting maps
$$S^{2n-1} \xto{P} S^n \xto{E} \Omega S^{n+1}$$
as providing a homotopy fibre sequence in degrees $< 3n-2$, giving an exact sequence
$$\pi_i(S^{2n-1}) \xto{P_*} \pi_i(S^n) \xto{E_*} \pi_{i+1}(S^{n+1}) \xto{H_*} \pi_{i-1}(S^{2n-1}) \xto{P_*} \pi_{i-1}(S^n) \to \cdots$$
for $i \leq 3n-2$. The map $P$ represents the Whitehead square $[\iota_n, \iota_n]$, so the induced map on homotopy groups is $f \mapsto [\iota_n, \iota_n] \circ f : \pi_i(S^{2n-1}) \to \pi_{i}(S^n)$. As above, if $i < 3n-2$ then using that $\pi_i(S^{2n-1})$ consists of $n$-fold suspensions the identity \eqref{eq:WhiteheadCompFormula} means that this may be written as $P_*(\Sigma^n g) = [\iota_n, \Sigma g]$ (cf.\ \cite[Corollary XII.2.5]{WhiteheadBook}). The connecting map $H_* : \pi_{i+1}(S^{n+1}) \to \pi_{i-1}(S^{2n-1})$ is the Hopf invariant.

\begin{lem}\label{lem:Ord2}
Suppose $i < 2n-1$. If $n$ is even then the group
$$\mathrm{Ker}([\iota_n, -] : \pi_{i}(S^n) \to \pi_{n+i-1}(S^n))$$
consists of elements of order 2, and if $n$ is odd  it contains all elements of $\pi_{i}(S^n)$ divisible by $2$.
\end{lem}
\begin{proof}
If $i < 2n-1$ then every element $f \in \pi_i(S^n)$ is a suspension, so we have $[\iota_n, f] = P_*(\Sigma^{n-1} f)$. Furthermore, under this condition $E_*^{n-1} : \pi_i(S^n) \to \pi_{i+n-1}(S^{2n-1})$ is an isomorphism, so the map in question is identified with $P_*(-) = [\iota_n, \iota_n] \circ - : \pi_{i+n-1}(S^{2n-1}) \to \pi_{i+n-1}(S^{n})$.

If $n$ is odd then $[\iota_n, \iota_n]$ has order 2, so all elements divisible by 2 lie in the kernel of $[\iota_n, \iota_n] \circ -$. If $n$ is even then we consider the composition
$$ \pi_{i+n-1}(S^{2n-1}) \xto{P_*} \pi_{i+n-1}(S^{n}) \xto{H_*}  \pi_{i+n-3}(S^{2n-3}).$$
When $i=n$ the composition is the Hopf invariant of $[\iota_n, \iota_n]$, which is $2$, and as the outer groups are in the stable range this map is identified with multiplication by 2 on $\pi_{i-n}^s$. Thus an element in the kernel of $P_*$ must have order 2.
\end{proof}

\begin{lem}\label{lem:Hopf}
We have
$$\mathrm{Ker}([\iota_n, -] : \pi_{n+1}(S^n) \to \pi_{2n}(S^n)) \cong \begin{cases}
\Z/2\{\eta\} & n \equiv 3 \mod 4 \text{ or } n=2,6\\
0 & \text{ otherwise}
\end{cases}$$
and
$$\mathrm{Coker}([\iota_n, -] : \pi_{n+2}(S^n) \to \pi_{2n+1}(S^n)) \cong \Sigma \pi_{2n+1}(S^n).$$
\end{lem}
\begin{proof}
Under the suspension isomorphisms $E^{n-1} : \pi_{n+i}(S^n) \to \pi_{2n-1+i}(S^{2n-1})$ these are the maps $\pi_{2n}(P)$ and $\pi_{2n+1}(P)$ in the metastable EHP sequence.

The first map has domain $\Z/2\{\eta_n\}$, so the question is whether $[\iota_n, \eta_n]$ vanishes. Hilton \cite[p.\ 232]{Hilton} has shown that $[\iota_n, \eta_n]  = 0$ for $n=2, 6$ or for $n \equiv 3 \mod 4$, and that $[\iota_n, \eta_n]  \neq 0$ for $n \equiv 0,1 \mod 4$. The remaining case $n \equiv 2 \mod 4$, as well as many of the others, is covered by work of Mahowald \cite{MahowaldWh} which  shows that in this case $[\iota_n, \eta_n]  \neq 0$.

For the second map the portion
$$\pi_{n+2}(S^n) \xto{P_*} \pi_{2n+1}(S^n) \xto{E_*} \pi_{2n+2}(S^{n+1})$$
 of the metastable EHP sequence shows that the cokernel of $P_*$ is identified with the image $\Sigma \pi_{2n+1}(S^n) \subset \pi_{2n+2}(S^{n+1})$ of the suspension map. 
\end{proof}

\subsection{The quadratic approximation}\label{sec:QuadApprox}

Goodwillie's calculus of homotopy functors applied to the identity functor from pointed spaces to pointed spaces has second stage $P_2(X)$. The space $P_2(X)$ may be described as the homotopy fibre of the stable Hopf--James map
$$j_2(X) : Q(X) \to Q(D_2(X)),$$
where $D_2(X) = (X \wedge X)_{h \mathfrak{S}_2}$\footnote{This seems to be folklore. It is mentioned in \cite{arone-mahowald} and \cite{GoodwillieTaylor}. It appears stated as a theorem with a hint for its proof in \cite[Theorem 4.2.1]{behrens}.}. As the $k$-th derivative of the identity functor is a wedge of $(1-k)$-spheres \cite[Section 8]{GoodwillieTaylor}, for $X=S^n$ the map $S^n \to P_2(S^n)$ is $(3n-2)$-connected.

If $p$ is an odd prime number, then working $p$-locally we have
$$D_2(S^n) = (S^n \wedge S^n)_{h \mathfrak{S}_2} \simeq_{(p)} \begin{cases}
* & \text{ if $n$ is odd}\\
S^{2n} & \text{ if $n$ is even}.
\end{cases}$$
In particular, if $n$ is odd then the map $j_2(S^n)$ is zero on $p$-local homotopy.  If $n$ is even, the commutative diagram
\begin{equation*}
\xymatrix{
Q(S^{n-1}) \ar[d]^-{\cong} \ar[rr]^-{j_2(S^{n-1})} & & Q(D_2(S^{n-1})) \ar[d]\\
\Omega Q(\Sigma S^{n-1}) \ar[rr]^-{\Omega j_2(\Sigma S^{n-1})}& & \Omega Q(D_2(\Sigma S^{n-1}))
}
\end{equation*}
from \cite[Theorem 1.2 (1)]{KuhnJH} and the fact that $Q(D_2(S^{n-1})) \simeq_{(p)} *$ shows that $j_2(S^n)$ is also zero on $p$-local homotopy. Thus $p$-locally for $i < 3n-2$ there are short exact sequences
\begin{equation}
0 \to \begin{cases}
0 & \text{ if $n$ is odd}\\
\pi_{i+1-2n}^s & \text{ if $n$ is even}\end{cases} \to \pi_i(S^n) \to \pi_{i-n}^s \to 0.
\end{equation}
The right-hand map is stabilisation, and the left-hand map is 
$$\pi_{i+1-2n}^s \xleftarrow{\sim} \pi_{i}(S^{2n-1}) \xto{[\iota_n, \iota_n] \circ -} \pi_{i}(S^n).$$
For $i<3n-2$ every element of $\pi_{i}(S^{2n-1})$ is an $n$-fold suspension by Freudenthal's theorem, so the identity \eqref{eq:WhiteheadCompFormula} shows that this agrees with
$$\pi_{i+1-2n}^s \xleftarrow{\sim}\pi_{i+1-n}(S^{n}) \xto{[\iota_n, -]} \pi_{i}(S^n).$$

\subsection{Destabilisation}

The following can be extracted from the literature (cf.\ \cite[Lemma 5.3 (i)]{madsen-taylor-williams}, who attribute it to \cite{Thomeier}), but for the reader's convenience we include a proof using the quadratic approximation.

\begin{lem}\label{lem:Epi}
The stabilisation map $\pi_{2n+1}(S^n) \to \pi_{n+1}^s$ is surjective for $n \geq 3$ with $n \neq 6$. For $n=6$ it has cokernel $\Z/4$.
\end{lem}

\begin{proof}
Let $\mathbb{RP}^{\infty}_n$ denote the stunted real projective space $\mathbb{RP}^{\infty}/\mathbb{RP}^{n-1}$. Using $(S^n \wedge S^n)_{h\mathfrak{S}_2} \simeq S^n \wedge \mathbb{RP}^\infty_n$, the quadratic approximation gives a diagram
\beq
\xymatrix{
 & & \pi_{n+1}^s(S^n) \ar[d] \\
 \pi_{2n+1}(S^n) \ar[r] \ar[d]^E & \pi_{n+1}^s \ar[r] \ar@{=}[d] & \pi_{n+1}^s( \mathbb{RP}^\infty_n) \ar[d] \ar[r]^-{\partial} & \pi_{2n}(S^n) \ar[d]^E\\
\pi_{2n+2}(S^{n+1}) \ar[r] & \pi_{n+1}^s \ar[r] & \pi_{n+1}^s(\mathbb{RP}^\infty_{n+1}) \ar[r]& \pi_{2n+1}(S^{n+1})
}
\eeq
with rows and column exact. Showing that the stabilisation map $\pi_{2n+1}(S^n) \to \pi_{n+1}^s$ is surjective is equivalent to showing that the map $\partial : \pi_{n+1}^s( \mathbb{RP}^\infty_n) \to \pi_{2n}(S^n)$ is injective, which is what we shall do.

The map
\begin{align*}
\pi_{n+1}^s(\mathbb{RP}^\infty_{n+1}) \cong \begin{cases}
\Z & \text{$n+1$ even}\\
\Z/2 & \text{$n+1$ odd}
\end{cases} &\to \pi_{2n+1}(S^{n+1})
\end{align*}
sends a generator to the Whitehead square, so is injective unless $n+1 = 3$ or $7$.

The composition $\Z/2 \cong \pi_{n+1}^s(S^n) \to \pi_{n+1}^s( \mathbb{RP}^\infty_n) \xto{\partial} \pi_{2n}(S^n)$ sends a generator to $[\iota_n, \eta_n]$, which we have seen in the proof of Lemma \ref{lem:Hopf} is nonzero unless $n=2, 6$ or $n \equiv 3 \mod 4$. Thus for $n \geq 3$ with $n \neq 6$ and $n \not\equiv 3 \mod 4$ the map $\partial : \pi_{n+1}^s( \mathbb{RP}^\infty_n) \to \pi_{2n}(S^n)$ is injective as required.

If $n \equiv 3 \mod 4$ then we will show that $\Z/2 \cong \pi_{n+1}^s(S^n) \to \pi_{n+1}^s( \mathbb{RP}^\infty_n)$ is zero, so that $\partial$ is injective as required. In the module over the Steenrod algebra $H^*(\mathbb{RP}^\infty_n;\F_2)$, considered as as the span of the $x^i$ with $i \geq n$ inside $H^*(\mathbb{RP}^\infty; \F_2) = \F_2[x]$, we have $\mathrm{Sq}^2(x^n) = \binom{n}{2}x^{n+2}$ which is non-trivial if $n \equiv 3 \mod 4$. This detects the fact that in this case the $(n+2)$-cell is attached along $\eta$ times the $n$-cell, so $ \pi_{n+1}^s(S^n) \to \pi_{n+1}(\mathbb{RP}^\infty_n)$ is zero as claimed.

For $n=6$ the corresponding map is $\Z/60 \cong \pi_{13}(S^6) \to \pi_{7}^s \cong \Z/240$ so is not onto: it is injective, with cokernel $\Z/4$.
\end{proof}

\section{Coinvariant calculations}\label{sec:Coinv}

Let us write $EU_g(\Z[\pi], \Lambda_n^{\min}) \leq U_g(\Z[\pi], \Lambda_n^{\min})$ for the subgroup generated by the elementary unitary matrices, following \cite[\S5.3A]{HO}. Unwrapping their definitions it is generated by the elements which in terms of the bases $(a_1, a_2, \ldots, a_g, b_1, b_2, \ldots, b_g)$ are given by, writing $\epsilon = (-1)^n$, the matrices
\begin{equation*}
\begin{bmatrix} 1 & 0 & 0 & 0\\
0 & 1 & 0 & 0\\
0 & r & 1 & 0\\
-\epsilon\bar{r} & 0 & 0 & 1\end{bmatrix}, 
\begin{bmatrix} 1 & 0 & 0 & r\\
0 & 1 & -\epsilon\bar{r} & 0\\
0 & 0 & 1 & 0\\
0 & 0 & 0 & 1\end{bmatrix}, 
\begin{bmatrix} 1 & r & 0 & 0\\
0 & 1 & 0 & 0\\
0 & 0 & 1 & 0\\
0 & 0 & -\bar{r} & 1\end{bmatrix}, 
\begin{bmatrix} 1 & 0 & 0 & 0\\
r & 1 & 0 & 0\\
0 & 0 & 1 & -\bar{r}\\
0 & 0 & 0 & 1\end{bmatrix},
\begin{bmatrix} 1 & l \\
0 & 1\end{bmatrix},
\begin{bmatrix} 1 & 0\\
l & 1\end{bmatrix}
\end{equation*}
for $r \in \Z[\pi]$ and $l \in \Lambda_n^{\min}$, their stabilisations, and their conjugates obtained by permuting the $a$'s and $b$'s simultaneously. It will also be useful to have available the matrix
$$\begin{bmatrix}
 0 & 0 & 0 & -1\\
0 & 0 & \epsilon & 0\\
0 & -\epsilon & 0 & 0\\
1 & 0 & 0 & 0
\end{bmatrix} = \begin{bmatrix} 
1 & 0 & 0 & -1\\
0 & 1 & \epsilon & 0\\
0 & 0 & 1 & 0\\
0 & 0 & 0 & 1
\end{bmatrix} \cdot\begin{bmatrix} 
1 & 0 & 0 & 0\\
0 & 1 & 0 & 0\\
0 & -\epsilon & 1 & 0\\
1 & 0 & 0 & 1
\end{bmatrix} \cdot \begin{bmatrix} 
1 & 0 & 0 & -1\\
0 & 1 & \epsilon & 0\\
0 & 0 & 1 & 0\\
0 & 0 & 0 & 1
\end{bmatrix} .$$

By \cite[5.3.8]{HO} the groups $EU_g(\Z[\pi], \Lambda_n^{\min})$ are perfect as long as $g \geq 3$, so are contained in the kernel of any homomorphism from $U_g(\Z[\pi], \Lambda_n^{\min})$ to an abelian group: in particular, they are contained in $\Omega^{\min}_g$.

\begin{lem}\label{lem:coinvariantsH} 
For $g \geq 2$ we have
$$H_0(EU_g(\Z[\pi], \Lambda_n^{\min}) ; H) = 0.$$
\end{lem}
\begin{proof}
Using the third matrix with $r=1$ we see that in the coinvariants $a_2 \sim a_1 + a_2$, and so $a_1=0$; similarly it gives $b_1 \sim b_1 - b_2$, so $b_2=0$. By permuting we see that all $a_i$ and $b_i$ are zero in the coinvariants.
\end{proof}

Recall from Propositions \ref{prop:pi_khAut} and \ref{cor:pi_khAutWg} that we defined
\begin{align*}
S^+_X &:= \langle x \otimes x \, | \, x \in \pi_n(X_g) \rangle_\Z \subset \pi_n(X_g)^{\otimes_{\Z[\pi]} 2},\\
S^-_X &:= \langle x \otimes y - y \otimes x \, | \, x, y \in \pi_n(X_g) \rangle_\Z \subset \pi_n(X_g)^{\otimes_{\Z[\pi]} 2},\\
S^+_W &:= \langle x \otimes x \, | \, x \in \pi_n(W_{g,1}) \rangle_\Z \subset \pi_n(W_{g,1})^{\otimes 2},\\
S^-_W &:= \langle x \otimes y - y \otimes x \, | \, x, y \in \pi_n(W_{g,1}) \rangle_\Z \subset \pi_n(W_{g,1})^{\otimes 2},
\end{align*}
and there is a natural map $S^\pm_X \to S^\pm_W$ which is a split epimorphism and equivariant for $U_g(\Z[\pi], \Lambda_n^{\min}) \to U_g(\Z, \Lambda_n^{\min})$. Here $S^\pm_W$ is the set of elements of $\pi_n(W_{g,1})^{\otimes 2}$ $\pm$-invariant under swapping the factors, i.e.\ the kernel of $1 \mp T : \pi_n(W_{g,1})^{\otimes 2} \to \pi_n(W_{g,1})^{\otimes 2}$, so writing $M := \pi_n(W_{g,1}) \otimes_{\Z} \Z/2$ there is an exact sequence
$$ \mathrm{Hom}_{\Z/2}(M^{\otimes 2}, \Z/2)\xto{1 \mp T} \mathrm{Hom}_{\Z/2}(M^{\otimes 2}, \Z/2)\to \mathrm{Hom}_{\Z/2}(S^\pm_W \otimes \Z/2, \Z/2) \to 0$$
Following \cite[Lemma 13.5]{Bak}, $\mathrm{Hom}_{\Z}(S^\pm_W, \Z/2) \cong \mathrm{Hom}_{\Z/2}(S^\pm_W \otimes \Z/2, \Z/2)$ is in bijection with the set $Q_{\pm}(M)$ of $\pm$-quadratic forms on $M = \pi_n(W_{g,1}) \otimes \Z/2$. Whatever the parity of $n$, the hyperbolic quadratic form reduced modulo 2 determines an element of $Q_{\pm}(M)$, which is by definition invariant for $U_g(\Z, \Lambda_n^{\min})$, so gives an $U_g(\Z[\pi], \Lambda_n^{\min})$-invariant map
$$\phi^\pm : S^\pm_X \to S^\pm_W \to \Z/2.$$
Unwrapping the proof of \cite[Lemma 13.5]{Bak}, this is given by any sesquilinear form $f : M \otimes M \to \Z/2$ such that $\lambda(x,y) = f(x,y) + f(y,x)$ and $q(x) = f(x,x)$. For example we can take $f$ to be given by $f(a_i, b_j) = \delta_{ij}$ and zero on all other pairs of basis elements.

\begin{lem}\label{lem:coinvariants U_g}
The $U_g(\Z[\pi], \Lambda_n^{\min})$-invariant map $S^\pm_X \subset \pi_n(X_g)^{\otimes_{\Z[\pi]} 2} \xto{\lambda_X} \Z[\pi]$, as well as $\phi^\pm$ in the case of $S^\pm_X$ for $\pm = (-1)^{n+1}$, induces for $g \geq 3$ an isomorphism
$$H_0(EU_g(\Z[\pi], \Lambda_n^{\min}) ; S^\pm_X) \cong \bigoplus_{a > 0} \Z\{t^a \pm (-1)^n t^{-a}\} \oplus \begin{cases}
\Z/2 & \text{ if $\pm = (-1)^{n+1}$}\\
\Z\{2 t^0\} & \text{ if $\pm = (-1)^n$}.
\end{cases}$$
\end{lem}
\begin{proof}
Let us write $x \sim y$ to mean that $x$ and $y$ become equal after taking coinvariants. As a $\Z$-module $S^+_X$ is freely generated by
$$a_i \otimes a_i, \quad (t^f + t^{-f})a_i \otimes a_i, \quad b_i \otimes b_i, \quad (t^f + t^{-f})b_i \otimes b_i, \quad t^e x_i \otimes x_j + t^{-e}x_j \otimes x_i, $$
with $i < j$, $f \in \mathbb{N}_{>0}$, and $e \in \mathbb{Z}$. 

\begin{enumerate}[(i$^+$)]
\item Using the fourth matrix with $r=1$ on the 1st and 3rd hyperbolic pairs, we have
\begin{align*}
t^e a_1 \otimes b_2 + t^{-e} b_2 \otimes a_1 &\sim t^e (a_1+ a_3) \otimes b_2 + t^{-e} b_2 \otimes (a_1+ a_3)
\end{align*}
and so $t^{e} a_3 \otimes b_2 + t^{-e} b_2 \otimes a_3 \sim 0$. By permuting we find that $t^e a_i \otimes b_j + t^{-e} b_j \otimes a_i \sim 0$ for all $i \neq j$.

\item Using the first matrix with $r=(-1)^{n+1}$ we have $(t^f + t^{-f})a_1 \otimes a_1 \sim (t^f + t^{-f})(a_1+b_2) \otimes (a_1 + b_2)$ and so $(t^f + t^{-f})(a_1 \otimes b_2 + b_2 \otimes a_1) + (t^f + t^{-f})b_2 \otimes b_2 \sim 0$. Combining with the above we have $(t^f + t^{-f})b_2 \otimes b_2 \sim 0$. The analogous argument shows $(t^f + t^{-f}) a_1 \otimes a_1 \sim 0$, and by permuting we get the same for all other indices.

The analogous argument applied to $a_1 \otimes a_1$ shows that $b_1 \otimes b_1 \sim 0$, and so on.

\item Using the fourth matrix with $r=t^e$ we have that $a_1 \otimes a_1$ is equivalent to
$$ (a_1 + t^e a_2) \otimes (a_1 + t^e a_2) = a_1 \otimes a_1 + t^e a_2 \otimes a_1 + t^{-e} a_1 \otimes a_2 + a_2 \otimes a_2$$
so with (ii$^{+}$) it follows that $t^e a_2 \otimes a_1 + t^{-e} a_1 \otimes a_2 \sim 0$. We similarly get $t^e b_2 \otimes b_1 + t^{-e} b_1 \otimes b_2 \sim 0$, and by permuting get the same for all other indices.

\item The remaining basis elements are $t^e a_i \otimes b_i + t^{-e} b_i \otimes a_i$. Applying the seventh matrix to $t^e a_1 \otimes b_1 + t^{-e} b_1 \otimes a_1$ gives $(-1)^n (t^e b_2 \otimes a_2 + t^{-e} a_2 \otimes b_2)$, so permuting indices shows that
$$t^e a_i \otimes b_i + t^{-e} b_i \otimes a_i \sim (-1)^n (t^{-e} a_i \otimes b_i + t^{e} b_i \otimes a_i).$$
This lets us rewrite any $t^e a_1 \otimes b_1 + t^{-e} b_1 \otimes a_1$ to have $e \geq 0$, and also shows that $(1-(-1)^n)(a_1 \otimes b_1 + b_1 \otimes a_1) \sim 0$. Unwrapping the definition above shows that
$$\phi^+(a_1 \otimes b_1 + b_1 \otimes a_1) \neq 0 \in \Z/2,$$
so when $n$ is odd the element $a_1 \otimes b_1 + b_1 \otimes a_1$ has order precisely 2 in the coinvariants. The elements $t^e a_i \otimes b_i + t^{-e} b_i \otimes a_i$ with $e \geq 0$ (with $e > 0$ if $n$ is odd) are sent by $\lambda_X$ to the linearly independent elements $t^e + (-1)^n t^{-e} \in \Z[\pi]$, which finishes the argument for $S^+_X$.
\end{enumerate}

As a $\Z$-module $S^-_X$ is freely generated by
$$(t^f-t^{-f})a_i \otimes a_i,  \quad (t^f-t^{-f})b_i \otimes b_i, \quad t^e x_i \otimes x_j - t^{-e}x_j \otimes x_i$$
with $i \neq j$, $f \in \mathbb{N}_{>0}$, and $e \in \Z$.

\begin{enumerate}[(i$^-$)]
\item Just as in (i$^+$) we find that $t^e a_i \otimes b_j - t^{-e}b_j \otimes a_i \sim 0$ for all $i \neq j$.

\item Just as in (ii$^+$) we find that  $(t^f-t^{-f}) a_i \otimes a_i \sim 0$ and $(t^f-t^{-f}) b_i \otimes b_i \sim 0$. 

\item Just as in (iii$^+$) we find that $t^e a_i \otimes a_j - t^{-e} a_j \otimes a_i \sim 0$ and $t^e b_i \otimes b_j - t^{-e} b_j \otimes b_i \sim 0$ for all $i \neq j.$

\item The remaining basis elements are $t^e a_i \otimes b_i - t^{-e} b_i \otimes a_i$. As in (iv$^-$), applying the seventh matrix to $t^e a_1 \otimes b_1 - t^{-e} b_1 \otimes a_1$ gives $-(-1)^n (t^{-e} a_2 \otimes b_2 - t^e b_2 \otimes a_2)$, so permuting indices shows that
$$t^e a_1 \otimes b_1 - t^{-e} b_1 \otimes a_1 \sim -(-1)^n (t^{-e} a_1 \otimes b_1 - t^e b_1 \otimes a_1).$$
This lets us rewrite any $t^e a_i \otimes b_i - t^{-e} b_i \otimes a_i$ to have $e \geq 0$, and shows that $(1+(-1)^n)(a_1 \otimes b_1 -  b_1 \otimes a_1) \sim 0$. As
$$\phi^-(a_1 \otimes b_1 -  b_1 \otimes a_1) \neq 0 \in \Z/2$$
it follows that when $n$ is even the element $a_1 \otimes b_1 -  b_1 \otimes a_1$ has order precisely 2 in the coinvariants. The elements $t^e a_i \otimes b_i - t^{-e} b_i \otimes a_i$ with $e \geq 0$ (with $e > 0$ if $n$ is even) are sent by $\lambda_X$ to the linearly independent elements $t^e - (-1)^n t^{-e} \in \Z[\pi]$, which finishes the argument for $S^-_X$.\qedhere
\end{enumerate}
\end{proof}

\section{Relation to classical calculations}\label{sec:classical}

We wish to explain how our calculation relates to the classical approach via surgery theory, pseudoisotopy theory, and the algebraic $K$-theory of spaces, following the programme of Weiss--Williams \cite{WW1, WW2, WW3, WWSurvey}. As full details for that programme are not yet available, and this comparison is not our main goal, the following should be considered as provisional. It also assumes the reader is familiar with those papers: we take \cite{WWSurvey} as our main reference. We also suppose for simplicity that $p$ is odd: then one can probably replace the work of Weiss and Williams with that of Burghelea--Lashof \cite{BL}. 

The homotopy classes produced by Theorem \ref{thm:B} have the following description. Using Example \ref{ex:EmbeddingsOddprimes} and Proposition \ref{prop:EMCG up to concordance} we may form the composition
\beq
\xymatrix{
\Z/p \otimes S^{(-1)^{n+1}}_X \ar@{=}[d] & \pi_{2p-3}(\Emb^{\cong}_{\partial/2}(X_g))_{(p)} \ar[d]^-{\cong} \ar[r]^-{\partial}& \pi_{2p-3}(B\Diff_\partial(S^1 \times D^{2n-1}))_{(p)}\\
 \pi_{2p-3}(\blEmb^{\mathrm{sfr}, \cong}_{\partial/2}(X_g ; \ell_{1/2\partial X_g}))_{(p)} \ar[r]^-{\text{forget}} & \pi_{2p-3}(\blEmb^{\cong}_{\partial/2}(X_g))_{(p)}
}
\eeq
and our argument shows that this factors over the $\check{\Xi}_g^{\sfr,\ell}$-coinvariants of $\Z/p \otimes S^{(-1)^{n+1}}_X$, giving an injection; it remains injective when mapped further to $B\Top_\partial(S^1 \times D^{2n-1})$. We wish to explain how this is related to the calculation of $\pi_*(B\Top_\partial(S^1 \times D^{2n-1}))_{(p)}$ via pseudoisotopy theory.

The Weiss fibre sequences for homeomorphisms and block homeomorphisms give the diagram
\beq
\xymatrix{
\Top_{\partial}(S^1\times D^{2n-1})\ar[r] \ar[d] & \Top_{\partial}(X_g) \ar[r] \ar[d]& \Emb^{\mathrm{TOP}, \cong}_{\partial/2}(X_g) \ar[d] \ar[r]& B\Top_{\partial}(S^1\times D^{2n-1})\\
{*\simeq \blTop_{\partial}(S^1\times D^{2n-1})}\ar[r] \ar[d] & \blTop_{\partial}(X_g) \ar[r]^-{\simeq} \ar[d] & \blEmb^{\mathrm{TOP}, \cong}_{\partial/2}(X_g) \ar[d]\\
\frac{\blTop_{\partial}(S^1\times D^{2n-1})}{\Top_{\partial}(S^1\times D^{2n-1})}\ar[r] \ar[d]^-{\simeq}& \frac{\blTop_{\partial}(X_g)}{\Top_{\partial}(X_g)} \ar[r]& F\\
B\Top_{\partial}(S^1\times D^{2n-1}) &
}
\eeq
where rows and columns form homotopy fibre sequences. The analogue of the proof of Proposition \ref{prop:EMCG up to concordance} for homeomorphisms shows that $F$ is $(n-2)$-connected. Therefore in this range of degrees the diagram identifies the map
$$\partial : \pi_*(\Emb^{\mathrm{TOP}, \cong}_{\partial/2}(X_g)) \to \pi_*(B\Top_{\partial}(S^1\times D^{2n-1}))$$
with the composition
$$\pi_* \left(\Emb^{\mathrm{TOP}, \cong}_{\partial/2}(X_g)\right) \to \pi_*\left(\blEmb^{\mathrm{TOP}, \cong}_{\partial/2}(X_g)\right)\xleftarrow{\cong} \pi_*\left(\blTop_{\partial}(X_g)\right) \to \pi_*\left(\tfrac{\blTop_{\partial}(X_g)}{\Top_{\partial}(X_g)}\right).$$

The theory of Weiss--Williams provides a map
\beq
\xymatrix{
\frac{\blTop_{\partial}(S^1\times D^{2n-1})}{\Top_{\partial}(S^1\times D^{2n-1})} \ar[r] & \Omega^\infty(\mathrm{H}^s(S^1 \times D^{2n-1})_{hC_2})
}
\eeq
and similarly for $\frac{\blTop_{\partial}(X_g)}{\Top_{\partial}(X_g)}$. Here $\mathrm{H}^s(M)$ denotes the (simple) stable topological $h$-cobordism spectrum of $M$: the parameterised stable $h$-cobordism theorem \cite{WJR} identifies it with (the connective cover of) $\Sigma^{-1} \mathrm{Wh}^{\mathrm{TOP}}(M)$, the desuspension of the topological Whitehead spectrum of $M$; $\mathrm{H}^s(M)$ is equipped with a certain involution, and $(-)_{hC_2}$ denotes the homotopy orbit spectrum with respect to this involution. The main theorem concerning these maps is that they are equivalences in the pseudoisotopy stable range, but we shall not use this: the existence of these maps suffices.

It follows from \cite[Section 1.5]{WWSurvey} that there is a map $w$ making the square
\beq
\xymatrix{
\blTop_{\partial}(X_g) \ar[r] \ar[d] & \frac{\blTop_{\partial}(X_g)}{\Top_{\partial}(X_g)} \ar[r] & \Omega^\infty(\mathrm{H}^s(X_g)_{hC_2}) \ar[d]^-{\text{transfer}}\\
\hAut_\partial(X_g) \ar[r] & \hAut(X_g) \ar[r]^-{w}& \Omega^\infty \mathrm{H}^s(X_g)
} 
\eeq
commute up to homotopy. The map $w$ factors through a map $w_A : \hAut(X_g) \to \Omega^\infty \Sigma^{-1} \mathrm{A}(X_g)$, which one should think of as the crossed homomorphism (up to homotopy) corresponding to the splitting of the fibration sequence
$$A(X_g) \to A_{\pi}(E) \to B\hAut(X_g)$$
given by the family $A$-theory characteristic $\chi: B\hAut(X_g) \to A_{\pi}(E)$ associated to the universal $X_g$-fibration $X_g \to E \xto{\pi} B\hAut(X_g)$ (we have written $A(-) := \Omega^\infty \mathrm{A}(-)$).

Consider instead the universal $(X_g, \partial X_g)$-fibration $p : (E, \partial E) \to B\hAut_\partial(X_g)$. As the inclusion $B\hAut_\partial(X_g) \times S^1 \times S^{2n-2} = \partial E \to E$ is $n$-connected, there is a fibrewise map
$$\tau : E \to B\hAut_\partial(X_g) \times S^1$$
and hence a map $A(\tau) : A_p(E) \to B\hAut_\partial(X_g) \times A(S^1)$ over $B\hAut_\partial(X_g)$. This identifies the composition
$$\hAut_\partial(X_g) \to \hAut(X_g) \xto{w} \Omega^\infty \mathrm{H}^s(X_g) \to \Omega^\infty \mathrm{H}^s(S^1)$$
as the map obtained by looping the map $B\hAut_\partial(X_g) \to B\hAut_{S^1}(X_g) \xto{\iota} A(S^1)$ given by considering
$$B\hAut_\partial(X_g) \times S^1 \times \{*\} \subset \partial E \to E \xto{\tau} B\hAut_\partial(X_g) \times S^1$$
as a family of retractive spaces over $S^1$.

We may form the following commutative diagram:
\beq
\xymatrix{
\Z/p \otimes S^{(-1)^{n+1}}_X \ar@{=}[d]  \ar[rr]^-{U_g(\Z[\pi], \Lambda_n^{\min})\text{-coinvariants}}& & \bigoplus_{a > 0} \Z/p\{t^a - t^{-a}\} \ar[dd]\\
 \pi_{2p-3}(\blEmb^{\mathrm{sfr}, \cong}_{\partial/2}(X_g ; \ell_{1/2\partial X_g}))_{(p)} \ar[r]^-{\text{forget}} \ar[rdd]_-{\text{forget}} & \pi_{2p-3}(\blEmb^{\mathrm{TOP}, \cong}_{\partial/2}(X_g))_{(p)}\\
& \pi_{2p-3}(\blTop_{\partial}(X_g))_{(p)} \ar[u]^-{\cong} \ar[d] \ar[r]& \pi_{2p-3}(\frac{\blTop_{\partial}(X_g)}{\Top_{\partial}(X_g)})_{(p)} \ar[d]\\
 & \pi_{2p-3}(\hAut_{\partial}(X_g))_{(p)} \ar[r]^-{w_*} \ar[d] & \pi_{2p-3}(\mathrm{H}^s(X_g))_{(p)} \ar[d]^-{\simeq}\\
 & \pi_{2p-3}(\hAut_{S^1}(X_g))_{(p)} \ar[r]^-{\iota_*} & \pi_{2p-3}(\mathrm{H}^s(S^1))_{(p)}
}
\eeq
It follows from our results that the diagonal map and lower left-hand maps are injective on $U_g(\Z[\pi], \Lambda_n^{\min})$-coinvariants. The $U_g(\Z[\pi], \Lambda_n^{\min})$-action on $\pi_{2p-3}(\hAut_{S^1}(X_g))_{(p)}$  extends to a $\mathrm{GL}_{2g}(\Z[\pi])$-action, and using our description 
$$\pi_{2p-3}(\hAut_{S^1}(X_g))_{(p)} \cong \Hom_{\Z[\pi]}(\pi_n(X_g), \pi_{n}(X_g))\otimes \Z/p$$
we see that these $\mathrm{GL}_{2g}(\Z[\pi])$-coinvariants are $\Z/p[t, t^{-1}]$, and that the induced map on coinvariants is given by the inclusion
$$\bigoplus_{a > 0} \Z/p\{t^a - t^{-a}\} \subset \Z/p[t, t^{-1}].$$

Using Waldhausen's description $\Omega^\infty_0 \mathrm{A}(S^1) \simeq \hocolim_{n,g \to \infty} B\hAut_{S^1}(S^1 \vee \bigvee_{2g} S^n)^+$ and Freudenthal's stability range, it follows using Lemma \ref{lem:PlusConst} that the induced map
$$\Z/p[t, t^{-1}] \cong H_0(\mathrm{GL}_{2g}(\Z[\pi]) ; \pi_{2p-3}(\hAut_{S^1}(X_g))_{(p)}) \to \pi_{2p-3}(\Sigma^{-1} \mathrm{A}(S^1))_{(p)}$$
is an isomorphism, and by consideration of cyclotomic structures that the induced map
$$\Z/p[t, t^{-1}] \cong H_0(\mathrm{GL}_{2g}(\Z[\pi]) ; \pi_{2p-3}(\hAut_{S^1}(X_g))_{(p)}) \to \pi_{2p-3}(\mathrm{H}^s(S^1))_{(p)}$$
is injective with finitely-generated cokernel. It follows that the composition along the rightmost column in the diagram above is injective with finitely-generated cokernel. This explains the sense in which the copy of $\overset{\infty}{\bigoplus}\, \Z/p$ given by Theorem \ref{thm:B} is the same as that coming from pseudoisotopy theory and the algebraic $K$-theory of spaces, cf.\ \cite{GKM-Nil}, though our result holds beyond the range in which pseudoisotopy theory is known to apply.

\bibliographystyle{amsalpha}
\bibliography{S1xDisk}

\newcommand{\etalchar}[1]{$^{#1}$}
\providecommand{\bysame}{\leavevmode\hbox to3em{\hrulefill}\thinspace}
\providecommand{\MR}{\relax\ifhmode\unskip\space\fi MR }
\providecommand{\MRhref}[2]{%
  \href{http://www.ams.org/mathscinet-getitem?mr=#1}{#2}
}
\providecommand{\href}[2]{#2}
\begin{thebibliography}{CDH{\etalchar{+}}20b}

\bibitem[AM99]{arone-mahowald}
G.~Arone and M.~Mahowald, \emph{The {G}oodwillie tower of the identity functor
  and the unstable periodic homotopy of spheres}, Invent. Math. \textbf{135}
  (1999), no.~3, 743--788. \MR{1669268}

\bibitem[Bak81]{Bak}
A.~Bak, \emph{{$K$}-theory of forms}, Annals of Mathematics Studies, vol.~98,
  Princeton University Press, Princeton, N.J.; University of Tokyo Press,
  Tokyo, 1981.

\bibitem[Bau96]{Baues}
H.~J. Baues, \emph{On the group of homotopy equivalences of a manifold}, Trans.
  Amer. Math. Soc. \textbf{348} (1996), no.~12, 4737--4773.

\bibitem[Beh12]{behrens}
M.~Behrens, \emph{The {G}oodwillie tower and the {EHP} sequence}, Mem. Amer.
  Math. Soc. \textbf{218} (2012), no.~1026, xii+90. \MR{2976788}

\bibitem[Ber82]{Berrick}
A.~J. Berrick, \emph{The plus-construction and fibrations}, Quart. J. Math.
  Oxford Ser. (2) \textbf{33} (1982), no.~130, 149--157.

\bibitem[BG75]{BG}
J.~C. Becker and D.~H. Gottlieb, \emph{The transfer map and fiber bundles},
  Topology \textbf{14} (1975), 1--12.

\bibitem[BG19]{BudneyGabai}
R.~Budney and D.~Gabai, \emph{Knotted 3-balls in {$S^4$}},
  \url{https://arxiv.org/abs/1912.09029}, 2019.

\bibitem[BHS64]{BHS}
H.~Bass, A.~Heller, and R.~G. Swan, \emph{The {W}hitehead group of a polynomial
  extension}, Inst. Hautes \'{E}tudes Sci. Publ. Math. (1964), no.~22, 61--79.

\bibitem[BL82]{BL}
D.~Burghelea and R.~Lashof, \emph{Geometric transfer and the homotopy type of
  the automorphism groups of a manifold}, Trans. Amer. Math. Soc. \textbf{269}
  (1982), no.~1, 1--38.

\bibitem[BLR75]{BLR}
D.~Burghelea, R.~Lashof, and M.~Rothenberg, \emph{Groups of automorphisms of
  manifolds}, Lecture Notes in Mathematics, Vol. 473, Springer-Verlag,
  Berlin-New York, 1975, With an appendix (``The topological category'') by E.
  Pedersen.

\bibitem[BM13]{BerglundMadsen1}
A.~Berglund and I.~Madsen, \emph{Homological stability of diffeomorphism
  groups}, Pure Appl. Math. Q. \textbf{9} (2013), no.~1, 1--48.

\bibitem[Bur77]{burghelea-blockTop}
D.~Burghelea, \emph{The structure of block-automorphisms of {$M\times S^{1}$}},
  Topology \textbf{16} (1977), no.~1, 65--78.

\bibitem[CDH{\etalchar{+}}20a]{CDHHLMNNSII}
B.~Calmès, E.~Dotto, Y.~Harpaz, F.~Hebestreit, M.~Land, K.~Moi, D.~Nardin,
  T.~Nikolaus, and W.~Steimle, \emph{Hermitian {K}-theory for stable
  $\infty$-categories {II}: {C}obordism categories and additivity},
  \url{https://arxiv.org/abs/2009.07224}, 2020.

\bibitem[CDH{\etalchar{+}}20b]{CDHHLMNNSIII}
\bysame, \emph{Hermitian {K}-theory for stable $\infty$-categories {III}:
  {G}rothendieck-{W}itt groups of rings},
  \url{https://arxiv.org/abs/2009.07225}, 2020.

\bibitem[CS11]{CrowleySixt}
D.~Crowley and J.~Sixt, \emph{Stably diffeomorphic manifolds and
  {$l_{2q+1}(\Bbb Z[\pi])$}}, Forum Math. \textbf{23} (2011), no.~3, 483--538.

\bibitem[FH78]{farrell-hsiang}
F.~T. Farrell and W.~C. Hsiang, \emph{On the rational homotopy groups of the
  diffeomorphism groups of discs, spheres and aspherical manifolds}, Algebraic
  and geometric topology ({S}tanford, {C}alif., 1976), {P}art 1, Proc. Sympos.
  Pure Math., XXXII, Amer. Math. Soc., Providence, R.I., 1978, pp.~325--337.

\bibitem[Fri17]{Friedrich}
N.~Friedrich, \emph{Homological stability of automorphism groups of quadratic
  modules and manifolds}, Doc. Math. \textbf{22} (2017), 1729--1774.

\bibitem[GKM08]{GKM-Nil}
J.~Grunewald, J.R. Klein, and T.~Macko, \emph{Operations on the {A}-theoretic
  nil-terms}, J. Topol. \textbf{1} (2008), no.~2, 317--341.

\bibitem[Goo03]{GoodwillieTaylor}
T.~G. Goodwillie, \emph{Calculus. {III}. {T}aylor series}, Geom. Topol.
  \textbf{7} (2003), 645--711.

\bibitem[GRW14]{GRWActa}
S.~Galatius and O.~Randal-Williams, \emph{Stable moduli spaces of
  high-dimensional manifolds}, Acta Math. \textbf{212} (2014), no.~2, 257--377.

\bibitem[GRW17]{GRWHomStabII}
\bysame, \emph{Homological stability for moduli spaces of high dimensional
  manifolds. {II}}, Ann. of Math. (2) \textbf{186} (2017), no.~1, 127--204.

\bibitem[GRW18]{GRWHomStabI}
\bysame, \emph{Homological stability for moduli spaces of high dimensional
  manifolds. {I}}, J. Amer. Math. Soc. \textbf{31} (2018), no.~1, 215--264.

\bibitem[GRW19]{GRW-handbook}
\bysame, \emph{Moduli spaces of manifolds: a user's guide}, Handbook of
  homotopy theory, Chapman \& Hall/CRC, CRC Press, Boca Raton, FL, 2019,
  pp.~445--487.

\bibitem[Hat78]{HatcherConc}
A.~E. Hatcher, \emph{Concordance spaces, higher simple-homotopy theory, and
  applications}, Algebraic and geometric topology ({S}tanford, {C}alif., 1976),
  {P}art 1, Proc. Sympos. Pure Math., XXXII, Amer. Math. Soc., Providence,
  R.I., 1978, pp.~3--21.

\bibitem[Hes09]{hesselholt-WhS1}
L.~Hesselholt, \emph{On the {W}hitehead spectrum of the circle}, Algebraic
  topology, Abel Symp., vol.~4, Springer, Berlin, 2009, pp.~131--184.

\bibitem[Hig40]{Higman}
G.~Higman, \emph{The units of group-rings}, Proc. London Math. Soc. (2)
  \textbf{46} (1940), 231--248.

\bibitem[Hil55]{Hilton}
P.~J. Hilton, \emph{A note on the {$P$}-homomorphism in homotopy groups of
  spheres}, Proc. Cambridge Philos. Soc. \textbf{51} (1955), 230--233.

\bibitem[HLN21]{HLN}
F.~Hebestreit, M.~Land, and T.~Nikolaus, \emph{On the homotopy type of
  {L}-spectra of the integers}, Journal of Topology \textbf{14} (2021), no.~1,
  183--214.

\bibitem[HO89]{HO}
A.~J. Hahn and O.T. O'Meara, \emph{The classical groups and {$K$}-theory},
  Grundlehren der Mathematischen Wissenschaften [Fundamental Principles of
  Mathematical Sciences], vol. 291, Springer-Verlag, Berlin, 1989, With a
  foreword by J. Dieudonn\'{e}.

\bibitem[HS76]{HsiangSharpe}
W.~C. Hsiang and R.~W. Sharpe, \emph{Parametrized surgery and isotopy}, Pacific
  J. Math. \textbf{67} (1976), no.~2, 401--459.

\bibitem[HS21]{HebestreitSteimle}
F.~Hebestreit and W.~Steimle, \emph{Stable moduli spaces of hermitian forms},
  \url{https://arxiv.org/abs/2103.13911}, 2021.

\bibitem[HW73]{HatcherWagoner}
A.~Hatcher and J.~Wagoner, \emph{Pseudo-isotopies of compact manifolds},
  Soci\'{e}t\'{e} Math\'{e}matique de France, Paris, 1973, With English and
  French prefaces, Ast\'{e}risque, No. 6.

\bibitem[Kos67]{Kosinski}
A.~Kosi\'{n}ski, \emph{On the inertia group of {$\pi $}-manifolds}, Amer. J.
  Math. \textbf{89} (1967), 227--248.

\bibitem[Kra20a]{krannich-concordances}
M.~Krannich, \emph{A homological approach to pseudoisotopy theory. {I}},
  Invent.\ Math.\, to appear. \url{https://arxiv.org/abs/2002.04647}, 2020.

\bibitem[Kra20b]{KrannichMCG}
\bysame, \emph{Mapping class groups of highly connected {$(4k+2)$}-manifolds},
  Selecta Math. (N.S.) \textbf{26} (2020), no.~5, Paper No. 81, 49.

\bibitem[Kre79]{kreckisotopy}
M.~Kreck, \emph{Isotopy classes of diffeomorphisms of {$(k-1)$}-connected
  almost-parallelizable {$2k$}-manifolds}, Algebraic topology, {A}arhus 1978
  ({P}roc. {S}ympos., {U}niv. {A}arhus, {A}arhus, 1978), Lecture Notes in
  Math., vol. 763, Springer, Berlin, 1979, pp.~643--663.

\bibitem[KRW20a]{KR-WAlg}
A.~Kupers and O.~Randal-Williams, \emph{The cohomology of {T}orelli groups is
  algebraic}, Forum of Mathematics, Sigma \textbf{8} (2020), e64.

\bibitem[KRW20b]{KR-WDisc}
\bysame, \emph{On diffeomorphisms of even-dimensional discs},
  \url{https://arxiv.org/abs/2007.13884}, 2020.

\bibitem[KRW21]{KR-WFram}
\bysame, \emph{Framings of {$W_{g,1}$}}, Q. J. Math. \textbf{72} (2021), no.~3,
  1029--1053.

\bibitem[Kuh82]{KuhnJH}
N.~J. Kuhn, \emph{The geometry of the {J}ames-{H}opf maps}, Pacific J. Math.
  \textbf{102} (1982), no.~2, 397--412.

\bibitem[Kup19]{kupers-finiteness}
A.~Kupers, \emph{Some finiteness results for groups of automorphisms of
  manifolds}, Geom. Topol. \textbf{23} (2019), no.~5, 2277--2333.

\bibitem[Mah65]{MahowaldWh}
M.~Mahowald, \emph{Some {W}hitehead products in {$S^{n}$}}, Topology \textbf{4}
  (1965), 17--26.

\bibitem[May80]{MayFibrewise}
J.~P. May, \emph{Fibrewise localization and completion}, Trans. Amer. Math.
  Soc. \textbf{258} (1980), no.~1, 127--146.

\bibitem[MR90]{milgram-ranicki}
R.~J. Milgram and A.~A. Ranicki, \emph{The {$L$}-theory of {L}aurent extensions
  and genus {$0$} function fields}, J. Reine Angew. Math. \textbf{406} (1990),
  121--166.

\bibitem[MTW80]{madsen-taylor-williams}
I.~Madsen, L.R. Taylor, and B.~Williams, \emph{Tangential homotopy
  equivalences}, Comment. Math. Helv. \textbf{55} (1980), no.~3, 445--484.

\bibitem[Nic82]{Nicas}
A.~J. Nicas, \emph{Induction theorems for groups of homotopy manifold
  structures}, Mem. Amer. Math. Soc. \textbf{39} (1982), no.~267, vi+108.

\bibitem[Qui70]{Quinn}
F.~Quinn, \emph{A geometric formulation of surgery}, Topology of {M}anifolds
  ({P}roc. {I}nst., {U}niv. of {G}eorgia, {A}thens, {G}a., 1969), Markham,
  Chicago, Ill., 1970, pp.~500--511.

\bibitem[Qui73]{QuillenFG}
D.~Quillen, \emph{Finite generation of the groups {$K_{i}$} of rings of
  algebraic integers}, Algebraic {$K$}-theory, {I}: {H}igher {$K$}-theories
  ({P}roc. {C}onf., {B}attelle {M}emorial {I}nst., {S}eattle, {W}ash., 1972),
  1973, pp.~179--198. Lecture Notes in Math., Vol. 341.

\bibitem[Ran80]{ranicki-foundations}
A.~Ranicki, \emph{The algebraic theory of surgery. {I}. {F}oundations}, Proc.
  London Math. Soc. (3) \textbf{40} (1980), no.~1, 87--192.

\bibitem[Ran81]{RanickiLES}
\bysame, \emph{Exact sequences in the algebraic theory of surgery},
  Mathematical Notes, vol.~26, Princeton University Press, Princeton, N.J.;
  University of Tokyo Press, Tokyo, 1981.

\bibitem[Ran92]{ranicki-lower}
\bysame, \emph{Lower {$K$}- and {$L$}-theory}, London Mathematical Society
  Lecture Note Series, vol. 178, Cambridge University Press, Cambridge, 1992.

\bibitem[RWW17]{RWW}
O.~Randal-Williams and N.~Wahl, \emph{Homological stability for automorphism
  groups}, Adv. Math. \textbf{318} (2017), 534--626.

\bibitem[Sch09]{Schlichtkrull}
C.~Schlichtkrull, \emph{The cyclotomic trace for symmetric ring spectra}, New
  topological contexts for {G}alois theory and algebraic geometry ({BIRS}
  2008), Geom. Topol. Monogr., vol.~16, Geom. Topol. Publ., Coventry, 2009,
  pp.~545--592.

\bibitem[Sha69]{Shaneson}
J.~L. Shaneson, \emph{Wall's surgery obstruction groups for {$G\times Z$}},
  Ann. of Math. (2) \textbf{90} (1969), 296--334.

\bibitem[Sul77]{sullivan}
D.~Sullivan, \emph{Infinitesimal computations in topology}, Inst. Hautes
  \'{E}tudes Sci. Publ. Math. (1977), no.~47, 269--331 (1978).

\bibitem[Tho66]{Thomeier}
S.~Thomeier, \emph{Einige {E}rgebnisse \"{u}ber {H}omotopiegruppen von
  {S}ph\"{a}ren}, Math. Ann. \textbf{164} (1966), 225--250.

\bibitem[Wal62]{WallInertia}
C.~T.~C. Wall, \emph{The action of {$\Gamma _{2n}$} on {$(n-1)$}-connected
  {$2n$}-manifolds}, Proc. Amer. Math. Soc. \textbf{13} (1962), 943--944.

\bibitem[Wal78]{waldhausen}
F.~Waldhausen, \emph{Algebraic {$K$}-theory of topological spaces. {I}},
  Algebraic and geometric topology ({P}roc. {S}ympos. {P}ure {M}ath.,
  {S}tanford {U}niv., {S}tanford, {C}alif., 1976), {P}art 1, Proc. Sympos. Pure
  Math., XXXII, Amer. Math. Soc., Providence, R.I., 1978, pp.~35--60.

\bibitem[Wal99]{wall-book}
C.~T.~C. Wall, \emph{Surgery on compact manifolds}, second ed., Mathematical
  Surveys and Monographs, vol.~69, American Mathematical Society, Providence,
  RI, 1999, Edited and with a foreword by A. A. Ranicki.

\bibitem[Wat20]{watanabe2020thetagraph}
T.~Watanabe, \emph{Theta-graph and diffeomorphisms of some 4-manifolds},
  \url{https://arxiv.org/abs/2005.09545}, 2020.

\bibitem[Wei21]{weiss-dalian}
M.~S. Weiss, \emph{Rational {P}ontryagin classes of {E}uclidean fiber bundles},
  Geom. Topol. \textbf{25} (2021), no.~7, 3351--3424.

\bibitem[Whi78]{WhiteheadBook}
G.~W. Whitehead, \emph{Elements of homotopy theory}, Graduate Texts in
  Mathematics, vol.~61, Springer-Verlag, New York-Berlin, 1978.

\bibitem[WJR13]{WJR}
F.~Waldhausen, B.~Jahren, and J.~Rognes, \emph{Spaces of {PL} manifolds and
  categories of simple maps}, Annals of Mathematics Studies, vol. 186,
  Princeton University Press, Princeton, NJ, 2013.

\bibitem[WW88]{WW1}
M.~Weiss and B.~Williams, \emph{Automorphisms of manifolds and algebraic
  {$K$}-theory. {I}}, $K$-Theory \textbf{1} (1988), no.~6, 575--626.

\bibitem[WW89]{WW2}
\bysame, \emph{Automorphisms of manifolds and algebraic {$K$}-theory. {II}}, J.
  Pure Appl. Algebra \textbf{62} (1989), no.~1, 47--107.

\bibitem[WW01]{WWSurvey}
\bysame, \emph{Automorphisms of manifolds}, Surveys on surgery theory, {V}ol.
  2, Ann. of Math. Stud., vol. 149, Princeton Univ. Press, Princeton, NJ, 2001,
  pp.~165--220.

\bibitem[WW14]{WW3}
M.~S. Weiss and B.~E. Williams, \emph{Automorphisms of manifolds and algebraic
  {$K$}-theory: {P}art {III}}, Mem. Amer. Math. Soc. \textbf{231} (2014),
  no.~1084, vi+110.

\end{thebibliography}
\end{document}